\documentclass[reqno]{amsart}
\usepackage{tkz-euclide}
\usepackage{hyperref} 
\usepackage{amsmath}  
\usepackage{amssymb} 
\usepackage{mathrsfs}   
\usepackage{accents}
\usepackage{graphicx}   
\usepackage{mathtools}

\usepackage{tabularx,colortbl}

\usepackage{tikz} 
\usepackage{tikz-cd}
\tikzset{tail reversed/.code={\pgfsetarrowsstart{tikzcd to}}}

\usetikzlibrary{calc} 
\usepackage{xcolor}  
\usetikzlibrary{arrows.meta} 
\usepackage{amsthm} 
\usepackage{float}
\usepackage{enumitem}
\usepackage[english]{babel}
\usepackage[font=footnotesize, labelfont=bf]{ caption}
\usepackage{ esint }
\usepackage{ dsfont }
\usepackage{stackengine,scalerel,graphicx}
\stackMath
\newcommand\overarc[1]{\ThisStyle{%
  \setbox0=\hbox{$\SavedStyle#1$}%
  \stackon[.5pt]{\SavedStyle#1}{%
  \rotatebox{90}{$\SavedStyle\scaleto{)}{.95\wd0}$}}}}

\newcommand{\dsum}[2]{\sum_{\begin{smallmatrix}#1\\#2\end{smallmatrix}}}

\definecolor{trueblue}{rgb}{0.0, 0.45, 0.81}

\newcommand{\EEE}{\color{black}}

\newcommand{\dx}{\, {\rm d}x}
\newcommand{\dy}{\, {\rm d}y}

\newcommand{\ds}{\, {\rm d}\sigma}
\newcommand{\e}{\varepsilon}

\newcommand{\eg}{{\it e.g.}, }
\newcommand{\ie}{{\it i.e.}, }

\theoremstyle{plain}
\begingroup
\newtheorem{theorem}{Theorem}[section]
\newtheorem{lemma}[theorem]{Lemma}
\newtheorem{proposition}[theorem]{Proposition}

\endgroup

\newenvironment{step}[2]{\medskip\underline{\textit{Step #1:}} \textit{#2.}\smallskip}{}

\theoremstyle{definition}
\begingroup
\newtheorem{definition}[theorem]{Definition}
\newtheorem{example}[theorem]{Example}
\endgroup
\theoremstyle{remark}
\newtheorem{remark}[theorem]{Remark}
\renewcommand{\tilde}{\widetilde}
\newcommand{\sm}{\setminus}
\newcommand{\cf}{{\it cf.}}
\DeclareMathOperator{\argmin}{argmin}

\renewcommand{\d}{ \mathrm{d}}

\DeclareMathOperator{\curl}{curl}
\DeclareMathOperator{\conv}{conv}
\DeclareMathOperator{\dist}{dist}
\DeclareMathOperator{\supp}{supp}

\numberwithin{equation}{section}
\newcommand{\N}{\mathbb{N}}
\newcommand{\Z}{\mathbb{Z}}
\newcommand{\R}{\mathbb{R}}
\newcommand{\C}{\mathbb{C}}
\newcommand{\NN}{\mathcal{N}}

\renewcommand{\S}{\mathbb{S}}
\renewcommand{\L}{\mathcal{L}}
\newcommand{\LL}{\mathbb{L}}
\renewcommand{\H}{\mathcal{H}}

\newcommand{\W}{\mathbb{W}}
\newcommand{\WW}{\mathcal{W}}
\newcommand{\D}{\mathcal{D}}
\newcommand{\dH}{\, {\rm d}\H^1}
\newcommand{\de}{\partial}
\newcommand{\T}{\mathcal{T}}
\newcommand{\F}{\mathcal{F}}
\newcommand{\m}{\mathbf{m}}
\newcommand{\SF}{\mathcal{SF}}

\newcommand{\x}{{\times}}
\newcommand{\defas}{:=}
\newcommand{\wto}{\rightharpoonup}
\newcommand{\wsto}{\overset{*}{\wto}}

\newcommand{\eqZ}{\overset{\Z}{\equiv}}

\newcommand{\wcont}{\subset\subset}
\newcommand{\mres}{\mathbin{\vrule height 1.6ex depth 0pt width 0.13ex\vrule height 0.13ex depth 0pt width 1.3ex}}
\newcommand{\Q}{\mathcal{Q}}

\newcommand{\fto}{\overset{\mathrm{flat}}{\to}}
\newcommand{\dS}{\mathrm{d}_{\S^1}}
\newcommand{\AD}{\mathcal{AD}}

\newcommand{\loc}{\mathrm{loc}}
\newcommand{\screw}{\mathrm{screw}}

\newcommand{\edge}{\mathrm{edge}}
\newcommand{\p}{{\mathrm{p}\text{-}\mathrm{edge}}}

\newcommand{\h}{{\mathrm{p}\text{-}\mathrm{edge,h}}}
\renewcommand{\v}{{\mathrm{p}\text{-}\mathrm{edge,v}}}
\newcommand{\hor}{\mathrm{hor}}
\newcommand{\even}{\mathrm{even}}
\newcommand{\odd}{\mathrm{odd}}

\newcommand{\thetas}{\theta_\mathrm{s}}

\newcommand{\A}[2]{A_{#1,#2}}
\newcommand{\Asigma}{\A{\frac{\sigma}{2}}{\sigma}}

\newcommand{\Ten}{\mathscr{S}}

\newcommand{\AZ}[1]{\Z_\e^{#1}(A)}

\newcommand{\AZS}[2]{\Z_{2\e,{#1}}^{#2}(A)}

\usepackage[normalem]{ulem}

\title[Stacking faults in the limit of a discrete model for partial edge dislocations]{Stacking faults in the limit of a discrete model for partial edge dislocations}

\author[A. Bach]{Annika Bach}
\address[A. Bach]{Technische Universiteit Eindhoven. Den Dolech 2, 5600 MB, Eindhoven, Netherlands.}
\email{a.bach@tue.nl}

\author[M. Cicalese]{Marco Cicalese}
\address[M. Cicalese]{Technische Universit\"{a}t M\"{u}nchen. Boltzmannstr. 3, 85748 Garching bei M\"{u}nchen, Germany.}
\email{cicalese@ma.tum.de}

\author[A. Garroni]{Adriana Garroni}
\address[A. Garroni]{Sapienza Universit\`{a} di Roma. P.le A. Moro 5 00185 Roma, Italy.}
\email{garroni@mat.uniroma1.it}

\author[G. Orlando]{Gianluca Orlando}
\address[G. Orlando]{Politecnico di Bari. Via E. Orabona 4, 70125 Bari BA, Italy.}
\email{gianluca.orlando@poliba.it}

\keywords{Partial dislocations; Stacking faults; Edge dislocations; Discrete-to-continuum limit; $\Gamma$-convergence.}
\subjclass[2020]{49J45, 70G75, 74C05, 58K45}
 
\begin{document}

\maketitle

\begin{abstract}
  In the limit of vanishing lattice spacing we provide a rigorous variational coarse-graining result for a next-to-nearest neighbor lattice model of a simple crystal. We show that the $\Gamma$-limit of suitable scaled versions of the model leads to an energy describing a continuum mechanical model depending on partial dislocations and stacking faults. Our result highlights the necessary multiscale character of the energies setting the groundwork for more comprehensive models that can better explain and predict the mechanical behavior of materials with complex defect structures.
\end{abstract}

\tableofcontents

%%%%%%%%%%%%%%%%%%%%%%%%%%%%%%%%%%%%%%%%%%
%          Introduction
%%%%%%%%%%%%%%%%%%%%%%%%%%%%%%%%%%%%%%%%%%
\section{Introduction}

Defects in the crystal structure of materials such as dislocations, stacking faults, impurities, or grain boundaries play a fundamental role in determining the mechanical properties of the material under consideration. In this paper we are concerned with the first two examples mentioned above, namely with the presence and interplay of dislocations and stacking faults. 
Broadly speaking, dislocations are irregularities in the atomic arrangement of crystals that allow layers of atoms to slip past one another under stress, a phenomenon crucial for plastic deformation. In the case of so-called edge dislocations these irregularities can be thought of -- in three dimensions -- as line defects corresponding to the boundary of extra half-planes of atoms perturbing locally the crystalline structure, or alternatively as the line separating regions on the slip planes that have undergone different slips. Figure~\ref{fig:edge dislocation} illustrates this situation in a simplified two-dimensional setting by depicting a portion of a cross section of a cubic crystal. 

\begin{figure}[ht]
  \centering
  % \begin{tikzpicture}[scale = 0.6]
  %     \pgfmathsetmacro{\n}{5}

  %      % highlighted area
  %     \draw[fill=black!10, color=black!10] (0.5,0.5) circle (1.3cm);
      
  %     \pgfmathsetmacro{\m}{\n-1}

  %     \draw[color=black!20] (-\n+1,-\n+1) grid (\n,\n);
  %     % reference configuration
  %     \foreach \i in {-\m,...,\n}{
  %         \foreach \j in {-\m,...,\n}{
  %             \draw[color=black!20, fill=black!20] (\i, \j) circle (2pt);
  %         }
  %     }

  %     % deformed configuration
  %     \foreach \i in {1,...,\n}{
  %         \foreach \j in {1,...,\n}{
  %             \pgfmathsetmacro{\s}{atan((\j-0.5)/(\i-0.5))/360}
  %             \draw[color=black, fill=black] (\i+\s, \j) circle (2pt);
  %             \draw[dotted] (\i, \j) -- (\i+\s, \j);
  %         }
  %     }

  %     \foreach \i in {-\m,...,0}{
  %         \foreach \j in {-\m,...,\n}{
  %             \pgfmathsetmacro{\s}{1/2 + atan((\j-0.5)/(\i-0.5))/360}
  %             \draw[color=black, fill=black] (\i+\s, \j) circle (2pt);
  %             \draw[dotted] (\i, \j) -- (\i+\s, \j);
  %         }
  %     }

  %     \foreach \i in {1,...,\n}{
  %         \foreach \j in {-\m,...,0}{
  %             \pgfmathsetmacro{\s}{1 + atan((\j-0.5)/(\i-0.5))/360}
  %             \draw[color=black, fill=black] (\i+\s, \j) circle (2pt);
  %             \draw[dotted] (\i, \j) -- (\i+\s, \j);
  %         }
  %     }
  %      \end{tikzpicture}
  \includegraphics{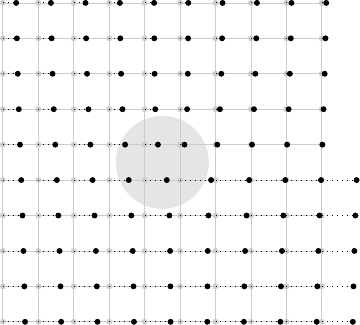}
  \caption{Example of an edge dislocation with Burgers vector $e_1$. The grey dots represent the reference configuration, while the black dots represent the displaced lattice points. We shall see that the elastic energy concentrating outside the singularity, more precisely outside the area highlighted in grey, is detected by the $\Gamma$-limit of the energies $F_\e^\edge$ defined in~\eqref{def:F-edge}.}
  \label{fig:edge dislocation}
\end{figure}

A specific type of dislocations we are interested here are partial dislocations. They typically appear in close-packed crystals such as face-centred cubic (FCC) crystals and are due to the phenomenon that in those crystals for certain slip directions it is necessary to insert two extra half-planes of atoms to restore the crystalline structure far from the defect. It is then energetically more favourable to split up those two half-planes, which in turn leads to the presence of a pair of partial dislocations   
%Partial dislocations are a specific type of dislocation where a full dislocation splits into two or more dislocations with smaller slips. This splitting typically reduces the energy of the dislocation, making partial dislocations more stable under certain conditions (see, \eg~\cite[Chapter 5.3]{HB} or~\cite[Chapter 10.3]{HL}). Partial dislocations often occur in pairs and are 
separated by a region of disrupted atomic order known as a stacking fault (see, \eg~\cite[Chapters 5.1--5.3]{HB} or~\cite[Chapters 10.1--10.3]{HL}).
%Understanding the interplay between partial dislocations and stacking faults is crucial for comprehending the plastic behavior of materials. 
It is worth noting that, unlike dislocations, stacking faults are planar defects in three dimensions. In a nutshell their presence is related to a disruption in the normal stacking sequence of atomic layers in a crystal. For example, in an FCC structure, a defect-free stacking sequence is three-periodic, usually denoted as ``...ABCABC...''. The presence of a stacking fault alters this sequence, producing a stacking like  ``...ABCABABC...'', which deviates from the normal periodic stacking sequence. When such a defect appears in a confined bulk region of the material, partial dislocations must be present near the boundary of this defect region to restore the normal periodic stacking. 

\smallskip
\textbf{Brief literature overview and goal of this paper.} In view of their prominent role in plasticity, dislocations have drawn the attention of both the engineering and the mathematical community. From a mathematical point of view, in the last decades a huge effort has been spent to derive phenomenological/macroscopic plasticity models such as line-tension or strain-gradient models from more fundamental dislocation models using variational methods such as $\Gamma$-convergence. In particular, in a reduced two-dimensional setting both the aforementioned phenomenological models have by now been successfully characterised as $\Gamma$-limits of semi-discrete models within the core-radius approach~\cite{GLP, SZ, MSZ, DGP, Ginster1, Ginster2} as well as of suitable fully discrete models~\cite{P07, ACP11, ADGP14, ADLPP23, ADLPP23a}. It is worth mentioning that in the 3-d setting both semi-discrete and discrete models have been recently analysed under additional diluteness condition between defects leading to recent progresses in the derivation of line-tension models~\cite{CGO15, GMS, CGM23, CGO24}. In this paper we will work with a fully discrete two-dimensional model, but in contrast to the results mentioned above, here we aim at characterising the interplay between partial dislocations and stacking faults in a variational coarse-graining procedure.
%Here the viewpoint is energetic and the results have been obtained in the framework of $\Gamma$-convergence. The latter technique allows to derive effective energies from energy functionals which depend on a vanishing parameters $\e$ by averaging their effect out in the limiting process. For instance, microscopic lattice models of dislocations can be described by energies $E_\e$ depending on the lattice spacing $\e$. A macroscopic theory can then be obtained by taking the $\Gamma$-limit of proper scalings of $E_\e$ as $\e\to 0$ in a sort of rigorous variational coarse-graining procedure.
In fact, while from a mechanical point of view it is well-understood how stacking faults are linked to partial dislocations, a rigorous variational derivation of a continuum energy model from a discrete molecular mechanical model, which would make this link evident, is lacking. The main goal of this paper is to move the first steps in this direction. 
With the idea of singling out the main mathematical difficulties of such a derivation, we introduce and analyze a simple toy model in the context of edge dislocations which can be seen as a minimal two-dimensional discrete model for which one expects in the limit as the lattice spacing vanishes, an effective continuum energy that detects the (possible) presence of partial dislocations and stacking faults. 

\smallskip
\textbf{Derivation and setup of the model.}
With the future aim of explaining the interplay between partial dislocations and stacking faults in a three-dimensional FCC/HCP setting, we propose a two-dimensional model that contains some of its key features. Indeed a two-dimensional model can be obtained constraining deformations to be uniform along specific lattice directions and identifying kinematic variables on a suitable two-dimensional cross-section of the reference lattice. We simplify the latter setting and consider a model energy defined on a portion of $\e\Z^2$ with lattice spacing $\e>0$. As in \cite{ZSWT} we further simplify the model assuming the kinematic constraint that only horizontal deformations (say in direction $e_1$) are allowed. As a consequence, our order parameters will be scalar displacements $u\colon\e\Z^2\to\R$ and our energy will allow only for slips in horizontal direction. In this setting edge dislocations can be represented as point defects due to horizontal slips (\cf~Figure \ref{fig:edge dislocation}), while stacking faults correspond to line defects. A typical phenomenological assumption is that the energetic contribution to those stacking faults does not stem from nearest-neighbour interactions, but from interactions between neighbours at a larger distance (see~\cite[Section 1.3]{HB} or~\cite[Chapter 10.3]{HL}). In the following we wish to introduce a minimal energy model capturing this behaviour. Specifically, we will consider energies accounting for interactions only between nearest and next-to-nearest neighboring atoms, where nearest neighbors promote the existence of multiple ground state configurations, such as the two configurations illustrated in Figure~\ref{fig:ground states}, while next-to-nearest neighbour interactions contribute to the stacking fault energy illustrated in Figure~\ref{fig:partial dislocation}. 

\begin{figure}[ht]
  \begin{center}
  % \begin{tikzpicture}[scale=0.6]
  %     \begin{scope}

  %         \pgfmathsetmacro{\l}{0.6}
  %         \draw (0,3*\l) node[anchor=south] {$\varepsilon \mathbb{Z}^2$};
      
  %         \foreach \i in {-2,-1,...,2}{
  %             \foreach \j in {-2,-1,...,2}{
  %                 \draw[fill=black] (\i*\l, \j*\l) circle (2pt);
  %             }
  %         }
  %          \draw[|-|] (-3*\l,1*\l) -- (-3*\l,2*\l);
  %          \draw (-3*\l, 1.5*\l) node[anchor=east] {$\varepsilon$};
  %     \end{scope}

  %     \draw[->] (2,0) to[out=0, in=180] (6,3);
  %     \draw (5, 0) node {$u$};
  %     \draw[->] (2,0) to[out=0, in=180] (6,-3);
      
  %     \begin{scope}[xshift=10cm, yshift=3cm]
  %          \foreach \i in {-2,-1,...,2}{
  %             \foreach \j in {-2,-1,...,2}{
  %                 \draw[fill=black] (\i, \j) circle (2pt);
  %             }
  %         }
  %         \draw[|-|] (-2.7,1) -- (-2.7,2);
  %         \draw (-2.7, 1.5) node[anchor=east] {$1$};
  %     \end{scope}

  %     \begin{scope}[xshift=10cm, yshift=-3cm]
  %         \foreach \i in {-2,-1,...,2}{
  %             \foreach \j in {-2,-1,...,2}{
  %                 \pgfmathsetmacro{\s}{0.5*(1+(-1)^(\j+1))/2}
  %                 \draw[fill=black] (\i + \s, \j) circle (2pt);
  %             }
  %         }
  %         \draw[|-|] (-2.7,1) -- (-2.7,2);
  %         \draw (-2.7, 1.5) node[anchor=east] {$1$};
  %     \end{scope}
  % \end{tikzpicture}
  \includegraphics{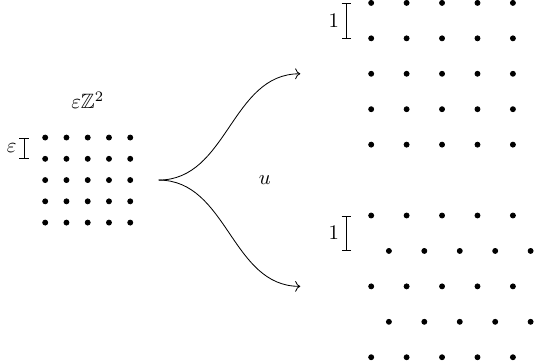}
  \end{center}
  \caption{Local aspect of ground states of the energy $\F_\e^\p$ studied in this paper introduced in~\eqref{def:F-part}. On the left: the reference lattice $\varepsilon \mathbb{Z}^2$. On the top-right: Image of an admissible displacement $u \in \mathcal{AD}_\varepsilon$ where the relative horizontal slips on vertical next-to-nearest neighbours are integers. On the bottom-right: Image of an admissible displacement $u \in \mathcal{AD}_\varepsilon$ where the relative horizontal slips on vertical next-to-nearest neighbours are half-integers.}
  \label{fig:ground states}
\end{figure}

\begin{figure}[ht]
  \centering
  \includegraphics{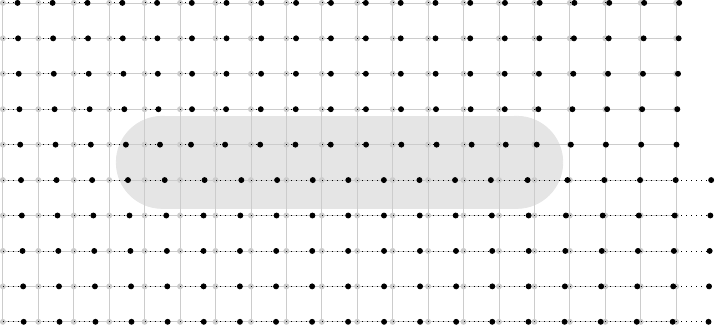}
  \caption{Example of two partial edge dislocations, both with Burgers vector $-\frac{1}{2}e_1$. A stacking fault connecting the two singularities is highlighted in the grey area. There, the displaced lattice is close to a ground state different from the one outside the grey area. The energy of these partial dislocations and stacking faults is detected by the $\Gamma$-limit of the energies $F_\e^\p$ introduced in~\eqref{def:F-part}.}
  \label{fig:partial dislocation}
\end{figure}

Before introducing more precisely the energies considered in this paper, it is convenient to have a closer look at Figures~\ref{fig:ground states} and~\ref{fig:partial dislocation} to motivate the definition of our nearest neighbors and next-to-nearest neighbors interaction energies. A convenient way to describe the structures in Figure~\ref{fig:ground states} is by thinking that they are obtained by the stacking of horizontal layers of atoms on top of each other in two possible ways, say A and B. In configuration A, atoms are displaced by an integer multiple of the horizontal lattice vector with respect to the previous layer, while this displacement becomes half integer in configuration B. 
Vertical next-to-nearest neighbors interactions instead favour integer displacement between next-to-nearest neighboring layers and are assumed to be weaker than nearest-neighbor interactions, leading to an energy penalizing non-homogeneous stacking and being proportional to the length of the defect region. As a result, the total energy penalizes structures like ``...AABAA...'' as the one shown in Figure \ref{fig:partial dislocation}. 
%The energy penalization is proportional to the formation of two partial dislocations and one stacking fault in between. 

%As it is well known, the energy needed to produce a (partial) edge dislocation on the lattice $\e\Z^2$ diverges as $|\log\e|$ in the continuum limit $\e\to 0$ and it dominates the energy necessary to produce a stacking fault. As a consequence, the stacking fault energy can be detected via a renormalization procedure which consists in first removing the logarithmic contribution and eventually passing to the continuum limit as $\e\to 0$. Such a renormalization procedure, which has been first introduced in the context of Ginzburg-Landau energies in \cite{BBH, SS}, has already been successfully exploited in the analysis of discrete to continuum energies in \cite{BC22, BCDP18, COR22ARMA, COR22CPAM}.

Motivated by the above considerations we introduce the following lattice model. 
%which is partly inspired by \cite{ZSWT} and can be seen as a minimal two-dimensional energy model for which one expects in the limit as the lattice spacing vanishes, an effective continuum energy to detect the (possible) presence of partial dislocations and stacking faults. 
Given $\Omega\subset \R^2$ open and bounded, we work on the portion $\e\Z^2\cap\Omega$ of the square lattice $\e\Z^2$ that we consider as a reference configuration of a single crystal that can be deformed via horizontal (scaled) displacements $u \colon \e \Z^2 \to \R$. The energy stored by one such displacement is 
\begin{equation}\label{eqintro:F-part}
  \begin{split}
    F_\e^\p(u, \Omega) & = 2 \pi^2 \underbrace{ \sum_{i} \big| u(i+\e e_1) - u(i) \big|^2}_{\text{horizontal n.n.}} + 2 \pi^2 \underbrace{  \sum_{i} \dist^2 \Big(u(i+\e e_2) - u(i); \frac{1}{2} \Z \Big) }_{\text{vertical n.n.}} \\
    & \quad + 2 \pi^2 \underbrace{ \sum_{i} \frac{\alpha}{\pi^2}  \, \e     \dist^2 \big(u(i+ 2\e e_2) - u(i); \Z \big) }_{\text{vertical n.n.n.}}\,,
  \end{split}
	\end{equation}
%For a more precise definition we refer to~\eqref{def:F-part}. 
where $\alpha > 0$ is a fixed constant. 
%and the energy accounts for horizontal nearest neighbors and vertical next-to-nearest neighbors interactions which depend on the horizontal and vertical discrete displacement, respectively. 

In~\eqref{eqintro:F-part} the horizontal interactions are purely elastic, hence favouring small deformations. This corresponds to the constraint that only horizontal slips are allowed. The vertical nearest-neighbors interactions and next-to-nearest neighbors interactions are defined via $\frac{1}{2}$- and $1$-periodic potentials, respectively. As a consequence, nearest neighbors interactions are invariant under variations of the deformation by half-integer shifts, while next-to-nearest neighbors interactions are invariant under integer shifts. In particular, the nearest-neighbors interactions allow for multiple ground states as the ones depicted in Figure~\ref{fig:ground states}, while the second neighbors will penalize non homogeneous stacking between next-to-nearest neighboring layers. The weight $\e>0$ multiplying the latter interactions makes this energy contribution proportional to the length of the stacking fault and thus in particular of order one. It is well known that the energy needed to produce an edge dislocation on the lattice $\e\Z^2$ diverges as $|\log\e|$ in the continuum limit $\e\to 0$. The energy contribution for a partial edge dislocation is expected to display the same logarithmic divergence, thus dominating the energy necessary to produce a stacking fault. To detect the stacking fault energy it is then convenient to apply a renormalization procedure which consists in first removing the logarithmic contribution and eventually passing to the continuum limit as $\e\to 0$. Such a procedure, which has been first introduced in the context of Ginzburg-Landau energies in \cite{BBH, SS}, has already been successfully exploited in the analysis of discrete to continuum energies in \cite{BC22, BCDP18, COR22ARMA, COR22CPAM}.

\smallskip
\textbf{Main result of this paper.} Our main result characterises the $\Gamma$-limit as $\e \to 0$ of~\eqref{eqintro:F-part} after removing the above mentioned logarithmic contribution of a fixed number of (partial) dislocations. To state the result, it is convenient to redefine the energies in terms of dislocation measures. To this end, to any deformation $u$ we associate a measure $\mu_u$. 
The latter is a sum of dirac deltas supported on the centers of lattice cells, and for every such cell it measures the circulation of the discrete gradient of $u$ over its boundary (see~\eqref{def:circulation} and~\eqref{def:vorticity-u}). The family of all measures consisting of finitely many dirac masses with integer weights and supported in $\Omega$ is denoted by $X(\Omega)$ (see~\eqref{def:X} for a precise definition). We then define $\F_\e^\p$ on $X(\Omega)$ by setting  
\begin{equation} \label{eqintro:Fpmu}
  \F_\e^\p(\mu, \Omega) = \min \{ F_\e^\p(u,\Omega) : \ u \colon \e \Z^2 \to \R \, , \ \mu_{2 u} \mres \Omega = \mu \}\;\text{ for any }\mu\in X(\Omega) \, , 
\end{equation}
that is, to any $\mu\in X(\Omega)$ we associate an energy by optimizing the $F_\e^\p$ over all admissible displacements $u$ such that the doubled displacement $2u$ has dislocation measure $\mu$ in $\Omega$. 
%Since by its very definition $\mu_{2u}=2\mu_u$, 
The use of the double-displacement $2u$ (instead that on $u$) in the constraint in the optimization is justified by the following rationale. Stacking faults (detected by the variable $u$) are canceled the double-displacement $2u$, which restores the integer periodic stacking. As a consequence, the double-displacement $2u$ can be exploited to identify defects in the lattice different from the stacking faults — the dislocations. With this artifice, it is possible to compare the energies in~\eqref{eqintro:Fpmu} to other discrete energies which instead describe phenomena related to the formation and interaction of integer-degree topological singularities like screw dislocations~\cite{P07, ADGP14} or spin vortices~\cite{AC09, ACP11}. 
%Around each limiting singularity these energies store a contribution of order $|\log\e|$. One is then interested in 
In the spirit of the above mentioned renormalisation procedure we analyse the asymptotic behaviour of the energies after removing the logarithmic contribution of a fixed number of $M$ singularities. Specifically, we show that the energies $(\F_\e^\p(\mu,\Omega)-\frac{M\pi}{4}|\log\e|)$ converge to an energy of the form
	\begin{equation}\label{eqintro:gamma-limit}
	\F^\p(\mu,\Omega)= \frac{M}{4}\gamma + \frac{1}{4}\W(\mu,\Omega)+ \alpha \LL(\mu,\Omega)\,,
	\end{equation}
defined on $X_M(\Omega)=\big\{\mu=\sum_{h=1}^{M}d_h\delta_{x_h}\in X(U)\; \text{ with }\; d_h\in\{-1,1\}\big\}$ in the sense specified below. 
%More precisely, we prove the following theorem.
We will comment on the different terms in the definition of $\F^\p$ after giving the precise statement of the result.
\begin{theorem}\label{thm:main2}
  Let $\F_\e^\p$ be as in~\eqref{eqintro:Fpmu} and $\F^\p$ be as in~\eqref{eqintro:gamma-limit}. Then the following results hold true.
  \begin{enumerate}[label=(\roman*)]
  \item (Compactness) Let $(\mu_\e)_\e \subset X(\Omega)$ be a sequence satisfying 
  \begin{align*}
  \sup_\e \big(\F_\e^\p(\mu_\e,\Omega)-\frac{M}{4}\pi|\log\e|\big)\leq C\,. 
  \end{align*}
Then, up to subsequence $\mu_\e\fto\mu$ for some $\mu=\sum_{h=1}^Nd_h\delta_{x_h}\in X(\Omega)$ with $|\mu|(\Omega)\leq M$.
  %with $N\in\N$, $d_h\in\Z\sm\{0\}$, $x_h\in\Omega$ pairwise different, and $\sum_{h=1}^N|d_h|\leq M$. 
  Moreover, if $|\mu|(\Omega)=M$, then $N=M$ and $ \mu \in X_M(\Omega)$.
  \item (Lower bound) Let $(\mu_\e)_\e \subset X(\Omega)$ be such that $\mu_\e\fto\mu$ for some $\mu\in X_M(\Omega)$. Then
    \begin{equation}\label{lb:energy-mu} 
    \liminf_{\e\to 0}\Big( \F_\e^\p(\mu_\e,\Omega)-\frac{M\pi}{4} |\log\e|\Big)\geq \F^\p(\mu,\Omega)\,. 
    \end{equation}
  \item (Upper bound) For any $\mu\in X_M(\Omega)$ there exists $(\mu_\e)_\e \subset X(\Omega)$ with $\mu_\e\fto\mu$ and
    \begin{equation}\label{ub:energy-mu}
    \limsup_{\e\to 0}\Big( \F_\e^\p(\mu_\e,\Omega)-  \frac{M\pi}{4} |\log\e|\Big) \leq \F^\p(\mu,\Omega)\,. 
    \end{equation}
  \end{enumerate}
  \end{theorem}
The limit energy $\F^\p(\mu,\Omega)$ features three terms referred to as the \emph{core energy}, the \emph{renormalized energy} and the \emph{stacking fault energy}, respectively. The core energy is proportional to a fixed quantity $\gamma$ introduced in~\eqref{def:gamma} and is concentrated around each point in the support of $\mu$. Those points in turn represent the location of the partial dislocations. The renormalised energy $\W(\mu,\Omega)$ corresponds to the Coulombian interaction between points in the support of $\mu$, see~\eqref{def:ren-energy-Phi-mu} for the definition. Consequently, it can be interpreted as the effective interaction energy between the limiting partial dislocations. The last term $\LL(\mu,\Omega)$ accounts for the line-tension energy due to stacking faults which in our model are finitely many horizontal segments connecting partial dislocations between each other or to be boundary of $\Omega$. For our model $\LL(\mu,\Omega)$ is proportional to the total length of those segments and it is interpreted as the energy required to resolve partial dislocation tension in the sense of Definitions~\ref{def:connect singularities} and~\ref{def:resolve}.  

\smallskip
\textbf{Proof strategy and main contributions.} Theorem~\ref{thm:main2} will be a consequence of a more general result which indeed characterises the limit behaviour of $\big(F_\e^\p(u_\e,\Omega)-\frac{M\pi}{4}|\log\e|\big)$, that is, of the energies still depending on the displacement variable. The precise formulation of this result stated in Theorem \ref{thm:main} requires some preparation and is postponed to Section~\ref{sec:proof-main-field}. Here we just mention that it is based on a compactness result for spin fields $w_\e=\exp(2\pi\iota u_\e)$ with $F_\e^\p(u_\e,\Omega)-\frac{M\pi}{4}|\log\e|$ uniformly bounded (see Theorem~\ref{thm:main} (i)). Specifically, we show that such sequences $(w_\e)$ converge up to a subsequence to a function $w\in SBV(\Omega;\S^1)$ with $\H^1(S_w)<+\infty$ whose jump set consists (up to $\H^1$-negligible sets) of finitely many horizontal segments (see Theorem~\ref{thm:main} and Proposition~\ref{prop:countable segments}). Moreover, the jacobian of the complex square $w^2$ coincides with the flat limit of the measures $\pi\mu_{2u_\e}$. This compactness result is obtained by carefully estimating from below our energies with suitable discrete energies whose equi-coerciveness properties are by now well understood. 

Firstly, we observe that $F_\e^\p(u_\e,\Omega)$ is an upper bound for the screw-dislocation energy evaluated in the double displacement $2u_\e$ and consequently for the $XY$-model energy evaluated in the complex square $w_\e^2$. This estimate allows us in particular to deduce (up to subsequences) the flat convergence of the dislocation measures $\mu_{2u_\e}$ to a limiting measure $\mu$ and the weak convergence of $w_\e^2$ to some $v\in W^{1,2}_\loc(\Omega\sm\supp\mu;\S^1)$, satisfying $J(v)=\pi\mu$. The additional compactness in $SBV$ for the original variable $w_\e$ requires a more careful argument based on the characterization of the optimal interplay between nearest and next-to-nearest neighbour interactions in the spirit of \cite{BGS,BC07}. This in turn allows us to estimate our energies from below by suitable anisotropic weak-membrane type energies as those considered in~\cite{Ch99} on a double-spaced lattice and conclude by relying on known compactness properties for those energies. The additional geometric properties of the jump set $S_w$, namely the constraint of being horizontal, are then a consequence of the elastic nearest neighbour interactions. The estimates explained above also lead to the liminf-inequality for the energies $F_\e^\p(u_\e,\Omega)$. In particular, the lower bound we obtain with the above sketched procedure features the same core energy $\gamma$ as in the screw-dislocation model.
%{\BBB More specifically, $\gamma$ is characterised via an asymptotic minimisation formula which involves minimising the screw-dislocation energy on balls over displacements that coincide with a lifting of $\frac{x}{|x|}$ on the boundary of the ball (see~\eqref{def:core-screw}--\eqref{def:gamma}).}
Thus, a crucial step to prove the optimality of this lower bound consists in showing that the core energy $\gamma$ can be equivalently characterised in terms of the energy $F_\e^\p$. 
This equivalence is non trivial, and the main obstruction in obtaining such and equivalent characterisation are the elastic horizontal interactions. The procedure we finally apply to obtain the desired characterisation relies on a recent result about minima of discrete energies with degree constraints on the boundary proved in~\cite{GPS} which allows us to choose suitable competitors for the minimisation problem characterising $\gamma$ (see~\eqref{def:core-screw}--\eqref{def:gamma}). Finally, it is worth mentioning that in proving our main result, in Section \ref{Section:edge} we also prove a similar result for the $\Gamma$-limit of a simple edge-dislocation energy corresponding to only the first two terms in~\eqref{eqintro:F-part}.

Some final comments are in order. To characterise the asymptotic behaviour of our energies in~\eqref{eqintro:F-part} we compare them with discrete systems driven solely by periodic type potentials, for which a rigorous variational coarse-graining has been proved. These discrete systems fall into the general class of those lattice systems for which nonlinear elasticity can be formulated as a Landau theory with an infinite number of equivalent energy wells related to the crystal symmetries (see e.g. \cite{BABCTZS}). For screw dislocation models such a theory becomes scalar and its variational coarse-graining has been the object of \cite{ADGP14, P07} where the emergence and the interactions of the topological singularities has been considered. Thanks to the variational equivalence between screw dislocations and classical XY spin systems established in \cite{ACP11}, one can translate the results proved in \cite{BCDP18} on the generalized XY model to the context of crystal plasticity (see also~\cite{B24}). This equivalence identifies spin field singularities (vortices) with screw dislocations. Consequently the results obtained in \cite{BCDP18} can be interpreted as a first attempt to model partial screw dislocations by an energy depending only on nearest-neighboring interactions. Therefore a comparison between the generalized XY model and the present model deserves some comments. In the generalized XY model the limiting spin field can form point singularities known as half-vortices. These half-vortices (corresponding to partial dislocations in the context of crystal plasticity) are connected either to each other or to the boundary by one-dimensional singularities, known as 'string defects' (analogous to the stacking faults in crystal plasticity), which constitute the discontinuity set of an $SBV$ function. However, the authors of~\cite{BCDP18} cannot precisely characterize the geometry of the discontinuity set (it consists of merely rectifiable curves) which makes it challenging to accurately relate fractional vortices to string defects. In contrast, in our model, the line singularities -- specifically, the stacking faults arising from geometric incompatibility between nearest and next-to-nearest interactions -- result in horizontal segment discontinuities. This rigidity allows us to fully characterize the interactions between these singularities and the partial dislocations. This characterization helps us properly understand stacking faults as necessary configurations to resolve the tension in partial dislocations. 
We finally mention that examples of $\Gamma$-convergence of functionals concentrating energy on singular sets of different dimensions already appeared in the continuum setting in~\cite{GMM20} and in~\cite{BC22}.

%%%%%%%%%%%%%%%%%%%%%%%%%%%%%%%%%%%%%%%%%%
%          Discrete Set Up
%%%%%%%%%%%%%%%%%%%%%%%%%%%%%%%%%%%%%%%%%%
\section{Set up for the discrete model}\label{sec:setup-discrete}
In this section we introduce the relevant discrete energies for our problem and we state and prove some preliminary results about the relations between those energies.
\subsection{Basic notation} We start fixing some notation employed throughout.  We let $\dS(a,b)\defas 2\arcsin(\tfrac{1}{2}|a-b|)$ denote the geodesic distance between two unit vectors $a,b\in\S^1$; it satisfies
	\begin{equation}\label{est:geo-euclidian}
	|a-b|\leq\dS(a,b)\leq\frac{\pi}{2}|a-b|\,,
	\end{equation}
where $|\cdot|$ denotes the euclidian distance in $\R^2$. For any $R>r>0$ we set $B_r\defas\{x\in\R^2\colon |x|<r\}$ and $\A{r}{R}\defas B_R\sm \overline{B}_r$. Moreover, we let 
	\begin{equation}\label{def:plane}
	\Pi\defas\{x=(x_1,x_2)\in\R^2\colon x_2=0\}\quad\text{and}\quad \Pi^\pm\defas\{x\in\Pi\colon \pm x_1\geq 0\}
	\end{equation}
denote the plane orthogonal to $e_2$ and passing through the origin and its intersection with the positive and the negative real axis, respectively. The upper half space with boundary $\Pi$ is denoted by
	\begin{equation}\label{def:halfspace}
	H\defas\{x=(x_1,x_2)\in\R^2\colon x_2\geq 0\}\,.
	\end{equation}
%We fix an orientation of $\Pi$ by choosing $-e_2$ as a normal to it.
Eventually, for any $x\in\R^2$ we set $B_r(x)\defas B_r+x$, $\A{r}{R}(x)\defas \A{r}{R}+x$, $\Pi(x)\defas \Pi+x$, $\Pi^\pm(x)\defas\Pi^\pm+x$, and $H(x)\defas H+x$.
%and we define $\A{r}{R}^\top(x)$, $\A{r}{R}^\bottom(x)$, $B_r^\top(x)$, $B_r^\bottom(x)$, and $\Pi^\pm(x)$ analogously.
The closed segment joining two points $x,y\in\R^2$ is denoted by $[x,y]$.
%%%%%%%%%%%%%%%%%%%%%%%%%%%%%%%%%%%%%%%%%%%%%%%%%%%%%%%%%%%%%%%%%%%%%%
%                       Lattice
%%%%%%%%%%%%%%%%%%%%%%%%%%%%%%%%%%%%%%%%%%%%%%%%%%%%%%%%%%%%%%%%%%%%%%
\subsection{The discrete lattice}
We set the notation employed throughout for the underlying discrete lattice. For any $\e>0$, $k\in\{1,2\}$, and any Borel subset $A\subset\R^2$ we set
	\begin{equation}\label{def:lattice-A}
	\AZ{e_k}\defas\big\{i\in\e\Z^2\cap A\colon[i,i+\e e_k]\subset A\big\}\quad\text{and}\quad \AZ{2e_k}\defas\big\{i\in\e\Z^2\cap A\colon[i,i+2\e e_k]\subset A\big\}\,.
	\end{equation}
Moreover, we define the collection of closed cubes subordinated to the lattice $\e\Z^2$ via
	\begin{equation}\label{def:collection-cubes}
	\Q_\e\defas\big\{Q_\e=Q_\e(i)=i+[0,\e]^2\colon i\in\e\Z^2\big\}\,,
	\end{equation}
while for any Borel set $A\subset\R^2$ we let $\Q_\e(A)\defas\{Q_\e\in\Q_\e\colon Q_\e\subset A\}$ and $\partial_\e  A\defas\e\Z^2\cap\partial\big(\bigcup_{Q_\e\in\Q_\e(A)}Q_\e\big)$ denote the subclass of lattice cubes contained in $A$ and the discrete boundary of $A$, respectively. Eventually, for any cube $Q_\e=i+[0,\e]^2\in\Q_\e$ we write $b(Q_\e)\defas i+\frac{\e}{2}(e_1+e_2)$ for its barycentre.
It is also convenient to fix a triangulation $\T_\e$ of $\R^2$ subordinated to $\e\Z^2$ by setting
	\begin{equation}\label{triangulation}
	\T_\e\defas\Big\{T_\e^+=\conv(i,i+\e e_2,i+\e(e_1+e_2))\,,\ T_\e^-=\conv(i,i+\e e_1,i+\e (e_1+e_2))\colon i\in\e\Z^2\Big\}\,,
	\end{equation}
where $\conv$ denotes the closed and convex hull, see Figure~\ref{fig:lattices}. 

\smallskip
Moreover, we consider the shifted and double-spaced lattices 
	\begin{equation}\label{def:double-lattice-shifted}
	2\e\Z^2+\e s_j\subset\e\Z^2\;\text{ with }\; s_0\defas 0\,,\ s_1\defas e_1\,,\ s_2\defas e_2\,,\ s_3\defas e_1+e_2\,,
	\end{equation}
see Figure~\ref{fig:lattices}. Using the above notation we associate to any Borel subset $A\subset\R^2$  the discrete sets
	\begin{equation}\label{def:double-lattice-shifted-local}
	\AZS{s_j}{e_k}\defas\big\{i\in 2\e\Z^2+\e s_j\colon [i,i+2\e e_k]\subset A\big\}\,.
	\end{equation} 
In addition, for $j\in\{0\,,\ldots ,3\}$ we set $\Q_{2\e,s_j}\defas\Q_{2\e}+\e s_j$, $\T_{2\e,s_j}\defas\T_{2\e}+\e s_j$ and we finally define
	\begin{equation}\label{def:Z-even-Z-odd}
	2\e\Z^2_\even\defas(2\e\Z^2+\e s_0)\cup (2\e\Z^2+\e s_1)\,,\qquad 2\e\Z^2_\odd\defas (2\e\Z^2+\e s_2)\cup (2\e\Z^2+\e s_3)\,,
	\end{equation}
  see Figure~\ref{fig:lattices}. 

  \begin{figure}[ht]
    \begin{center}
    \scalebox{0.8}{
    \includegraphics{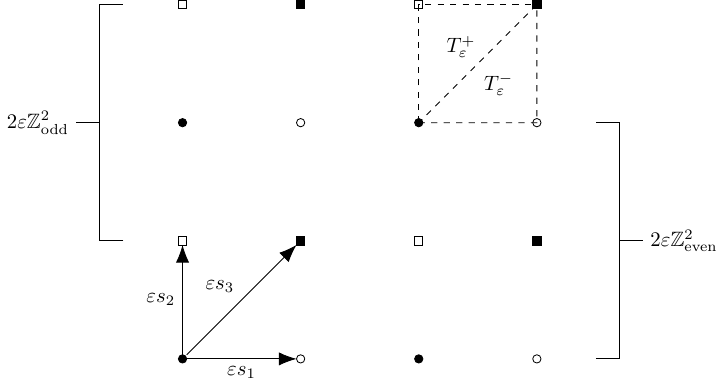}
    }
    \end{center}
    \caption{Local aspect of the discrete sets formed via the four double-spaced lattices: $\mathbb{Z}^{e_k}_{2 \e, s_0}(A)$ (black circles), $\mathbb{Z}^{e_k}_{2 \e, s_1}(A)$ (white circles), $\mathbb{Z}^{e_k}_{2 \e, s_2}(A)$ (white squares), 
    $\mathbb{Z}^{e_k}_{2 \e, s_3}(A)$ (black squares). The whole lattice $\e \mathbb{Z}^2$ is divided into the two sublattices $2\e \mathbb{Z}^2_{\mathrm{even}}$ (black and white circles) and $2\e \mathbb{Z}^2_{\mathrm{odd}}$ (black and white squares). The picture further shows the subdivision of a cell of the lattice $\e \mathbb{Z}^2$ into the two triangles $T^\pm_\e$.}
    \label{fig:lattices}
  \end{figure}

Next we introduce the sets of discrete variables taking values in the real numbers and the unit sphere, denoted respectively by
	\begin{equation*}
	\AD_\e\defas\{u:\e\Z^2\to\R\}\quad\text{and}\quad\SF_\e\defas\{v:\e\Z^2\to\S^1\}\,.
	\end{equation*}
Elements in $\AD_\e$ should be interpreted as $\e$-scaled admissible displacements of the lattice  $\e\Z^2$. As for elements in $\SF_\e$, we will refer to them as spin fields, due to the analogy to magnetic discrete models.

For $u,\tilde{u}\in\AD_\e$ we write 
	\begin{equation*}
	u\eqZ\tilde{u} \quad \text{ if } \quad u(i) - \tilde{u}(i) \in \Z \text{ for every } i\in\e\Z^2\,.
	\end{equation*} 
Note that to any $u\in\AD_\e$ we can associate a corresponding $v\in\SF_\e$ by interpreting the values of $v$ as complex numbers and by considering the complex exponential 
	\begin{equation}\label{eq:angular-lift}
	v=\exp(2\pi\iota u)\,, 
	\end{equation}
where $\iota$ is the imaginary unit. Vice versa, any $v\in\SF_\e$ can be written in the form~\eqref{eq:angular-lift} for some $u\in\AD_\e$ (not unique). We shall often interpret points on $\S^1$ as complex numbers and implicitly use complex products and complex powers.

%We also refer to $u$ satisfying~\eqref{eq:angular-lift} as an \textit{angular lifting} of $v$.
%%%%%%%%%%%%%%%%%%%%%%%%%%%%%%%%%%%%%%%%%%%%%
%             Discrete Gradients
%%%%%%%%%%%%%%%%%%%%%%%%%%%%%%%%%%%%%%%%%%%%
\subsection{Discrete gradients and discrete topological singularities} 
For $i,j\in\e\Z^2$ with $|i-j|=\e$ and $u\in\AD_\e$, $v\in\SF_\e$ we consider the directional discrete derivatives
	\begin{equation}\label{def:der-discrete}
	\d u(i,j)\defas u(j)-u(i)\qquad\d v(i,j)\defas v(j)-v(i)
	\end{equation}
Moreover, to any $u\in\AD_\e$ we associate a discrete vorticity measure as follows. For any $t\in\R$ let
	\begin{equation}\label{def:proj-Z}
	P_\Z(t)\defas \argmin\{|t-z|\colon z\in\Z\} = \Big\lceil t - \frac{1}{2} \Big\rceil 
	\end{equation}
denote its projection onto $\Z$ (with the convention of taking the minimal argmin in~\eqref{def:proj-Z} if it is not unique). Note that $P_\Z(t+z) = P_Z(t) + z$ for every $t\in\R$ and every $z\in\Z$. For $i,j\in\e\Z^2$ with $|i-j|=\e$ the elastic part of $\d u(i,j)$ is defined by
	\begin{equation}\label{def:d-elastic}
	\d^e u(i,j)\defas 
	\begin{cases}
	\d u(i,j)-P_\Z\big(\d u(i,j)\big) &\text{if}\ i\leq j\,,\\
	\d u(i,j)+P_\Z\big(\d u(j,i)\big) &\text{if}\ j\leq i\,,
	\end{cases}
	\end{equation}
where $i=(i_1,i_2)\leq j=(j_1,j_2)$ means that $i_1\leq j_1$ and $i_2\leq j_2$.
Note that $\d^e u(i,j)=-\d^e u(j,i)$. Moreover
	\begin{equation}\label{eq:d-elastic-distance}
	|\d^e u(i,j)|=\dist\big(\d u(i,j);\Z\big)\,.
	\end{equation}
For any square $Q_\e\in\Q_\e$ with vertices $\{i,j,k,\ell\}$ ordered counter clockwise we define a discrete circulation of $\d u$ around $Q_\e$ by setting
	\begin{equation}\label{def:circulation}
	\mu_u(Q_\e(i))\defas\d^e u(i,j)+\d^e u(j,k)+\d^e u(k,\ell)+\d^e u(\ell,i)\,.
	\end{equation}
By construction, $\mu_u(Q_\e)\in\{-1,0,1\}$. Eventually, we define the measure
	\begin{equation}\label{def:vorticity-u}
	\mu_u\defas\sum_{Q_\e\in\Q_\e}\mu_u(Q_\e)\delta_{b(Q_\e)}\,.
	\end{equation}
For any $j\in\{0,\ldots,3\}$ and any  $Q_{2\e}=i+[0,2\e]\in \Q_{2\e,s_j}  $ with $i\in\e\Z^2+s_j$ we define $\mu_u^{s_j}(Q)$ according to~\eqref{def:circulation} and we let $\mu_u^{s_j}$ be as in~\eqref{def:vorticity-u} with $\Q_\e$ replaced by $\Q_{2\e,s_j} $ and $\mu_u(Q_\e)$ replaced by $\mu_u^{s_j}(Q_{2\e})$.
If $v\in\SF_\e$ we write $v=\exp(2\pi\iota u)$ for some $u\in\AD_\e$ and set 
\begin{equation}\label{def:vorticity-v}
	\mu_v(Q_\e)\defas\mu_u(Q_\e)\,,\quad\mu_v\defas\mu_u\,,\quad \mu_v^{s_j}\defas\mu_u^{s_j}\,. 
\end{equation}
\begin{remark}\label{rem:lifting-v}
For any $v\in\SF_\e$ the measure $\mu_v$ is well-defined, since it does not depend on the choice of the angular lifting $u$. To see this, let $u,\tilde{u}\in\AD_\e$ with $u\eqZ\tilde{u}$. Since by construction $P_\Z(t+z)=P_\Z(t)+z$ for every $t\in\R$ and every $z\in\Z$, one can check that $\d^e u(i,j)=\d^e\tilde{u}(i,j)$ for every $i,j\in\e\Z^2$, $|i-j|=\e$. Thus, $\mu_u=\mu_{\tilde{u}}$.
\end{remark}
%%%%%%%%%%%%%%%%%%%%%%%%%%%%%%%%%%%%%%%%%%%%%%%%%%%
%                 Discrete energies
%%%%%%%%%%%%%%%%%%%%%%%%%%%%%%%%%%%%%%%%%%%%%%%%%%
\subsection{The discrete energies}\label{subsec:energies}
This subsection is a reference for all the discrete energies that we will use throughout the paper, listed in Table~\ref{table:energies} for convenience and defined precisely in the paragraphs below.
\begin{table}[ht] 
  \begin{tabularx}{\textwidth}{| c | c | >{\hsize=.3\hsize}X | c | >{\hsize=.7\hsize}X |} 
  \hline 
  {\bfseries Symbol} & {\bfseries Domain} & {\bfseries Name} & {\bfseries Formula}  & {\bfseries Comments} \\
  \hline
  $F_\e^\edge$ & $\AD_\e$ & Edge dislocations energy & \eqref{def:F-edge} & Displacements are in-plane and horizontal. No partial dislocations. Here, it will be often used on the doubled displacement. \\
  \hline  
  \rowcolor{black!5}
  $F_\e^\p$ & $\AD_\e$ & Partial edge dislocations energy & \eqref{def:F-part} & {\bfseries The energy studied in this paper}. Displacements are in-plane and horizontal. Accounts for partial dislocations. \\
  \hline
  \begin{tabular}[c]{@{}c@{}}$F_\e^\h$\\$F_\e^\v$\end{tabular} & $\AD_\e$ & Horizontal and vertical part of partial edge dislocations energy & \eqref{def:F-hor-ver} & Used to decompose $F_\e^\p$ conveniently. \\
  \hline
  $F_\e^\screw$ & $\AD_\e$ & Screw dislocations energy & \eqref{def:SD} & Displacements are out-of-plane. No partial dislocations. Here, it will be often used on the doubled displacement. \\
  \hline
  $XY_\e$ & $\SF_\e$ & $XY$ model & \eqref{def:XY} & Accounts for vortices in a spin field. Here, it will be often used on the complex square of spin fields. \\
  \hline
  $W\!M_\e^{\tau_1,\tau_2}$ & $\e \Z^2 \to \R^2$ & Weak-membrane energy & \eqref{def:WM} & Similar to the energy studied in~\cite{Ch99}, it is a discrete counterpart of the Mumford-Shah energy. Characterised by a threshold separating elastic and brittle behaviors. The constants $\tau_1$ and $\tau_2$ refer to the anisotropy for the threshold. \\
  \hline
  \end{tabularx}
  \caption{List of the discrete energies used in the paper. In the ``Domain'' column we stress whether the energy is a functional of displacements or spin fields.}
  \label{table:energies}
\end{table} 
 \EEE  We start introducing the pairwise interaction-energy densities used to define the energies. We let $f_0,f_1,f_{\frac{1}{2}}:\R\to[0,+\infty)$ defined by\footnote{We are using $\dist(t;\Z)$ instead of $\dist(t;\e \Z)$ since we are working with $\e$-scaled admissible displacements.}
	\begin{equation}\label{def:f1-f2}
	f_0(t)\defas 2\pi^2 t^2\,,\qquad f_1(t)\defas 2\pi^2\dist^2\big(t;\Z\big)\,,\qquad f_{\frac{1}{2}}(t)\defas 2\pi^2\dist^2\Big(t;\frac{1}{2}\Z\Big)=\frac{1}{4}f_1(2t)\,\,,
	\end{equation}
see Figure~\ref{fig:interaction-energy densities}. The precise value of the factor $2 \pi^2$ is not relevant for the behavior of the model studied in the paper and is introduced to have a more convenient expression of the energy expressed in terms of spin fields, see~\eqref{lb:edge-SD} below. The three pairwise interaction-energies \EEE satisfy $f_0\geq f_1\geq f_{\frac{1}{2}}$ and
\begin{equation}\label{equality:f0-f1-f2}
f_0(t)=f_1(t)\;\Longleftrightarrow\; P_\Z(t)=0\,,\quad\quad f_0(t)=f_{\frac{1}{2}}(t)\;\Longleftrightarrow\; P_\Z(2t)=0\,.
\end{equation}

\EEE

% Moreover, it is convenient to introduce \BBB [[Is it?]] \EEE \RRR for every threshold $\tau>0$ the functions $\f{\tau \e}:\R\to[0,+\infty)$ defined by
% \begin{equation}\label{def:fe} 
% \f{\tau \e}(t)\defas \max\Big\{ f_{\frac{1}{2}}(t), \tau \e \sum_{k\in\Z}\mathds{1}_{(k+1/4,k+3/4)}(t) \Big\}
% \end{equation} \EEE
% \RRR see Figure~\ref{fig:interaction-energy densities}. 

%This interaction energy will be used to define the partial screw dislocations energy (or generalised $XY$ model) considered in~\cite{BCDP18}. \EEE 

\begin{figure}[ht]
  \begin{center}
  \scalebox{0.8}{ 
  \includegraphics{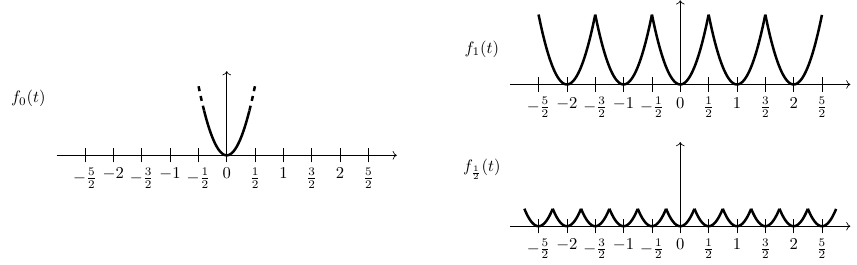}
  }
  \end{center}
  \caption{Comparison between the interaction-energy densities $f_0(t)$, $f_1(t)$, and $f_{\frac{1}{2}}(t)$. } %, and $\f{\tau \e}(t)$. }
  \label{fig:interaction-energy densities}
\end{figure}

\subsubsection{Edge dislocations energy $F_\e^\edge$}  
For any $u\in\AD_\e$ and any Borel set $A\subset\R^2$ we set
	\begin{equation}\label{def:F-edge}
	F_\e^\edge(u,A)\defas \sum_{i\in\AZ{e_1}}f_0\big(\d u(i,i+\e e_1)\big)+\sum_{i\in\AZ{e_2}}f_1\big(\d u(i,i+\e e_2)\big)
	\end{equation}
In~\eqref{def:F-edge}, $u$ should be interpreted as a horizontal (scaled) displacement of the lattice. The energy only depends on horizontal and vertical nearest-neighbors interactions. Horizontal interactions are defined in terms of the purely elastic potential $f_0$, while vertical interactions are defined via the 1-periodic potential $f_1$ typical of dislocations models. In this model, dislocations occur as a consequence of horizontal slips. \EEE

\subsubsection{Partial edge dislocations energy $F_\e^\p$} 
We define here the energy of the model studied in this paper. For any $u\in\AD_\e$ and any Borel set $A\subset\R^2$ we set
	\begin{equation}\label{def:F-part}
	F_\e^\p(u,A)\defas \!\!\! \sum_{i\in\AZ{e_1}} \!\!\! f_0\big(\d u(i,i+\e e_1)\big)+\!\!\! \sum_{i\in\AZ{e_2}}\!\!\! f_{\frac{1}{2}}\big(\d u(i,i+\e e_2)\big)+  \frac{\alpha}{\pi^2}  \, \e   \!\!\! \sum_{i\in\AZ{2e_2}}\!\!\! f_1\big(\d u(i,i+2\e e_2 )\big)\,.
	\end{equation}
%Eventually, we extend $F_\e^\p(\cdot,A)$ and $F_\e^\edge(\cdot,A)$ to the space
%	\begin{equation}\label{def:X}
%	X(A)\defas\bigg\{\mu=\sum_{h=1}^{N}d_h\delta_{x_h}\; \text{ with }\; N\in\N\,,\ d_h\in\Z\sm\{0\}\,,\ x_h\in A\,,\ x_h\neq x_{h'}\;\text{ for }\; h\neq h'\bigg\}
%	\end{equation}
%by setting
%	\begin{align}
%	\F_\e^\edge(\mu,A) &\defas\inf\big\{F_\e^\edge(u,A)\colon u\in\AD_\e\,,\ \mu_u\mres A=\mu\}\label{def:F-edge-mu}\\
%	\F_\e^\p(\mu,A) &\defas\inf\big\{F_\e^\p(u,A)\colon u\in\AD_\e\,,\ \mu_{2 u}\mres A=\mu\big\}\label{def:energy-mu}
%	\end{align}
%with the conventions $\inf\emptyset=+\infty$.
It will sometimes be convenient to split $F_\e^\p$ into an horizontal and a vertical part by writing
	\begin{equation}\label{eq:splitting-hor-ver} 
	F_\e^\p(u,A)= F_\e^\h(u,A)+F_\e^\v(u,A)
	\end{equation}
with $F_\e^\h$, $F_\e^\v$ defined by
	\begin{equation}\label{def:F-hor-ver}
	\begin{split}
	F_\e^\h(u,A)&\defas \sum_{i\in\AZ{e_1}}f_0\big(\d u(i,i+\e e_1)\big)\,,\\
	F_\e^\v(u,A) &\defas \sum_{i\in\AZ{e_2}}f_{\frac{1}{2}}\big(\d u(i,i+\e e_2)\big)+  \frac{\alpha}{\pi^2}  \, \e  \!\!\!  \sum_{i\in\AZ{2e_2}} \!\!\! f_1\big(\d u(i,i+ 2\e e_2 )\big)\,.  
	\end{split}
	\end{equation}

\subsubsection{Screw dislocations energy $F_\e^\screw$} 
We recall the definition of the screw dislocations energy studied, \eg in~\cite{ACP11,ADGP14}.  For every $u\in\AD_\e$ and any Borel set $A\subset\R^2$ we let \EEE 
\begin{equation}\label{def:SD}  
	F_\e^\screw(u,A)\defas \sum_{k=1}^2\sum_{i\in\AZ{e_k}} f_1\big(\d u(i,i+\e e_k)\big)
	\end{equation} 
  In~\eqref{def:SD}, $u$ should be interpreted as a out-of-plane (scaled) displacement of the lattice. The energy depends on horizontal and vertical nearest-neighbors interactions. Both are defined via the 1-periodic potential $f_1$.

\subsubsection{The $XY$ model} We recall here the energy of the $XY$ model studied in~\cite{AC09}. For any $v\in\SF_\e$ and any Borel set $A\subset\R^2$ we let
	\begin{equation}\label{def:XY}
	 XY_\e(v,A)\defas\frac{1}{2}\sum_{k=1}^2\sum_{i\in\AZ{e_k}}|\d v(i,i+\e e_k)|^2\,.
	\end{equation}
	\subsubsection{The weak membrane energy $W\!M_\e$}
Eventually, we will make use of the so-called weak membrane energies analogous to the one studied in~\cite{Ch99}. In our setting, it is convenient to consider them in the following form: For any $w \colon \e\Z^2\to\R^2$, any $\tau_1, \tau_2 >0$, and $A\subset\R^2$ Borel we define  
	\begin{equation}\label{def:WM}
	W\!M_\e^{\tau_1,\tau_2}(w,A)\defas\sum_{i\in\AZ{e_1}}\min\Big\{ \frac{1}{2}  |\d w(i,i+\e e_1)|^2,\tau_1 \e\Big\}+\sum_{i\in\AZ{e_2}}\min\Big\{ \frac{1}{2}  |\d w(i,i+\e e_2)|^2,\tau_2 \e\Big\}\,.
	\end{equation}

We recall the following compactness and lower bound results for $W\!M_\e$ proven~\cite{Ch99, Ruf19} and stated here in the form applied in this paper.

\begin{theorem}[Compactness for $W\!M_\e$]\label{thm:WM compactness}
  Let $\Omega\subset\R^2$ be open, bounded, and with Lipschitz boundary. Let $\tau_1, \tau_2 > 0$. Let $w_\e:\e\Z^2\to\R^2$ be such that $\sup_{\e > 0} W\!M_\e^{\tau_1,\tau_2}(w_\e,\Omega) < +\infty$. Assume that  $w_\e$ are equiintegrable. Then there exists $w \in L^1(\Omega;\R^2) \cap GSBV^2(\Omega;\R^2)$ and a subsequence (not relabeled) such that $w_\e \to w$ strongly in $L^1(\Omega;\R^2)$. In particular, if $\sup_{\e > 0} \|w_\e\|_{L^\infty} < +\infty$, then $w \in SBV^2(\Omega;\R^2)$. 
\end{theorem}

  \begin{theorem}[Lower bound for $W\!M_\e$]\label{thm:WM}
  Let $\Omega\subset\R^2$ be open, bounded, and with Lipschitz boundary. Let $w\in SBV^2(\Omega;\R^2)$ and suppose that $w_\e \colon \e\Z^2\to\R^2$ are such that $w_\e$ converge strongly in $L^1(\Omega;\R^2)$ to $w$. Then the lower bound
    \begin{equation}\label{liminf-WM}
    \liminf_{\e\to 0}W\!M_\e^{\tau_1,\tau_2}(w_\e,\Omega)\geq \frac{1}{2} \int_\Omega|\nabla w|^2\dx+\int_{S_w}\big(\tau_1 |\nu_w\cdot e_1|+\tau_2|\nu_w\cdot e_2|\big)\dH
    \end{equation}
  holds true for every $\tau_1, \tau_2 >0$.
  \end{theorem}

\subsection{Comparison between the discrete energies} In this subsection we compare the discrete energies introduced in Subsection~\ref{subsec:energies}. The results presented here will be useful for the proofs in this paper. We sum up the relations between the discrete energies in the schematic Diagram~\eqref{diag:comparison}, referring to the subsections where they are explained in detail.  
\begin{equation} \label{diag:comparison}
  \begin{tikzcd}
      {F_\e^\p} &&&&&&&&&& {W\!M_\e} \\
      \\
      {F_\e^\edge} &&&&& {F_\e^\screw} &&&&& {XY_\e}
      \arrow["{{\gtrsim \text{ (Subsec.~\ref{subsec:Fpedge Fedge}})}}"{description}, tail, from=1-1, to=3-1]
      \arrow["{{\simeq  \text{ (Subsec.~\ref{subsec:Fedge Fscrew}})}}"{description}, tail reversed, from=3-1, to=3-6]
      \arrow["{{\simeq  \text{ (Subsec.~\ref{subsec:Fscrew XY}})}}"{description}, tail reversed, from=3-6, to=3-11]
      \arrow["{\gtrsim \text{ (Subsec.~\ref{subsec:Fpedge Fpscrew})}}"{description}, tail, from=1-1, to=1-11]
\end{tikzcd}
\end{equation}
  
  \subsubsection{Comparison between $F_\e^\p$ and $F_\e^\edge$} \label{subsec:Fpedge Fedge}
  We start  observing that the energies $F_\e^\p$ and $F_\e^\edge$ can be related by a transformation that doubles the displacement. More precisely, by~\eqref{def:F-edge} and~\eqref{def:F-part}, we have that 
	\begin{equation}\label{eq:F-edge-F-part}
	F_\e^\p(u,A)=\frac{1}{4}F_\e^\edge(2u,A)+ \frac{\alpha}{\pi^2} \, \e  \sum_{i\in\AZ{2e_2}}f_1\big(\d u(i,i+ 2\e e_2 )\big)\,.
	\end{equation}
  In particular, \eqref{eq:F-edge-F-part} yields the following inequality that will be often used throughout the paper:
  \begin{equation} \label{eq:4F-part-larger-F-edge}
   4 F_\e^\p(u,A) \geq F_\e^\edge(2u,A) \, .
  \end{equation}

  % \subsubsection{Comparison between $F_\e^\pscrew$ and $F_\e^\screw$} \label{subsec:Fpscrew Fscrew} \RRR By the definition of $\f{\e}$ in~\eqref{def:fe} we obtain that 
  % \begin{equation} \label{eq:F-pscrew-larger-F-screw}
  %   4 F_\e^\pscrew(u,A) \geq F_\e^\screw(2u,A) \, .
  %  \end{equation}

  \subsubsection{Comparison between $F_\e^\edge$ and $F_\e^\screw$} \label{subsec:Fedge Fscrew} Here we show that the energies $F_\e^\edge$ are equivalent to the energies $F_\e^\screw$. This observation will be the key for the proof of a result on the asymptotic behavior of $F_\e^\edge$.

  \smallskip
  \begin{remark}\label{rem:lb-edge-screw}
  For any $u\in\AD_\e$ such that $2 \pi u$ is an angular lifting of $v$, \ie $v=\exp(2\pi\iota u)$, and any $i,j\in\e\Z^2$ with $|i-j|=\e$ we have thanks to~\eqref{est:geo-euclidian}
  \begin{equation}\label{eq:geo-dist}
    f_1(\d u(i,j))=2\pi^2\dist^2\big(\d u(i,j);\Z\big)=\frac{1}{2}\d_{\S^1}^2\big(v(i),v(j)\big)\geq \frac{1}{2}|\d v(i,j)|^2\,.
  \end{equation}
  This motivates the factor $2\pi^2$ in~\eqref{def:f1-f2}. Indeed, in this way the chain of inequalities
    \begin{equation}\label{lb:edge-SD}
    F_\e^\edge(u,A)\geq F_\e^\screw(u,A)\geq XY_\e(v,A) 
    \end{equation}
  holds for any Borel subset $A\subset\R^2$. Moreover, \eqref{equality:f0-f1-f2} implies that
    \begin{equation}\label{cond:SD-equals-F-edge}
    F_\e^\edge(u,A)=F_\e^\screw(u,A)\quad\Longleftrightarrow\quad \d^e u(i,i+\e e_1)=\d u(i,i+\e e_1)\;\text{ for all}\ i\in\Z_\e^{e_1}(A)\,.
    \end{equation}
  \end{remark}
Remark~\ref{rem:lb-edge-screw} above in particular gives a lower bound for $F_\e^\edge$ in terms of $F_\e^\screw$. The following lemma shows that equality can be reached for a suitable representative, that is, we can always construct a representative satisfying~\eqref{cond:SD-equals-F-edge}.  
  \begin{lemma}\label{lem:SD-equals-F}
  Let $U\subset\R^2$ be open and bounded, $\e>0$ and $u\in\AD_\e$. There exists $\tilde{u}\in\AD_\e$ with $\tilde{u}\eqZ u$ such that 
    \begin{equation}\label{eq:de-equals-d}
    \d^e\tilde{u}(i,i+\e e_1)=\d^e u(i,i+\e e_1)=\d\tilde{u}(i,i+\e e_1)
    \end{equation}
  for every $i\in\Z_\e^{e_1}(U)$ and, consequently,
    \begin{equation}\label{eq:SD-equals-F}
    F_\e^\screw(u,U)=F_\e^\screw(\tilde{u},U)=F_\e^\edge(\tilde{u},U)\,.
    \end{equation}
  \end{lemma}
  \begin{proof}
  Given $u\in\AD_\e$ and $U\subset\R^2$ open and bounded we decompose the lattice portion $\e\Z^2\cap U$ into the union of horizontal slices $\e\Z^2\cap\Pi(\e k)\cap U$ with $k\in\Z$ and we will suitably modify $u$ by removing horizontal jumps along those slices. Let us first assume that $U$ is convex, so that $\Pi(\e k)\cap U$ is connected for every $k\in\Z$. Thus, for any $k\in\Z$ with $\Pi(\e k)\cap U\neq\emptyset$ there exists $i_k=i_k(\e)\in\e\Z^2\cap\Pi(\e k)\cap U$ and $N_k=N_k(\e)\in\N$ such that
    \begin{equation*}
    \e\Z^2\cap\Pi(\e k)\cap U=\{i_k+j\e e_1\colon j=0,\ldots,N_k\}\,.
    \end{equation*}
  We now define $\tilde{u}$ on $\e\Z^2\cap\Pi(\e k)\cap U$ recursively by setting $\tilde{u}(i_k)\defas u(i_k)$ and
    \begin{align}\label{def:utilde}
    \tilde{u}(i_k+ \e j e_1)\defas\tilde{u}(i_k+ \e (j-1) e_1)+\d^eu(i_k+\e (j-1)e_1,i_k+\e je_1)\;\text{ for}\ j\in\{1,\ldots,N_k\}\,.
    \end{align}
  Since the family $\{ \e\Z^2\cap\Pi(\e k)\cap U \}_{k\in\Z}$ decomposes $\e\Z^2\cap U$, the resulting configuration $\tilde{u}$ is well defined on $\e\Z^2\cap U$ and we simply extend it by $u$ to $\e\Z^2\sm U$. Moreover, by definition the equality 
    \begin{equation}\label{eq:dtilde-equals-de}
    \d \tilde{u}(i,i+\e e_1)=\d^e u(i,i+\e e_1)
    \end{equation}
  holds true for any $i\in\Z_\e^{e_1}(U)$. Eventually, \eqref{def:utilde} can be rewritten as
    \begin{align*}
    \tilde{u}(i_k+ \e j e_1)= u(i_k) \EEE  -\sum_{\ell=0}^jP_\Z\big(\d u(i_k+\e \ell e_1,i_k+\e(\ell+1) e_1 ) \big)\,,
    \end{align*}	 
   so that $\tilde{u}\eqZ u$, hence $F_\e^\screw(u,U)=F_\e^\screw(\tilde{u},U)$.
  Moreover, $\tilde{u}\eqZ u$ together with~\eqref{eq:dtilde-equals-de} gives~\eqref{eq:de-equals-d} and we conclude in view of~\eqref{cond:SD-equals-F-edge}.
  
  If $U$ is  non-convex we repeat the above procedure on the connected components of $\Pi(\e k)\cap U$.
  \end{proof}
  \begin{remark}\label{rem:mu-equals-mutilde}
  Let $u,\tilde{u}$ be as in Lemma~\ref{lem:SD-equals-F}; since $\tilde{u}\eqZ u$, by Remark~\ref{rem:lifting-v} we clearly have $\mu_{\tilde{u}}=\mu_u$.
  \end{remark}

  \subsubsection{Comparison between $F_\e^\screw$ and $XY_\e$} \label{subsec:Fscrew XY} Here we simply recall that in~\cite{ACP11} it is shown that $XY_\e$ defined in~\eqref{def:XY} is asymptotically equivalent to $F_\e^\screw$ defined in~\eqref{def:SD} from a variational point of view. This equivalence is based on the correspondence between spin fields $v \in \SF_\e$ and  angular liftings $2\pi u$, with $u \in \AD_\e$.  

  \EEE

\subsubsection{Comparison between $F_\e^\p$ and $W\!M_\e$}  \label{subsec:Fpedge Fpscrew}
Now we show that the energies $F_\e^\p$ can be suitably bounded from below by the weak-membrane energies on the collection of double-spaced lattices $2\e\Z+s_j$ defined in~\eqref{def:double-lattice-shifted-local}. To be precise, for $j\in\{0\,,\ldots ,3\}$ and $u \colon (2 \e \Z^2 + s_j) \to \R$ we set  
	\begin{equation}\label{def:XYgen-shifted}
    F_{2\e,s_j}^{\edge}(u,A)\defas \sum_{i\in\AZS{s_j}{e_1}}f_0\big(\d u(i,i+ 2\e e_1)\big)+\sum_{i\in\AZS{s_j}{e_2}}f_1\big(\d u(i,i+ 2\e e_2)\big) \,,
	\end{equation}
and we define $F_{2\e,s_j}^{\screw}$, $XY_{2\e,s_j}$, and $W\!M_{2\e,s_j}^{\tau_1,\tau_2}$ accordingly. For any function $w$ defined on $\e\Z^2$ and any  $j\in\{0\,,\ldots ,3\}$ we define the restrictions $w_{2\e, s_j} := w_{\e | 2\e \Z^2 + s_j}$. 

We have the following result.
\begin{lemma}\label{lem:Fpart-XYgen-WM}
Let $A\subset\R^2$ be a Borel set, let $C_0>0$, and let $(u_\e)_\e$ be a sequence of configurations $u_\e\in\AD_\e$ satisfying $F_\e^\p(u_\e,A)\leq C_0|\log\e|$. Moreover, let $w_\e\defas \exp(2\pi\iota u_\e)$. Then 
% there exist $C_1>0$ and $\e_0>0$ such that for every $\e\in(0,\e_0)$ we have \BBB [[Not needed]] \EEE
% 	\begin{equation}\label{est:Fe-XYgen}
% 	\BBB  F^\p_\e(u_\e,A)\geq\frac{1}{4}\sum_{j=0}^3 F_{2\e,s_j}^{\pscrew}(u_\e,A)-C_1\sqrt{\e}|\log\e| \EEE
% 	\end{equation}
for every $R>0$ there exists $\e_R>0$ such that the estimate 
	\begin{equation}\label{est:Fe-WMR}
	F^\p_\e(u_\e,A)\geq\frac{1}{4}\sum_{j=0}^3 W\!M_{2\e,s_j}^{R,\alpha}(w_{2\e,s_j},A)-C_1\sqrt{\e}|\log\e| 
	\end{equation}
holds for every $\e\in(0,\e_R)$. 
\end{lemma}
\begin{proof}
 For $A\subset\R^2$ and $u_\e\in\AD_\e$ as in the statement we use the decomposition~\eqref{eq:splitting-hor-ver} and we estimate separately $F_\e^\v(u_\e,A)$ and $F_\e^\h(u_\e,A)$. Namely, we find $C_1>0$ such that
	\begin{equation}\label{est:Fe-ver}
	F^\v_\e(u_\e,A) \geq\frac{1}{4}\sum_{j=0}^3 \sum_{i\in\AZS{s_j}{e_2}}  \min \Big\{ \frac{1}{2} |\d w_\e(i,i+2\e e_2) |^2 , 2 \alpha \e \Big\} - C_1\sqrt{\e}|\log\e| 
	\end{equation}
for every $\e>0$, and for any $R>0$ we find $\e_R>0$ such that
	\begin{equation}\label{est:Fe-hor}
    F^\h_\e(u_\e,A) \geq\frac{1}{4}\sum_{j=0}^3 \sum_{i\in\AZS{s_j}{e_1}} \min \Big\{ \frac{1}{2} |\d w_\e(i,i+2\e e_1) |^2 , 2 R \e \Big\} \,,
	\end{equation}
for any $\e\in(0,\e_R)$. Then~\eqref{est:Fe-WMR} follows by the very definition of $W\!M_{2\e,s_j}^{R,\alpha}$ in~\eqref{def:WM}. 

\begin{step}{1}{Proof of~\eqref{est:Fe-ver}}
We start observing that
	\begin{equation}\label{est:Fe-ver-1}
  \begin{split}  
    & F_\e^\v(u_\e,A) \\
    & \geq \hspace*{-.5em}\sum_{i\in\AZ{2e_2}}\hspace*{-.75em}  \Big( \frac{1}{2}  f_{\frac{1}{2}}\big(\d u_\e(i,i+\e e_2)\big)+ \frac{1}{2}  f_{\frac{1}{2}}\big(\d u_\e(i+\e e_2,i+2\e e_2)\big) +  \frac{\alpha}{\pi^2} \e  f_1\big(\d u_\e (i,i+2\e e_2)\big) \Big)\,.   
  \end{split}
	\end{equation}
  The right-hand side in~\eqref{est:Fe-ver-1} can be estimated by a convexity argument as follows: For every $t_1,t_2\in\R$ we have
	\begin{equation}\label{est:inf-conv-convexity0}
	f_{\frac{1}{2}}(t_1)+f_{\frac{1}{2}}(t_2)\geq\frac{1}{2}f_{\frac{1}{2}}(t_1+t_2)\,.
	\end{equation}
To see this, it is enough to recall that $f_{\frac{1}{2}}(t)=\frac{1}{4}\dist^2(2t;\Z)$ and to choose $z_1,z_2\in\Z$ satisfying
	\begin{equation*}
	\dist^2\big(2t_1;\Z\big)=\big(2t_1-z_1\big)^2\qquad\text{and}\qquad\dist^2\big(2t_2;\Z\big)=\big(2t_2-z_2\big)^2\,.
	\end{equation*}
By convexity one obtains that
	\begin{equation*}
	\big(2t_1-z_1\big)^2+\big(2t_2-z_2\big)^2\geq\frac{1}{2}\big(2(t_1+t_2)-(z_1+z_2)\big)^2
	\geq\frac{1}{2}\dist^2\big(2(t_1+t_2);\Z\big)\,,
	\end{equation*}
which gives~\eqref{est:inf-conv-convexity0}.
To apply~\eqref{est:inf-conv-convexity0} in~\eqref{est:Fe-ver-1}, for every $i\in\AZ{2e_2}$ we write
	\begin{equation*}
	\d u_\e(i,i+2\e e_2)=\d u_\e(i,i+\e e_2)+\d u_\e(i+\e e_2,i+2\e e_2)\,,
	\end{equation*}
  to obtain from~\eqref{est:Fe-ver-1}  that 
  \[
      F_\e^\v(u_\e,A)\geq \hspace*{-.5em}\sum_{i\in\AZ{2e_2}}\hspace*{-.75em}  \Big( \frac{1}{4} f_{\frac{1}{2}}\big(\d u_\e(i,i+2 \e e_2)\big) + \frac{\alpha}{\pi^2} \e f_1\big(\d u_\e (i,i+2\e e_2)\Big) \, .
  \] 
Summing up the contributions on the four lattices $2\e\Z^2+s_j$,   we deduce that
	\begin{equation}\label{est:inf-conv}
	 F_\e^\v(u_\e,A) \geq  \frac{1}{4}  \sum_{j=0}^3\sum_{i\in\AZS{s_j}{e_2}}  g_\e\big(\d u_\e(i,i+2\e e_2)\big)\,,  
	\end{equation}
where $g_\e:\R\to[0,+\infty)$ is defined by  
	\begin{equation}\label{est:inf-conv-convexity}
	 g_\e(t)\defas f_{\frac{1}{2}}(t)+{ \frac{4 \, \alpha}{\pi^2} \e  }f_1(t)=\frac{1}{4}f_1(2t)+{\frac{4 \, \alpha}{\pi^2} \e }f_1(t)\,. 
	\end{equation}
Let us show that for a constant $C>0$ we have that
	\begin{equation}\label{est:claim-ge-fe}
    g_\e\big(\d u_\e(i,i+2\e e_2)\big) \geq  \min \Big\{ \frac{1}{2} |\d w_\e(i,i+2\e e_2) |^2 , 2 \alpha \e \Big\}  -C\e\sqrt{\e}f_1\big(\d u_\e(i,i+2\e e_2)\big) \,.
	\end{equation}
  Then, from~\eqref{est:inf-conv} we obtain~\eqref{est:Fe-ver} with $C_1=C\,C_0$. To prove~\eqref{est:claim-ge-fe}, we distinguish two cases. 

  {\itshape Case 1}: If $\dist\big(\d u(i,i+2\e e_2);\Z\big) \leq \frac{1}{4}$, then $\d u(i,i+2\e e_2)$ is far enough from $\frac{1}{2}\Z \setminus \Z$ to deduce that $\frac{1}{4}f_1\big(2\d u(i,i+2\e e_2)\big) = f_1\big(\d u(i,i+2\e e_2) \big)$. This, together with~\eqref{eq:geo-dist} gives us that 
  \[
    g_\e\big(\d u_\e(i,i+2\e e_2)\big) \geq  \frac{1}{4}f_1\big(2\d u(i,i+2\e e_2)\big) = f_1\big(\d u(i,i+2\e e_2) \big) \geq  \frac{1}{2} |\d w_\e(i,i+2\e e_2)|^2 \, ,
  \]
  whence~\eqref{est:claim-ge-fe}. 

  {\itshape Case 2}: If, instead, $\dist\big(\d u(i,i+2\e e_2);\Z\big) > \frac{1}{4}$, let $t = \d u(i,i+2\e e_2)$. We assume without loss of generality that $t\in(1/4,3/4)$, since the general case follows by periodicity. For such $t$ we have the trivial estimate
	\begin{equation}\label{est:trivial}
	\e t^2\geq\frac{\e}{16}\,, 
	\end{equation}
which however is not enough. To improve the above estimate we consider the following exhaustive cases: If $|t-1/2|\geq \beta \sqrt{\e}$ (for a suitable $\beta>0$ to be chosen below), we clearly have 
	\begin{equation}\label{est:close-to-half}
	f_{\frac{1}{2}}(t)=2\pi^2(t-1/2)^2\geq 2 \pi^2 \beta^2 \e \,.
	\end{equation}
If instead $|t-1/2|<\beta \sqrt{\e}$, a direct computation yields
	\begin{equation*}
	\frac{4 \, \alpha}{\pi^2} \e f_1(t)= 8 \alpha \e\min\big\{t^2,(t-1)^2\big\}\geq 8 \alpha \e\Big(\frac{1}{2}- \beta \sqrt{\e}\Big)^2\geq 2 \alpha \e - 8 \alpha \beta \e \sqrt{\e}\,.  
	\end{equation*}
 Choosing $\beta = \frac{\sqrt{\alpha}}{\pi}$ so that $2 \pi^2 \beta^2 = 2 \alpha$,  combining this with~\eqref{est:inf-conv-convexity} and the trivial bound~\eqref{est:trivial} finally yields
	\begin{equation}\label{est:ge-jump}
	g_\e(t)\geq 2 \alpha \e -C\e\sqrt{\e}f_1(t) \,,  
	\end{equation}
which yields~\eqref{est:claim-ge-fe} and concludes this step.
\EEE

\end{step}

\begin{step}{2}{Proof of~\eqref{est:Fe-hor}} By convexity we have that
	\begin{equation}\label{est:horizontal-convexity}
	\begin{split}
	F_\e^\h(u_\e,A) &\geq \pi^2\sum_{i\in\AZ{2e_1}}|\d u_\e(i,i+\e e_1)|^2+|\d u_\e(i+\e e_1,i+2\e e_1)|^2\\
	&\geq \frac{1}{4} \sum_{i\in\AZ{2e_1}} 2 \pi^2 |\d u_\e(i,i+2\e e_1)|^2\,.
	\end{split}
	\end{equation}
If $|\d u_\e(i,i+2\e e_1)| \geq \frac{1}{4}$, for any $R>0$ we have that $2 \pi^2 |\d u_\e(i,i+2\e e_1)|^2 \geq 2R\e$, provided $\e\leq\frac{\pi^2}{16R}$. If $|\d u_\e(i,i+2\e e_1)| < \frac{1}{4}$, then~\eqref{eq:geo-dist} yields that $2 \pi^2 |\d u_\e(i,i+2\e e_1)|^2 = f_1\big(\d u_\e(i,i+2\e e_1) \big) \geq \frac{1}{2} |\d w_\e(i,i+2\e e_1)|$. Hence~\eqref{est:Fe-hor} follows from~\eqref{est:horizontal-convexity} by summing up the contributions on the four sublattices $2\e\Z^2+\e s_j$. This concludes the proof.
\end{step}
\end{proof}
\begin{remark}\label{rem:F-edge-double}
Combining~\eqref{est:Fe-ver-1}, \eqref{est:inf-conv-convexity0}, and~\eqref{est:horizontal-convexity} and using that $f_{\frac{1}{2}}(t)=\frac{1}{4}f_1(t)$, we also obtain that
	\begin{equation}\label{est:F-edge-double}
	F_\e^\p(u_\e,A)\geq\frac{1}{4}\sum_{j=0}^3F_{2\e,s_j}^{\edge}(2 u_{2\e,s_j},A) 
	\end{equation}
for $u_\e\in\AD_\e$ and any Borel set $A\subset\R^2$.
\end{remark} 

%%%%%%%%%%%%%%%%%%%%%%%%%%%%%%%%%%%%%%%%%%%%%%%%%%
%             Set Up Continuum
%%%%%%%%%%%%%%%%%%%%%%%%%%%%%%%%%%%%%%%%%%%%%%%%%%
\section{Set up for the continuum model}
In this section we introduce the relevant function spaces to characterise the $\Gamma$-limit of suitable rescalings of the discrete energies introduced in Section~\ref{sec:setup-discrete}. 
\subsection{Limiting configurations of singularities} \label{subsec:dislocation measures}
We start this section by recalling the definition of the spaces $X$ and $X_M$ of relevant (limiting) singularity configurations. For any $U\subset\R^2$ open and $M\in\N$ we consider the families of measures
	\begin{equation}\label{def:X}
	X(U)\defas\bigg\{\mu=\sum_{h=1}^{N}d_h\delta_{x_h}\; \text{ with }\; N\in\N\,,\ d_h\in\Z\sm\{0\}\,,\ x_h\in U\,,\ x_h\neq x_{h'}\;\text{ for }\; h\neq h'\bigg\}
	\end{equation}
and
	\begin{equation}\label{def:XM}
	X_M(U)\defas\bigg\{\mu=\sum_{h=1}^{M}d_h\delta_{x_h}\in X(U)\; \text{ with }\; d_h\in\{-1,1\}\bigg\}\,.
	\end{equation}
It will be convenient to equip $X(U)$ with the convergence induced by the flat topology. Namely, for any distribution $T\in\D'(U)$ we let
	\begin{equation*}
	\|T\|_{\rm flat}\defas\sup\big\{\langle T,\psi\rangle\colon \psi\in C_c^\infty(U)\,,\ \|\psi\|_{L^\infty(U)}\leq 1\,,\ \|\nabla\psi\|_{L^\infty(U)}\leq 1\big\}
	\end{equation*}	 
be its flat norm.
We say that a sequence $(\mu_n)_n \subset X(U)$ converges flat to some $\mu\in X(U)$ and we write $\mu_n\fto\mu$, if $\|\mu_n-\mu\|_{\rm flat}\to 0$ as $n\to+\infty$.

\smallskip
The space $X_M$ contains the relevant limiting configurations of singularities (\cf~Theorem~\ref{thm:main2}). The relevant limiting fields will belong to some classes of Sobolev and $SBV$-functions and will be related to a measure $\mu\in X_M(U)$ via their Jacobian. To introduce those classes properly we start recalling the notion of Jacobian and degree for maps in the continuum together with some basic properties of $SBV$-functions. 
\subsection{Jacobians and degree}
We recall here some definitions and basic results concerning topological singularities. Let $U \subset \R^2$ be an open set and let $v = (v_1,v_2) \in W^{1,1}(U;\R^2) \cap L^\infty(U;\R^2)$. We define the {\em pre-Jacobian} (also known as {\em current}) of $v$ by 
\begin{equation*}
    j(v) := \frac{1}{2}( v_1 \nabla v_2 - v_2 \nabla v_1 ) \, .
\end{equation*}
The {\em distributional Jacobian} of $v$ is defined by 
\begin{equation*}
    J(v) := \curl (j(v)) \, ,
\end{equation*}
in the sense of distributions, \ie 
\begin{equation*}
    \langle J(v), \psi \rangle = - \int_U j(v) \cdot \nabla^\perp \psi \, \d x \quad \text{for every } \psi \in C^\infty_c(U) \, ,
\end{equation*}
where $\nabla^\perp = (-\de_2, \de_1)$. Note that $J(v)$ is also well-defined when $v \in H^1(U;\R^2)$, and, in that case, it coincides with the $L^1$ function $\det \nabla v$.
 
Given $v =(v_1,v_2) \in H^{\frac{1}{2}}(\de B_\rho(x_0);\S^1)$, its {\em degree} is defined by 
\begin{equation} \label{def:degree}
    \deg(v,\de B_\rho(x_0)) := \frac{1}{2 \pi} \big( \langle \nabla_{\de B_\rho(x_0)} v_2 , v_1 \rangle_{H^{-\frac{1}{2}},H^\frac{1}{2}} - \langle \nabla_{\de B_\rho(x_0)} v_1 , v_2 \rangle_{H^{-\frac{1}{2}},H^\frac{1}{2}} \big) \, ,
\end{equation}
where $\langle \cdot  , \cdot  \rangle_{H^{-\frac{1}{2}},H^\frac{1}{2}}$ denotes the duality between $H^{-\frac{1}{2}}(\de B_\rho(x_0);\S^1)$ and $H^{\frac{1}{2}}(\de B_\rho(x_0);\S^1)$ and we let  $\nabla_{\de B_\rho(x_0)}$ denote the derivative on $\de B_\rho(x_0)$ with respect to the unit speed parametrization of $\de B_\rho(x_0)$. Note that, by definition, the map $v \in H^{\frac{1}{2}}(\de B_\rho(x_0);\S^1) \mapsto \deg(v,\de B_\rho(x_0))$ is continuous. We remark that
\begin{equation*}
    \deg(v,\de B_\rho(x_0)) = \frac{1}{2\pi}\int_{\de B_\rho(x_0)}  \big( v_1 \nabla_{\de B_\rho(x_0)} v_2  - v_2 \nabla_{\de B_\rho(x_0)} v_1\big) \, \d \H^1 \quad \text{if } v \in H^1(\de B_\rho(x_0);\S^1)
\end{equation*} 
(and thus $v$ is continuous) and this notion coincides with the classical notion of degree.
%\footnote{One can see this by noticing that
%\begin{equation*}
%    \frac{1}{2\pi}\int_{\de B_\rho}   v_1 \nabla_{\de B_\rho} v_2  - v_2 \nabla_{\de B_\rho} v_1 \, \d \H^1 = \fint_{\de B_\rho}   v^* \omega_{\de B_\rho}  = \deg(v,\de B_\rho) \fint_{\de B_\rho} \omega_{\de B_\rho}  = \deg(v,\de B_\rho) \, ,
%\end{equation*} 
%where $v^* \omega_{\de B_\rho}$ is the pull-back through $v$ of the volume form $\omega_{\de B_\rho}$ on $\de B_\rho$ and the second equality is due to the topological definition of degree. 
%}
Also when $v \in H^{\frac{1}{2}}(\de B_\rho(x_0);\S^1)$ is discontinuous, the degree defined in~\eqref{def:degree} inherits from the continuous setting some characterizing properties. In particular, a result due to L.\ Boutet de Monvel \& O.\ Gabber~\cite[Theorem~A.3]{BdMBGeoPur} ensures that
%\footnote{In~\cite[Theorem~A.3]{BdMBGeoPur} the degree formula is written in an alternative form, equivalent to~\eqref{def:degree}, by interpreting~$v$ as a complex-valued function.} 
$\deg(v,\de B_\rho(x_0)) \in \Z$.

A further fundamental property of the degree is the following. Let $v \in H^1(\A{r}{R}(x_0);\S^1)$. By the trace theory, $v|_{\de B_{\rho}(x_0)} \in H^{\frac{1}{2}}(\de B_{\rho}(x_0);\S^1)$ for every $\rho \in [r, R]$. Then 
\begin{equation} \label{eq:degree independent}
    \deg(v,\de  B_{\rho}(x_0)) = \deg(v,\de  B_{\rho'}(x_0)) \quad \text{for every } \rho, \rho' \in [r,R] \, .
\end{equation}
This follows from the fact that $\deg(v,\de  B_{\rho}(x_0)) \in \Z$, by the continuity of the degree with respect to the $H^\frac{1}{2}$ norm, and by the continuity of the map
 \begin{equation*}  
     \rho \in [r,R] \mapsto v( x_0 + \rho \, \cdot \, )|_{\de B_{1}} \in H^{\frac{1}{2}}(\de B_{1};\S^1) \, , 
 \end{equation*}
 which is a consequence of the trace theory for Sobolev functions.

We conclude this summary about the degree by recalling the following property. Let $v \in H^1(\A{r}{R}(x_0);\S^1)$. By the theory of slicing of Sobolev functions (\cf~\cite[Proposition~3.105]{AFP} with a change of coordinates), for a.e.\ $\rho \in (r,R)$ the restriction $v|_{\de B_\rho(x_0)}$ belongs to $H^1(\de B_\rho(x_0); \S^1)$ and $\nabla_{\de B_\rho(x_0)}(v|_{\de B_\rho(x_0)})(y) = \nabla v(y) \tau_{\de B_\rho(x_0)}(y)$ for $\H^1$-a.e.\ $y \in \de B_\rho(x_0)$, where $\tau_{\de B_\rho(x_0)}(y)$ is the unit tangent vector to $\de B_\rho(x_0)$ at $y$. Therefore
 \begin{equation} \label{eq:degree for H1}
     \deg(v, \de B_\rho(x_0)) = \frac{1}{\pi} \int_{\de B_\rho(x_0)} j(v)|_{\de B_\rho(x_0)} \cdot \tau_{\de B_\rho(x_0)} \, \d \H^1  \quad \text{for a.e.\ } \rho \in (r,R) \, ,
 \end{equation}
 which relates the degree to the pre-jacobian and, by Stokes' Theorem, to the distributional Jacobian.

 Finally, we recall the following lemma proven in~\cite[Lemma 3.1]{ACP11}, useful when dealing with the flat convergence of Jacobian of maps.
 \begin{lemma}\label{lem:conv-diff-jac}
 Let $U\subset\R^2$ be open and let $(v_\e),(\tilde{v}_\e)\subset H^1(U;\R^2)$ be two sequences satisfying
   \begin{equation}\label{cond:lemma-ACP}
   \|v_\e-\tilde{v}_\e\|_{L^2(U;\R^2)}\big(\|\nabla v_\e\|_{L^2(U;\R^{2\times 2})}+\|\nabla\tilde{v}_\e\|_{L^2(U;\R^{2\times 2})}\big)\to 0\,.
   \end{equation}
 Then $ J(v_\e)-J(\tilde{v}_\e) \fto 0$.
 \end{lemma}

\subsection{$SBV$ functions}
Given an open set $U\subset\R^2$ the space of {\em special functions of bounded variation} $w:U\to\R^m$ with $m\geq 1$ is denoted by $SBV(U;\R^m)$. If $m=1$ we simply write $SBV(U)$. We refer to~\cite{AFP} for the definition and main properties of this function space. Here we just recall that for every $w\in SBV(U;\R^m)$ its distributional derivative $Dw$ is a bounded radon measure that can be represented as
	\begin{equation}\label{eq:representation-Dw}
	Dw=\nabla w\L^2+[w]\otimes\nu_w\H^1\mres S_w\,.
	\end{equation}
In~\eqref{eq:representation-Dw} $\nabla w$ is the approximate gradient of $w$, $S_w$ is the set of discontinuity points of $w$ and $\nu_w$ the measure theoretic normal to $S_w$. Finally, $[w]=w^+-w^-$ where $w^+$ and $w^-$ are the one-sided approximate limits of $w$ on $S_w$, which exist up to an $\H^1$-negligible set. We also recall that $W^{1,1}(U;\R^m)\subset SBV(U;\R^m)$ and for any $w\in SBV(U;\R^m)$ we have $w\in W^{1,1}(U;\R^m)$ if and only if $\H^1(S_w)=0$ (\cf~\cite[Formula (4.2)]{AFP}).

Moreover, for any $p>1$ we consider the subspace
	\begin{equation*}
	SBV^p(U;\R^m)\defas\big\{w\in SBV(U;\R^m)\colon \nabla w\in L^p(U;\R^{m\times 2})\,,\ \H^1(S_w)<+\infty\big\}\,.
	\end{equation*}
Eventually, we set $SBV(U;\S^1)\defas\{w\in SBV(U;\R^2)\colon |w(x)|=1\ \text{for a.e.}\ x\in U\}$ and we define $SBV^p(U;\S^1)$ accordingly.

We say that a sequence $(w_k)\subset SBV^p(U;\R^m)$ converges weakly$^*$ to $w\in SBV^p(U;\R^m)$ and write $w_k\wsto w$ in $SBV^p(U;\R^m)$, if $w_k\to w$ in $L^1(U;\R^m)$ and $\sup_k\big(\|\nabla u_k\|_{L^p(U)}+\H^1(S_{w_k}\cap U)+\|u_k\|_{L^\infty(U)}\big)<+\infty$. Moreover, $w_k\wsto w$ in $SBV^p_\loc(U;\R^m)$, if $w_k\wsto w$ in $SBV^p(U';\R^m)$ for every $U'\wcont U$.
%%%%%%%%%%%%%%%%%%%%%%%%%%%%%%%%%%%%%%%
%          Limit Fields
%%%%%%%%%%%%%%%%%%%%%%%%%%%%%%%%%%%%%%%
\subsection{The spaces of limiting fields, renormalised energy, and core energy}\label{sec:limit-fields}
We are now able to introduce the spaces of limiting fields and characterise the renormalised energy $\W(\mu,\Omega)$.
From now on, if not specified otherwise, $\Omega\subset\R^2$ is an open, bounded, and simply connected subset of $\R^2$ with Lipschitz boundary and  $M\in\N$ is a fixed positive integer. For such $\Omega$ we set
	\begin{equation*}\label{def:DM}
	\D_M(\Omega)\defas\Big\{v\in W^{1,1}(\Omega;\S^1)\colon J(v)=\pi\mu\;\text{ for some }\; \mu\in X_M(\Omega)\; \text{ and }\; v\in H^1_\loc\big(\Omega\sm\supp \mu;\S^1\big)\Big\}\,.
	\end{equation*}
%Let $v\in D_M(\Omega)$ with $J(v)=\pi\mu$ and $\mu=\sum_{h=1}^Md_h\delta_{x_h}\in X_M(\Omega)$ and let $\sigma>0$ be sufficiently small such that
%	\begin{equation}\label{cond:sigma}
%	B_\sigma (x_h)\subset \Omega\quad\text{and}\quad B_\sigma(x_h)\cap B_\sigma(x_h')=\emptyset\; \text{ for}\ h,h'\in\{1,\ldots,M\}\,,\ h\neq h'\,.
%	\end{equation}
%From~\eqref{eq:degree for H1} together with Stokes' Theorem we deduce that for almost all such $\sigma$ we have
%	\begin{equation}\label{eq:stokes}
%	\pi d_h=J(v)\big(B_\sigma(x_h)\big)=\int_{B_\sigma(x_h)}\hspace*{-1.5em}\curl j(v)\dx=\int_{\partial B_\sigma(x_h)}\hspace*{-1.5em}j(v)\cdot\tau_{\partial B_\sigma(x_h)}\dH=\deg\big(v,\partial B_\sigma(x_h)\big)\,.
%	\end{equation}
For any $\mu=\sum_{h=1}^M d_h\delta_{x_h}\in X_M(\Omega)$ and $\sigma>0$ sufficiently small such that 
	\begin{equation}\label{cond:sigma}
	B_\sigma (x_h)\subset \Omega\quad\text{and}\quad B_\sigma(x_h)\cap B_\sigma(x_h')=\emptyset\; \text{ for}\ h,h'\in\{1,\ldots,M\}\,,\ h\neq h' 
	\end{equation}
	we set 
	\begin{equation}\label{def:Omega-sigma}
	\Omega^\sigma(\mu)\defas \Omega\sm\bigcup_{h=1}^M B_\sigma(x_h)\,,
	\end{equation} 
and to any $v\in\D_M(\Omega)$ with $J(v)=\pi\mu$ we associate the quantity
	\begin{equation}\label{def:ren-energy-v}
	\WW(v,\Omega)\defas\lim_{\sigma\to 0}\bigg(\frac{1}{2}\int_{\Omega^\sigma(\mu)}\hspace*{-1em}|\nabla v|^2\dx-M\pi|\log\sigma|\bigg)\in\R\cup\{+\infty\}\,,
	\end{equation}
which is well defined thanks to Remark~\ref{rem:DM-W} below.
\begin{remark}\label{rem:DM-W}
Let $v\in\D_M(\Omega)$ with $J(v)=\pi\mu$ and let $\sigma_2>\sigma_1>0$ satisfy~\eqref{cond:sigma}. 
%Then $v\in H^1(A_{\sigma_1,\sigma_2}(x_h);\S^1)$ for every $h\in\{1,\ldots,M\}$ and $v_{|\partial B_\sigma}\in H^1(\partial B_\sigma;\S^1)$ for a.e.\ $\sigma\in (\sigma_1,\sigma_2)$. 
Then~\eqref{eq:degree for H1} together with Stokes' Theorem yields for a.e.\ $\sigma\in(\sigma_1,\sigma_2)$
	\begin{equation}\label{eq:stokes}
	\pi d_h=J(v)\big(B_\sigma(x_h)\big)=\int_{B_\sigma(x_h)}\hspace*{-1.5em}\curl j(v)\dx=\int_{\partial B_\sigma(x_h)}\hspace*{-1.5em}j(v)\cdot\tau_{\partial B_\sigma(x_h)}\dH=\pi\deg\big(v,\partial B_\sigma(x_h)\big)\,.
	\end{equation}
Note that $2|j(v)|=|\nabla v|$, so that from~\eqref{eq:stokes} together with Jensen's inequality we deduce that
	\begin{equation}\label{est:annulus-jensen}
	\begin{split}
	\frac{1}{2}\int_{A_{\sigma_1,\sigma_2}(x_h)}\hspace*{-3em}|\nabla v|^2\dx &=2\int_{\sigma_1}^{\sigma_2}\int_{\partial B_\sigma(x_h)}\hspace*{-1.5em}|j(v)|^2\dH\ds\\
	&\geq\frac{1}{\pi}\int_{\sigma_1}^{\sigma_2}\frac{1}{\sigma}\bigg(\int_{\partial B_\sigma(x_h)}\hspace*{-1.5em}j(v)\cdot\tau_{\partial B_\sigma(x_h)}\dH\bigg)^2\ds\geq\pi\log\frac{\sigma_2}{\sigma_1}\,.
	\end{split}
	\end{equation}		
%	\begin{equation}\label{eq:degree-constant}
%	\deg\big(v,\partial B_\sigma(x_h)\big)=\deg\big(v,\partial B_{\sigma_1}(x_h)\big)\;\text{ for every}\ \sigma\in[\sigma_1,\sigma_2]\,,
%	\end{equation}
%where $\deg\big(v,\partial B_\sigma(x_h)\big)$ is the degree of $v$ on $\partial B_\sigma(x_h)$. 
This in turn implies that
	\begin{equation}\label{est:monotonicity-W}
	\begin{split}
	\frac{1}{2}\int_{\Omega^{\sigma_1}(\mu)}\hspace*{-1.5em}|\nabla v|^2\dx-M\pi|\log\sigma_1| &=\frac{1}{2}\int_{\Omega^{\sigma_2}(\mu)}\hspace*{-1.5em}|\nabla v|^2\dx-M\pi|\log\sigma_2|
	+\sum_{h=1}^M\bigg(\frac{1}{2}\int_{A_{\sigma_1,\sigma_2}(x_h)}\hspace*{-2.5em}|\nabla v|^2\dx-\pi\log\frac{\sigma_2}{\sigma_1}\bigg)\\
	&\geq\frac{1}{2}\int_{\Omega^{\sigma_2}(\mu)}\hspace*{-1.5em}|\nabla v|^2\dx-M\pi|\log\sigma_2|\,.
	\end{split}
	\end{equation}
In particular, the map $\sigma\mapsto \frac{1}{2}\int_{\Omega^\sigma(\mu)}|\nabla v|^2\dx-M\pi|\log\sigma|$ is decreasing and thus $\WW(v,\Omega)$ is well-defined (see also~\cite[Section 4.4]{ADGP14}). Suppose now that $\WW(v,\Omega)<+\infty$; applying~\eqref{est:monotonicity-W} with $\sigma_2=\sigma$, $\sigma_1=\frac{\sigma}{2}$ for $\sigma>0$ sufficiently small yields 
	\begin{align*}
	\WW(v,\Omega) &\geq\frac{1}{2}\int_{\Omega^\sigma(\mu)}|\nabla v|^2\dx-\pi|\log\sigma|+\sum_{h=1}^M\bigg(\frac{1}{2}\int_{A_{\frac{\sigma}{2},\sigma}(x_h)}\hspace*{-2em}|\nabla v|^2\dx-\pi\log 2\bigg)
	\geq\WW(v,\Omega)-r(\sigma)
	\end{align*}
with $r(\sigma)\to 0$ as $\sigma\to 0$. This in turn implies that
	\begin{equation}\label{cond:limit-dyadic-annulus}
	\lim_{\sigma\to 0}\frac{1}{2}\int_{A_{\frac{\sigma}{2},\sigma}(x_h)} |\nabla v|^2\dx=\pi\log 2\;\text{ for every}\ h\in\{1,\ldots,M\}\,.
	\end{equation}
Since the functional $\frac{1}{2}\int_{A_{\frac{\sigma}{2},\sigma}(x_h)}|\nabla v|^2\dx$ attains its minimum value $\pi\log2$, among all functions $v\in H^1_\loc(B_\sigma(x_h)\sm\{x_h\})$ with $J(v)\mres B_\sigma(x_h)=d_h\delta_{x_h}$, precisely on rotations of the function $\big(\frac{x-x_h}{|x-x_h|}\big)^{d_h}$, we deduce from~\eqref{cond:limit-dyadic-annulus} together with the continuity of the lifting operator in $H^1$ that
	\begin{equation}\label{limit-dyadic-lifting}
	\lim_{\sigma\to 0}\int_{A_{\frac{\sigma}{2},\sigma}(x_h)} |\nabla\varphi(x)-d_h\nabla\theta(x-x_h)|^2\dx=0
	\end{equation}
for every local lifting $\varphi$ of $v$ and any lifting $\theta$ of $\frac{x}{|x|}$.
\end{remark}
We finally recall the definition of the {\em renormalised energy} $\W(\mu,\Omega)$ introduced in~\cite{BBH}. To this end, let $\Phi_\mu$ be a solution to the boundary value problem
	\begin{equation}\label{def:Phi-mu}
	\begin{cases}
	\Delta\Phi_\mu=2\pi\mu &\text{in}\ \Omega\,,\\
	\Phi_\mu=0 &\text{on}\ \partial\Omega
	\end{cases}
	\end{equation}
and set $R_\mu(x)\defas\Phi_\mu(x)-\sum_{h=1}^Md_h\log(|x-x_h|)$. Then the renormalised energy associated to $\mu$ and $\Omega$ is given by
	\begin{equation}\label{def:ren-energy-Phi-mu}
	\W(\mu,\Omega)\defas-\pi\sum_{h\neq h'}d_hd_{h'}\log(|x_h-x_{h'}|)-\pi\sum_{h=1}^MR_\mu(x_h)\,.
	\end{equation}
We observe that 
	\begin{equation}\label{eq:ren-energy-ball}
	\W(d \delta_0, B_\sigma(0)) = - d |\log \sigma| \, .
	\end{equation}
As shown in~\cite[Theorem 1.7]{BBH}, $\W(\mu,\Omega)$ can be characterised as the limit
	\begin{equation}\label{eq:ren-energy-limit-min-problem}
	\W(\mu,\Omega)=\lim_{\sigma\to 0}\big(\m_\sigma(\mu,\Omega)-M\pi|\log\sigma|\big)\,,
	\end{equation}
where
	\begin{equation}\label{def:min-degree}
	\m_\sigma(\mu,\Omega)\defas\min\bigg\{\frac{1}{2}\int_{\Omega^\sigma(\mu)}\hspace*{-1em}|\nabla w|^2\dx\colon w\in H^1\big(\Omega^\sigma(\mu);\S^1\big)\,,\ \deg(v,\partial B_\sigma(x_h))=d_h\bigg\}\,.
	\end{equation}
Note that thanks to~\eqref{eq:stokes} we have $\m_\sigma(\mu,\Omega)\leq\frac{1}{2}\int_{\Omega^\sigma(\mu)}|\nabla v|^2\dx$ for every $v\in\D_M(\Omega)$ with $J(v)=\pi\mu$, so that in particular
	\begin{equation}\label{est:W-W-ren}
	\W(\mu,\Omega)\leq\WW(v,\Omega)\; \text{ for every $v\in\D_M(\Omega)$ with $J(v)=\pi\mu$.}
	\end{equation}
Moreover, to any $\mu\in X_M(\Omega)$ we can associate the canonical harmonic map $v_\mu$ via the relation $j(v_\mu)=\frac{1}{2}\nabla^\perp\Phi_\mu$, where $\Phi_\mu$ is given by~\eqref{def:Phi-mu}.
Then $v_\mu\in \D_{M}(\Omega)$ and moreover (\cf~\cite[Section 4.1]{ADGP14}) there holds
	\begin{equation*}
	\W(\mu,\Omega)=\lim_{\sigma\to 0}\bigg(\frac{1}{2}\int_{\Omega^\sigma(\mu)}\hspace*{-1em}|\nabla \Phi_\mu|^2\dx-M\pi|\log\sigma|\bigg)=\lim_{\sigma\to 0}\bigg(\frac{1}{2}\int_{\Omega^\sigma(\mu)}\hspace*{-1em}|\nabla v_\mu|^2\dx-M\pi|\log\sigma|\bigg)\,.
	\end{equation*}
Together with~\eqref{est:W-W-ren} this gives
	\begin{equation}\label{eq:W-can-harm}
	\W(\mu,\Omega)=\min\big\{\WW(v,\Omega)\colon v\in \D_M(\Omega)\,,\ J(v)=\pi\mu\big\}=\WW(v_\mu,\Omega)\,.
	\end{equation}
Eventually, for any $x_0\in\R^2$, $\e>0$, and $\sigma>4\e$ we set
	\begin{equation}\label{def:core-screw}
	\gamma_\e^\screw\big(B_\sigma(x_0)\big)\defas\min\Big\{F_\e^\screw\big(u,B_\sigma(x_0)\big)\colon   u(i)= \frac{1}{2\pi}\theta(i-x_0)\;\text{ for }\;i\in\partial_\e B_\sigma(x_0)\Big\}\,,
	\end{equation}
where $\theta$ is an angular lifting of $\frac{x}{|x|}$, \ie such that $\frac{x}{|x|}=\exp(\iota\theta(x))$ for every $x\in\R^2\sm\{0\}$. 
\begin{remark}\label{rem:core-screw}
Clearly the value $\gamma_\e^\screw\big(B_\sigma(x_0)\big)$ does not depend on the choice of the angular lifting of $\frac{x}{|x|}$ and we could have equivalently required in~\eqref{def:core-screw} that $\exp(2\pi u(i))=\frac{i-x_0}{|i-x_0|}$ on $\partial_\e B_\sigma(x_0)$. For later purpose it will however be more convenient to prescribe the boundary condition in the angular variable as in~\eqref{def:core-screw}.
We recall that the limit
	\begin{equation}\label{def:gamma}
	\gamma\defas\lim_{\e\to 0}\Big(\gamma_\e^\screw\big(B_\sigma(x_0)\big)-\pi\log\frac{\sigma}{\e}\Big)
	\end{equation}
exists and is independent of $x_0$ (see~\cite[Theorem 4.1]{ADGP14} and~\cite[Lemma 7.2]{COR22ARMA}).
\end{remark}
%%%%%%%%%%%%%%%%%%%%%%%%%%%%%%%%%%%%%%%%%%%%%%%%%%%%%%
%               Interpolation
%%%%%%%%%%%%%%%%%%%%%%%%%%%%%%%%%%%%%%%%%%%%%%%%%%%%%
\subsection{Interpolation of discrete functions}\label{sec:interpolation}
We conclude this section by introducing useful interpolations of discrete functions.

\subsubsection*{Piecewise constant interpolations}
Throughout the paper, discrete functions $u_\e \colon \e \Z^2 \to \R^m$ will be tacitly identified with their piecewise constant interpolations taking values $u_\e(i)$ on every cube $Q_\e(i) \in \Q_\e$.   

\subsubsection*{Piecewise affine interpolations}
 For $\e>0$ let $\T_\e$ be the triangulation defined in~\eqref{triangulation}; for any $v_\e\in\SF_\e$ we let $\hat v_\e$ denote the function satisfying $\hat v_\e(i)=v_\e(i)$ for every $i\in\e\Z^2$ and being affine on every triangle $T_\e^+,T_\e^-\in\T_\e$. 
In this way, we have on every cube $Q_\e\in\Q_\e$ the identity
	\begin{equation}\label{eq:gradient-affine}
	XY_\e(v_\e,Q_\e)=\int_{Q_\e}|\nabla \hat v_\e|^2\dx\,.
	\end{equation}
Note that $\hat{v}_\e$ belongs to $H^1_\loc(\R^2;\R^2)$, but it may take values in $\R^2\sm\S^1$.

\subsubsection*{$\S^1$-valued interpolations}
Following the approach in~\cite[Remark 3.2]{BCKO20a} (see also~\cite[Remark 3.3]{ACP11}), to any $v_\e\in\SF_\e$ we associate an $\S^1$-valued interpolation satisfying the properties of Lemma~\ref{lem:S1-interpolation} below. Since the proof of Lemma~\ref{lem:S1-interpolation} is completely analogous to~\cite[Remark 3.2]{BCKO20a}, we do not repeat it here, but just state the result. 
%{\BBB [[Decide if we want to keep the lemma and the remark in detail or just refer to~\cite{BCKO20a}]]} \RRR [[If we keep it, we need to change the letters in the proof. However, I think we can refer to~\cite{BCKO20a} for the proof and just leave the statement. ]] \EEE
%
\begin{lemma}[$\S^1$-valued interpolation]\label{lem:S1-interpolation}
Let $\e>0$, let $v_\e\in\SF_\e$, and let $2\pi u_\e$ with $u_\e\in\AD_\e$ be an arbitrary angular lifting of $v_\e$. Then there exists $\overarc{v}_\e\in W^{1,1}_\loc(\R^2;\S^1)\cap W^{1,\infty}_\loc(\R^2\sm\supp\mu_{u_\e};\S^1)$ satisfying the following properties:  
\begin{enumerate}[label=(\arabic*)]
\item \label{prop:interpolation} $\overarc{v}_\e(i)=v_\e(i)=\exp\big(2\pi\iota u_\e(i)\big)$ for any $i\in\e\Z^2$;
%\item $\overarc{v}\in W^{1,\infty}_\loc(\R^2\sm\supp\mu_u;\S^1)$;
\item \label{prop:jacobian} $J(\overarc{v}_\e)=\pi\mu_{u_\e}$; 
\item \label{prop:energ} $\int_{Q_\e}|\nabla\overarc{v}_\e|^2\dx=F_\e^\screw(u_\e,Q_\e)$ whenever $\mu_{u_\e}(Q_\e)=0$.
\end{enumerate}
\end{lemma}
\begin{remark}[Lifting of $\overarc{v}_\e$]\label{rem:lifting-S1-interpolation}
Let $\e>0$, $u_\e\in\AD_\e$, $v_\e=\exp(2\pi\iota u_\e)\in\SF_\e$, and let $\overarc{v}_\e$ be the $\S^1$-valued interpolation of $v_\e$ defined in Lemma~\ref{lem:S1-interpolation}. Suppose that $U\subset\R^2$ is an open, bounded and simply connected set with $\supp \mu_{u_\e}\cap U=\emptyset$. Since $\overarc{v}_\e\in W^{1,\infty}_\loc(U;\S^1)$ and $J(\overarc{v}_\e)=\pi\mu_{u_\e}$, we deduce that $\overarc{v}_\e$ admits a lifting $2\pi\hat{\xi}_\e \in W^{1,\infty}_\loc(U)$ satisfying $\overarc{v}_\e(x)=\exp\big(2\pi\iota \hat{\xi}_\e(x) \big)$ for every $x\in U$ and 
	\begin{equation*}
	2\pi|\nabla\hat{\xi}_\e(x) |=|\nabla\overarc{v}_\e(x)|\;\text{ for a.e.}\ x\in U\,.
	\end{equation*} 
%Clearly, \ref{prop:interpolation} implies that 
Denoting by $\xi_\e \defas{\hat{\xi}_\e }_{|\e\Z^2}$ the evaluation of $\hat{\xi}_\e$ on $\e\Z^2$, we deduce from Property~\ref{prop:interpolation} in Lemma~\ref{lem:S1-interpolation} that $\xi_\e \eqZ u_\e$. Let moreover $k\in\{1,2\}$ be fixed. As pointed out in~\cite[Remark 3.4]{BCKO20a} the equality $2\pi|\xi_\e (i+\e e_k)- \xi_\e(i)|=\d_{\S^1}\big(\overarc{v}_\e(i),\overarc{v}_\e(i+\e e_k)\big)$ holds for every $i\in\Z_\e ^{e_k}(U)$. Together wit~\eqref{eq:geo-dist} this implies that  
	\begin{equation}\label{eq:difference-phi}
	2\pi|\d \xi_\e (i,i+\e e_k)|=2\pi\dist\big(\d u_\e(i,i+\e e_k);\Z\big)
	\end{equation}
for every $i\in\Z_\e^{e_k}(U)$.
Eventually, as in~\cite[Remark 3.4]{BCKO20a} we deduce that $\hat{\xi}_\e$ is affine on any triangle $T\in\T_\e$ provided $T\subset U$ (note that $\supp \mu_{\xi_\e} \cap T = \emptyset$).  
\end{remark}

% \RRR [[End subsubsection here, remove these interpolations]]\EEE
% When dealing with the energies $F_\e^\pscrew$ and $F_\e^\p$ it is convenient to introduce another class of interpolations as done in~\cite{BCDP18}. Let $u_\e\in\AD_\e$ and $w_\e=\exp(2\pi\iota u_\e)\in\SF_\e$; we say that a pair $(i,j)$ with $i,j\in\e\Z^2$ and $|i-j|=\e$ is a jump pair for $u$ and write $(i,j)\in\J(u_\e)$, if
% 	\begin{equation} \label{def:jump pair}
% 	\dist\big(\d u_\e(i,j);\Z\big)\geq\frac{1}{4}\,.
% 	\end{equation}
% We say that a cube $Q_\e(i)\in\Q_\e$ with vertices $i,j,k,\ell$ is a jump cell, if at least one of the bonds $(i,j),(j,k),(k,\ell),(\ell,i)$ is a jump pair. The collection of jump cells is denoted by $\J_{\Q_\e}(u_\e)$. We now define $\overline{w}_\e:\R^2\to\R^2$ by setting 
% 	\begin{equation}\label{def:interpolation-w}
% 	\overline{w}_\e(x)\defas 
% 	\begin{cases}
% 	w_\e(i)&\text{if}\ x\in\inter{Q}_\e(i)\ \text{for some}\ Q_\e(i)\in\J_{\Q_\e}(u_\e)\,,\\
% 	\hat{w}_\e(x) &\text{otherwise in}\ \R^2\,,
% 	\end{cases}
% 	\end{equation}
% so that $\overline{w}_\e\in SBV^2_\loc(\R^2;\R^2)$.
% %In this way, $\overline{w}_\e\in SBV^2(\Omega;\R^2)$.
% %and
% %	\begin{equation}\label{eq:inclusion-jumpset}
% %	S_{\overline{w}_\e}\cap\Omega\subset\bigcup_{Q_\e\in\J_{\Q_\e}(u)}\partial Q_\e\cap\Omega\,.
% %	\end{equation}
We will also consider the corresponding interpolations of $v_\e$ and $w_\e$ on the double-spaced and shifted lattices $2\e\Z^2+s_j$ as in~\eqref{def:double-lattice-shifted}. Namely, for $j\in\{0\,,\ldots ,3\}$ we let   $\overarc{v}_{2\e,s_j}$, $\hat{v}_{2\e,s_j}$, $\hat{w}_{2\e,s_j}$, and $w_{2\e,s_j}$ be as above with $\Q_\e$, $\T_\e$ replaced by $\Q_{2\e,s_j}$, $\T_{2\e,s_j}$, respectively.

%%%%%%%%%%%%%%%%%%%%%%%%%%%%%%%%%%%%%%%%%%%%%%%%%%%%%%
%             Convergence Edge
%%%%%%%%%%%%%%%%%%%%%%%%%%%%%%%%%%%%%%%%%%%%%%%%%%%%%%%
\section{$\Gamma$-limit without partial dislocations: Convergence of $F_\e^\edge$}\label{Section:edge}
Before proving the convergence result for the energies $F_\e^\p$ stated in Theorem~\ref{thm:main} we characterise here the zero-order and the first-order $\Gamma$-limit of $F_\e^\edge$. Thanks to the comparison Lemma~\ref{lem:SD-equals-F}, we obtain the following $\Gamma$-convergence results for $F_\e^\edge$ based on the corresponding convergence results for $F_\e^\screw$ established in~\cite{ADGP14}. 
\begin{theorem}\label{thm:zero-order}
Let $\Omega\subset\R^2$ be open, bounded, and with Lipschitz boundary and let $F_\e^\edge$ be defined according to~\eqref{def:F-edge}. Then the following $\Gamma$-convergence result holds true:
\begin{enumerate}[label=(\roman*)]
\item (Compactness) Let $u_\e\in\AD_\e$ be such that 
	\begin{equation}\label{uniformbound:edge-zero}
	\sup_{\e>0}\frac{1}{|\log\e|}F_\e^\edge(u_\e,\Omega)<+\infty
	\end{equation}
and let $\mu_{u_\e}$ be given by~\eqref{def:vorticity-u}. Then, up to subsequences, $\mu_{u_\e}\mres\Omega\fto \mu$ for some $\mu\in X(\Omega)$.
\item (Lower bound) Let $u_\e\in\AD_\e$ and assume that $\mu_{u_\e}\mres\Omega\fto\mu=\sum_{h=1}^Nd_h\delta_{x_h}\in X(\Omega)$. Then there exists a constant $\overline{C}\in\R$ such that for all $\sigma>0$ satisfiyng~\eqref{cond:sigma}	and for all $h\in\{1,\ldots,N\}$ we have
	\begin{equation}\label{lb-local}
	\liminf_{\e\to 0}\Big(F_\e^\edge(u_\e,B_\sigma(x_h))-\pi|d_h|\log\frac{\sigma}{\e}\Big)\geq \overline{C}\,.
	\end{equation}
In particular, 
	\begin{equation*}
	\liminf_{\e\to 0}\frac{1}{|\log\e|}F_\e^\edge(u_\e,\Omega)\geq \pi|\mu|(\Omega)\,.
	\end{equation*}
\item (Upper bound) For every $\mu\in X(\Omega)$ there exist $u_\e\in\AD_\e$ such that $\mu_{u_\e}\mres\Omega\fto\mu$ satisfying
\begin{equation*}
	\limsup_{\e\to 0}\frac{1}{|\log\e|}F_\e^\edge(u_\e,\Omega)\leq \pi|\mu|(\Omega)\,.
\end{equation*}
\end{enumerate}
\end{theorem}
\begin{proof}
The compactness result (i) follows immediately from~\eqref{lb:edge-SD} and the corresponding result\footnote{A technical clarification is in order. We note that~\cite[Theorem 3.1]{ADGP14} actually provides compactness with respect to the flat convergence of the measures $\tilde{\mu}_{u_\e}\defas\sum_{Q\in\Q_\e(\Omega)}\mu_{u_\e}(Q)\delta_{b(Q)}$ with $\mu_{u_\e}(Q)$ as in~\eqref{def:circulation}. Since $\Omega$ has Lipschitz boundary, the measures $\mu_{u_\e}$ and $\tilde\mu_{u_\e}$ are equivalent along sequences $(u_\e)$ with $\sup_{\e>0}\frac{1}{|\log\e|}F_\e^\screw(u_\e,\Omega)<+\infty$ (and thus in particular along sequences satisfying~\eqref{uniformbound:edge-zero}) in the sense that for such sequences we have $(\tilde{\mu}_{u_\e}-\mu_{u_\e}\mres\Omega)\fto 0$ as $\e\to 0$ (see~\cite[Lemma 2.3]{COR22ARMA}). \label{foot:conv-mu}} 
for $F_\e^\screw$ in~\cite[Theorem 3.1(i)]{ADGP14}. Similarly, the lower bound (ii) follows from~\cite[Theorem 3.1(ii)]{ADGP14}. To prove the upper bound (iii) let $\mu\in X(\Omega)$ and consider a recovery sequence $(u_\e)$ satisfying $\mu_{u_\e}\mres\Omega\fto\mu$ and $\limsup_\e\frac{1}{|\log\e|}F_\e^\screw(u_\e,\Omega)\leq\pi|\mu|(\Omega)$. Such a sequence exists thanks to~\cite[Theorem 3.1 (iii)]{ADGP14} and again Footnote~\ref{foot:conv-mu}. Applying Lemma~\ref{lem:SD-equals-F} to the sequence $(u_\e)$ and $U=\Omega$ we obtain a sequence $(\tilde{u}_\e)$ satisfying
	\begin{equation*}
	\limsup_{\e\to 0}\frac{1}{|\log\e|}F_\e^\edge(\tilde{u}_\e,\Omega)=\limsup_{\e\to 0}\frac{1}{|\log\e|}F_\e^\screw(u_\e,\Omega)\leq\pi|\mu|(\Omega)\,.
	\end{equation*}
Moreover, in view of Remark~\ref{rem:mu-equals-mutilde} we have $\mu_{\tilde{u}_\e}\mres\Omega\fto \mu$, which concludes the proof of (iii).
\end{proof}

In a similar way we obtain the following result on the asymptotic behaviour of $F_\e^\edge$ after removing the logarithmic contribution of $M$ dislocations. We refer to Subsection~\ref{sec:limit-fields} for the definition and the properties of $\W$ and $\gamma$.
\begin{theorem}\label{thm:first-order-Fe}
Let $\Omega\subset\R^2$ be open, bounded, and with Lipschitz boundary, let $F_\e^\edge$ be as in~\eqref{def:F-edge}, and let $M\in\N$ be fixed. Then the following holds:
\begin{enumerate}[label=(\roman*)]
\item (Compactness) Let $u_\e\in\AD_\e$ be such that 
	\begin{equation}\label{uniformbound-F-edge}
	\sup_{\e>0}\Big(F_\e^\edge(u_\e,\Omega)-M\pi|\log\e|\Big)<+\infty\,.
	\end{equation}
Then, up to subsequences $\mu_{u_\e}\mres\Omega\fto\mu$ for some $\mu=\sum_{h=1}^Nd_h\delta_{x_h}\in X(\Omega)$ with $|\mu|(\Omega)\leq M$. Moreover, if $|\mu|(\Omega)=M$, then $|d_h|=1$ for every $h\in\{1\,,\ldots,N\}$, hence $\mu\in X_M(\Omega)$. 
In this case, if $v_\e\defas\exp(2\pi\iota u_\e)$ then (up to passing to a further subsequence) $\hat{v}_\e\wto v$ in $H^1_\loc\big(\Omega\sm\supp\mu;\R^2\big)$ for some $v\in\D_M(\Omega)$ with $J(v)=\pi\mu$.
\item (Lower bound) Let $u_\e\in\AD_\e$ be such that $\mu_{u_\e}\mres\Omega\fto\mu$ for some $\mu\in X_M(\Omega)$. Then
	\begin{equation*}
	\liminf_{\e\to 0}\big( F_\e^\edge(u_\e,\Omega)-M\pi|\log\e|\big)\geq \W(\mu,\Omega)+M\gamma\,.
	\end{equation*}
\item (Upper bound) For every $\mu\in X_M(\Omega)$ there exists $u_\e\in\AD_\e$ with $\mu_{u_\e}\mres\Omega\fto\mu$ and satisfying
	\begin{equation*}
	\limsup_{\e\to 0}\big( F_\e^\edge(u_\e,\Omega)-M\pi|\log\e|\big)\leq \W(\mu,\Omega)+M\gamma\,.
	\end{equation*}
\end{enumerate} 
\end{theorem}
\begin{proof}
The proof is analogous to the one of Theorem~\ref{thm:zero-order}, using~\eqref{lb:edge-SD} and~\cite[Theorem 4.2]{ADGP14} to obtain (i) and (iii) and applying Lemma~\ref{lem:SD-equals-F} to a recovery sequence provided by~\cite[Theorem 4.2 (iii)]{ADGP14}.
\end{proof}
\begin{remark}\label{rem:core}
Applying Lemma~\ref{lem:SD-equals-F} with $U=B_\sigma(x_h)$ we see that the quantity $\gamma$ can be equivalently characterised via
	\begin{equation}\label{eq:gamma-edge}
	\gamma=\lim_{\e\to 0}\Big(\gamma_\e^\edge\big(B_\sigma(x_0)\big)-\pi\log\frac{\sigma}{\e}\Big)
	\end{equation}
with 
	\begin{equation}\label{def:core-edge}
	\gamma_\e^\edge\big(B_\sigma(x_0)\big)\defas\min\Big\{F_\e^\edge\big(u,B_\sigma(x_0)\big)\colon 2\pi u(i)=\theta(i-x_0)\;\text{ for }\;i\in\partial_\e B_\sigma(x_0)\Big\}\,,
	\end{equation}
where now the suitable choice of a lifting $\theta$ is explicitly part of the minimisation problem.
\end{remark}

  \section{Domain of the $\Gamma$-limit of $F_\e^\p$}\label{sec:limit-configurations}\label{sec:properties}

  We consider the family of functions
    \begin{equation*}\label{def:DM2} 
      \begin{split}
        \D_M^{1/2}(\Omega)\defas \Big\{ w\in SBV(\Omega;\S^1) \colon & w^2\in\D_M(\Omega) \,,\ J(w^2)=\pi\mu \; \text{ for some }\; \mu\in X_M(\Omega) \,,\ \\
        &   w\in SBV^2_\loc(\Omega\sm\supp\mu;\S^1) \,,\ \H^1(S_w\cap \Omega)<+\infty  \Big\}\,,
      \end{split}
    \end{equation*} 
  where $w^2$ is the complex square. We shall see that the domain of the $\Gamma$-limit of $F_\e^\p$ is the family of functions
  \begin{equation*}\label{def:DM2-hor}
  \D_{M,\hor}^{1/2}(\Omega) \defas\big\{w\in\D_M^{1/2}(\Omega)\colon |\nu_w\cdot e_1|=0\ \H^1\text{-a.e.\ on}\ S_w\cap\Omega\big\}\,.
  \end{equation*}
In this section we collect some structure properties of limiting configurations $v\in\D_M(\Omega)$ and $w\in\D_{M,\hor}^{1/2}(\Omega)$.

  Before delving into the structure of maps in $\D_{M,\hor}^{1/2}(\Omega)$, to each $\mu \in X_M(\Omega)$ we associate a class of horizontal segments. This class will be relevant in the sequel since it will contain discontinuity sets of maps in $\D_{M,\hor}^{1/2}(\Omega)$. 

  \begin{definition} \label{def:connect singularities}
    Let $\mu = \sum_{h=1}^Md_h\delta_{x_h}\in X_M(\Omega)$. We say that an open segment $(y_1,y_2) \subset \Omega$ is an \emph{indecomposable stacking fault} if $(y_1,y_2) \cap \supp \mu = \emptyset$, if it is horizontal, \ie $y_1 \cdot e_2 = y_2 \cdot e_2$, and if one of the following conditions is satisfied:
    \begin{itemize} 
      \item either $y_1, y_2 \in \supp \mu$, \ie $(y_1,y_2)$ connects two singularities;
      \item or $y_1 \in \supp \mu$, $y_2 \in \de \Omega$ or $y_2 \in \supp \mu$, $y_1 \in \de \Omega$, \ie $(y_1,y_2)$ connects a singularity to the boundary; 
      \item or $y_1, y_2 \in \de \Omega$, \ie $(y_1,y_2)$ connects two boundary points.
    \end{itemize}
    We consider the class of \emph{stacking faults} 
    \[
      \begin{split}
    \Ten(\mu,\Omega) := \Big\{ S \subset \Omega \ : \ & S = \bigcup_{j \in \N} (y_j^1, y_j^2) \cup \NN \, , \ \H^1(\NN) = 0 \, , \ \H^1(S) < +\infty \, , \\
      & (y_j^1, y_j^2) \text{ is an \emph{indecomposable stacking fault} for every } j  \Big\} \, .  
      \end{split}
    \]
    Since the structure $S = \bigcup_{j \in \N} (y_j^1, y_j^2) \cup \NN \in \Ten(\mu,\Omega)$ is unique up to $\H^1$-negligible sets, we say that an \emph{indecomposable stacking fault} $(y_j^1, y_j^2)$ belongs to the decomposition of $S$.
    
    See Figures~\ref{fig:infinite tensions}--\ref{fig:OmegaGamma} for examples.
    \end{definition}

    \begin{remark} \label{rem:countable components}
      Given $S = \bigcup_{j \in \N} (y_j^1, y_j^2) \cup \NN \in \Ten(\mu,\Omega)$, we observe that the set of indices $j \in \N$ such that at least one of $y_j^1, y_j^2$ lies on $\supp \mu$ is finite. Hence, in Definition~\ref{def:connect singularities}, the infinitely many \emph{indecomposable stacking faults} in the decomposition of $S$ can only be those that connect two boundary points. In general, the finiteness of $\H^1(S)$ is not enough to conclude that the family of \emph{indecomposable stacking faults} is finite, see Figure~\ref{fig:infinite tensions}.

      However, for any $\Omega'\wcont\Omega$ we have upon relabeling
      \begin{equation}\label{covering:S-local}
      S \cap\Omega'=\bigcup_{j=1}^N (y_j^1, y_j^2) \cap\Omega'\cup\mathcal{N}
      \end{equation}
    for some $N=N(\Omega')\in\N$, \ie locally $S$ is the union of finitely many horizontal segments. Indeed, let $d\defas\dist(\Omega',\partial\Omega)>0$. Let moreover $j$ be such that $y_j^1, y_j^2 \in \de \Omega$ and $(y_j^1, y_j^2) \cap\Omega' \neq \emptyset$. Then $\H^1((y_j^1, y_j^2))\geq d$, and since $\H^1(S)<+\infty$, there can exist only finitely many segments of this form.
    \end{remark}

    \begin{figure}[ht]
      \centering
      % \begin{tikzpicture}
      %     \clip (-2.2,0) rectangle (2.2,2.2);
      %     \draw[line width = 1pt] (0,0) circle (2);

      %     \draw (0,0) node[anchor=south] {$\Omega$};
          
      %     \pgfmathsetmacro{\t}{2-1/(1)}
      %     \draw (-1.75, \t) node[anchor=east] {$t_j$};
      %     \pgfmathsetmacro{\t}{2-1/(2^2)}
      %     \draw (-1.2, \t) node[anchor=east] {$t_{j+1}$};
          
      %     \clip  (0,0) circle (2);
      %     \foreach \k in {1,2,...,100}{ 
      %         \pgfmathsetmacro{\t}{2-1/(\k^2)}
      %         \draw (-2, \t) -- (2,\t);
      %     }

      % \end{tikzpicture}
      \includegraphics{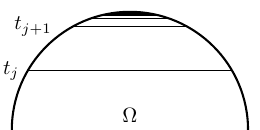}
      \caption{Example of set $S \in \Ten(\Omega;\mu)$ with countably many \emph{indecomposable stacking faults}, the horizontal lines in the picture. To build it, consider  a sequence $(t_j)_{j\in\N}$ such that $\H^1\big(\Pi(t_j e_2)\cap B_1\big)=2^{-j}$. Consider the set $S = \bigcup_{j \in \N} \Pi(t_j e_2)\cap B_1$. Then $\H^1(S) < +\infty$.  }
      \label{fig:infinite tensions}
  \end{figure}
  
      %\caption{Let $\Omega=B_1$, $\mu=\delta_0$ and let $\theta_\hor^+\in C^\infty(\R^2\sm\Pi^+;[0,2\pi])$ be a lifting of $\frac{x}{|x|}$ as in Example~\ref{ex:weven-wodd}. Below we construct a field $w\in\D_{1,\,\hor}^{1/2}(B_1)$ satisfying $w^2(x)=\frac{x}{|x|}$ whose jumpset consists of infinitely many segments. To this end, we observe that the map $g:[0,1]\to[0,2]$, $t\mapsto\H^1\big(\Pi(t e_2)\cap B_1\big)$ is continuous and strictly decreasing with $g(0)=2$ and $g(1)=0$. Thanks to the intermediate value theorem we thus find for every $j\in\N$ a $t_j\in (0,1)$ such that $\H^1\big(\Pi(t_j e_2)\cap B_1\big)=2^{-j}$ and $t_j<t_{j+1}$ for every $j\in\N$. This allows us to define a partition function $\chi\in SBV(B_1;\{0,\pi\})$ by setting $\chi(x)\defas 0$ if $x\cdot e_2\leq t_1$ and
    %   \begin{equation*}
    %   \chi(x)\defas
    %   \begin{cases}
    %   \pi &\text{if}\ t_j<x\cdot e_2\leq t_{j+1}\;\text{ for $j$ odd}\,,\\
    %   0 &\text{if}\ t_j<x\cdot e_2\leq t_{j+1}\;\text{ for $j$ even}\,.
    %   \end{cases}
    %   \end{equation*}
    % Let $w\defas\exp\big(\iota(\frac{\theta_\hor^+}{2}+\chi)\big)$; then $w^2(x)=\frac{x}{|x|}$ and~\eqref{inclusion-Jw-Jvarphi} implies that $S_w=\bigcup_{j=0}^\infty L_j$ with $L_0\defas\Pi^+\cap \overline{B_1}$ and $L_j\defas\big(\Pi(t_j e_2)\cap \overline{B_1})\big)$ for $j\in\N$. By the choice of $t_j$ we thus conclude that
    %   \begin{equation*}
    %   \H^1(S_w)=\sum_{j=0}^\infty \H^1\big(\Pi(t_j e_2)\cap B_1\big)=\sum_{j=0}^\infty 2^{-j}<+\infty\,,
    %   \end{equation*}
    % hence $w\in\D_{1,\,\hor}^2(B_1)$.}

    \begin{remark} \label{rem:Sc algebra}
      Note that $\Ten(\mu,\Omega)$ is closed under finite unions, finite intersections, and difference of sets. More precisely, if $S_1, S_2 \in \Ten(\mu,\Omega)$, then $S_1 \cup S_2 \in \Ten(\mu,\Omega)$, $S_1 \cap S_2 \in \Ten(\mu,\Omega)$, and $S_1 \sm S_2 \in \Ten(\mu,\Omega)$.
      %To show that it is also closed under finite intersections, we define a complementing operation $\complement_{\Ten} \colon \Ten(\mu,\Omega) \to \Ten(\mu,\Omega)$ as follows: 
      % \begin{itemize}
      %   \item Given $S = \bigcup_{j \in \N} (y_j^1, y_j^2) \cup \NN$, we group together the segments choosing a subfamily of points $\{y_i\}_i \subset \{y_j^1\}_j$ such that $\{\Pi(y_i)\}_i$ is a disjoint family of lines and for every $j \in \N$ there exists $i$ such that $(y_j^1, y_j^2) \subset \Pi(y_i)$.
      %   \item We observe that $\big( \Pi(y_i) \sm S \big) \cap \Omega \in \Ten(\mu,\Omega)$ for every $i$.
      %   \item We define the complement in $\Ten(\mu,\Omega)$ by
      %   \[
      %     \complement_{\Ten} S := \bigcup_{i} \big( \Pi(y_i) \sm S \big) \cap \Omega \in \Ten(\mu,\Omega) \, ,
      %   \]
      %   where the inclusion holds since $\Ten(\mu,\Omega)$ is closed under countable unions and elements of $\Ten(\mu,\Omega)$ can be modified up to a $\H^1$-negligible set. Note that $\complement_{\Ten} ( \complement_{\Ten} S ) = S$. 
      %   \item If $\{S_k\}_{k \in \N}$ is a countable family of elements in $\Ten(\mu,\Omega)$, then 
      %   \[
      %     \bigcup_{k\in \N} \complement_{\Ten} S_k   \in \Ten(\mu,\Omega) \implies \bigcap_{k\in \N}S_k  = \complement_{\Ten} \Big( \bigcup_{k\in \N} \complement_{\Ten} S_k \Big) \in \Ten(\mu,\Omega) \, .
      %   \]
      % \end{itemize}
      
    \end{remark}

    A special subclass of $\Ten(\mu,\Omega)$ is given in the following definition. We will show in Proposition~\ref{prop:countable segments} that discontinuity sets of elements in $\D_{M,\hor}^{1/2}(\Omega)$ satisfy the following property.  

    \begin{definition} \label{def:resolve}
    Let $\mu  = \sum_{h=1}^Md_h\delta_{x_h} \in X_M(\Omega)$ and $S \in \Ten(\mu,\Omega)$. We say that $S$ \emph{resolves dislocations tension} if, additionally, there exists $\sigma > 0$ such that, up to $\H^1$-negligible sets,
    \[
    S \cap B_\sigma(x_h) = \Pi^+(x_h) \cap B_\sigma(x_h)  \quad \text{or} \quad  S \cap B_\sigma(x_h)  = \Pi^-(x_h) \cap B_\sigma(x_h) \, , \quad \text{for } h = 1, \dots, M \, .
    \]
  \end{definition}

  \begin{remark}
    Definition~\ref{def:resolve} is given in this perspective: we will show that for admissible limit  configurations, partial dislocations cannot be isolated in the following sense. Each partial dislocation must be connected either to another partial dislocation, or to the boundary. Two stacking faults cannot stem from one partial dislocation.  \end{remark}

  \EEE

  \smallskip

  \begin{example}[Horizontal cuts] \label{ex:Gamma}
  In the next part of the section we will exploit a specific element $\Gamma \in \Ten(\mu,\Omega)$ consisting of horizontal one-directional cuts through dislocations, see also Figure~\ref{fig:OmegaGamma}. Let $\mu=\sum_{h=1}^M d_h\delta_{x_h}\in X_M(\Omega)$. For every $h = 1, \dots, M$ we let $\Gamma_h$ be the connected component of $\Pi^+(x_h) \cap \Omega$ containing $x_h$. We set $\Gamma := \bigcup_{h=1}^M \Gamma_h$ and we observe that $\Gamma \in \Ten(\mu,\Omega)$. More precisely, the decomposition of $\Gamma$ is a finite family of \emph{indecomposable stacking faults}, given by segments  $(x_h, \bar x)$ either connecting two dislocations (\ie $\bar x \in \supp \mu$) or a dislocation to the boundary (\ie $\bar x \in \de \Omega$).

  \end{example}

  %Moreover, we set $\Omega_\Gamma(\mu)\defas\Omega\sm \Gamma $ and for any $\sigma>0$ satisfying~\eqref{cond:sigma} we set $\Omega_\Gamma^\sigma(\mu)\defas \Omega^\sigma(\mu)\sm\Gamma$.
  \EEE 

  \smallskip

  \begin{figure}[ht]
    \begin{center}
    \includegraphics{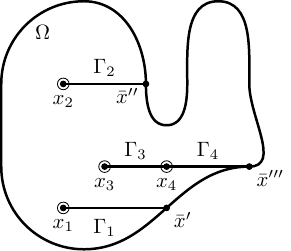}
    \end{center}
    %\caption{\RRR An example of $\Omega_\Gamma(\mu)$. Here, $\supp \mu = \{x_1,x_2,x_3,x_4\}$, where the points $x_h$ are labeled in such a way that \eqref{grouping:x-h-1}--\eqref{grouping:x-h-2} are satisfied. In this particular example, $M_0 = 2$, $K = 1$, and $M_1 = M_0 + 2 = 4$. The number of cuts is $N = M_0 + K = 3$. The small circles around the points $x_h$ have radii equal to $\sigma$. As the picture shows, then non-convexity of the domain may result in non-connectedness of $\Omega_\Gamma(\mu)$. According to Definition~\ref{def:connect singularities}, the \emph{indecomposable stacking faults} in this example are: $(x_1,y_1)$, $(x_3, x_4)$, $(x_4,y_3)$, $(x_2, y_2)$, $(y_2',y_2'')$. The (not indecomposable) \emph{stacking fault} $\Gamma_3$ does not \emph{resolve dislocations tension} according to Definition~\ref{def:resolve}. \EEE}
    \caption{An example of $\Omega^\sigma(\mu) \sm \Gamma$. Here, $\supp \mu = \{x_1,x_2,x_3,x_4\}$. According to Definition~\ref{def:connect singularities}, the \emph{indecomposable stacking faults} in this example are: $(x_1,\bar x')$, $(x_3, x_4)$, $(x_4,\bar x''')$, $(x_2, \bar x'')$. The \emph{stacking fault} $\Gamma$ does not \emph{resolve dislocations tension} according to Definition~\ref{def:resolve} because of the local aspect of $\Gamma$ around $x_4$.}
    \label{fig:OmegaGamma}
  \end{figure}

  We start by providing a lifting result for fields in $\D_M(\Omega)$. Exploiting the lifting, given $v\in\D_M(\Omega)$, we build a specific spin field $v^{1/2} \in \D_{M,\hor}^{1/2}(\Omega)$ satisfying $( v^{1/2} )^2 = v$ (note that $v^{1/2}$ is just a symbol adopted here). Some steps in the proof of the lifting result are classical, but we provide the details as some of them will turn out useful. 

  \begin{lemma}\label{lem:lifting-v}
  Let $\mu = \sum_{h=1}^M d_h \delta_{x_h} \in X_M(\Omega)$, let $v\in\D_M(\Omega)$ with $J(v)=\pi\mu$, and let $\Gamma$ be as in Example~\ref{ex:Gamma}. Then there exists a lifting $\varphi\in W^{1,1}(\Omega \sm \Gamma)$ satisfying $\exp(\iota\varphi)=v$ a.e.\ in $\Omega$. Moreover, the following conditions are satisfied:
  \begin{enumerate}[label=(\roman*)]
    \item \label{item:H1} $\varphi\in H^1(\Omega^\sigma(\mu) \sm \Gamma)$ for every $\sigma>0$ satisfying~\eqref{cond:sigma};
    \item \label{item:varphi SBV} $\varphi \in SBV(\Omega)$ with $S_\varphi \in \Ten(\mu,\Omega)$;
    \item \label{item:jump varphi constant} $[\varphi]\in 2\pi\Z$ and $[\varphi]$ is constant on the \emph{indecomposable stacking faults} in the decomposition of $S_\varphi$;
    \item \label{item:v root} $v^{1/2} := \exp\big(\iota \frac{\varphi}{2}\big) \in \D_{M,\hor}^{1/2}(\Omega)$ with $S_{v^{1/2}} \in \Ten(\mu,\Omega)$ \emph{resolving dislocations tension} according to Definition~\ref{def:resolve};
    \item \label{item:approximation} for every $\sigma>0$ satisfying~\eqref{cond:sigma}, there exists a sequence $(\varphi_n^\sigma)\subset C^\infty(\Omega^\sigma(\mu) \sm \Gamma)\cap H^1(\Omega^\sigma(\mu) \sm \Gamma)$ satisfying
    \begin{equation}\label{eq:approx-phi}
    \lim_{n\to +\infty}\|\varphi-\varphi_n^\sigma\|_{H^1(\Omega^\sigma(\mu) \sm \Gamma)}=0
    \end{equation}
    and $[\varphi_n^\sigma]=[\varphi]\in 2\pi\Z$ on $\Gamma \cap\Omega^\sigma(\mu)$ for $n$ sufficiently large.
  \end{enumerate}
  \EEE
  % and $[\varphi_n^\sigma]=[\varphi]\in 2\pi\Z$ on $(\Gamma_1\cup\ldots\cup\Gamma_N)\cap\Omega^\sigma(\mu)$ for $n$ sufficiently large.  \RRR Moreover, $\varphi \in SBV(\Omega)$ with $S_\varphi \in \Ten(\mu,\Omega)$,  and $[\varphi]$ is constant on the \emph{indecomposable stacking faults} in the decomposition of $S_\varphi$.\EEE
  \end{lemma}
  \begin{proof}
  %We split the proof into two steps, establishing first the existence of a lifting $\varphi$ and then the existence of a smooth approximation. 
  \begin{step}{1}{Construction of $\varphi$ and proof of~\ref{item:H1}}
  Since $v\in W^{1,1}(\Omega;\S^1)\subset SBV(\Omega;\S^1)$, \cite[Theorem 1.1 and Remark 4]{DI03} provides us with a function $\vartheta\in SBV(\Omega)$ such $v=\exp(\iota\vartheta)$ a.e.\ in $\Omega$. The chain rule for $BV$-functions~\cite[Theorem 3.96]{AFP} then implies that
    \begin{equation}\label{eq:chain-rule-v}
    \nabla v\L^2 = \iota v \otimes \nabla\vartheta\L^2\quad\text{and}\quad 0=D^jv= \big(\exp(\iota\vartheta^+)-\exp(\iota\vartheta^-)\big) \otimes \nu_\vartheta\H^1\mres S_\vartheta\,, 
    \end{equation}
  where $\nabla\vartheta$ denotes the approximate gradient of $\vartheta$. From the first equality we deduce that $\nabla\vartheta=2j(v)$ a.e.\ in $\Omega$. 
  This implies, in particular, that $\curl\nabla\vartheta=0$  in $\Omega \sm \Gamma$, since $\curl j(v)=J(v)=\pi\mu$. Thus, by the simply connectedness of $\Omega \sm \Gamma$, there exists $\varphi\in SBV(\Omega) \cap W^{1,1}(\Omega \sm \Gamma)$ with $\nabla\varphi=\nabla\vartheta$ a.e.\ in $\Omega$. In particular, that $S_\varphi \subset \Gamma$. The remainder satisfies $\varphi - \vartheta = \sum_{j\in\N}a_j\mathds{1}_{U_j}$ for some Caccioppoli partition $(U_j)_{j\in\N}$ and $a_j\in\R$ (see, \eg \cite[Theorem 4.23]{AFP}). % (The discontinuity set of $\sum_{j\in\N}a_j\mathds{1}_{U_j}$ is, up to $\H^1$-negligible sets, contained in $S_\vartheta\cup\Gamma$.) 
  In Step~3 below, we show that $[\varphi] \in 2 \pi \Z$, $\H^1$-a.e.\ on $S_\varphi$. Since by the second equality in~\eqref{eq:chain-rule-v} we also have that $[\vartheta]\in 2\pi\Z$, $\H^1$-a.e.\ on $S_\vartheta$, we deduce that the jump amplitudes of the Caccioppoli partition belong to $2 \pi \Z$, $\H^1$-a.e.\ and thus, up to removing a constant, we may assume that $a_j \in 2\pi \Z$. With this renormalisation, we obtain $\exp(\iota \varphi) = \exp(\iota \vartheta) = v$ a.e.\ in $\Omega$, hence the desired lifting identity. 

  Eventually, for every $\sigma>0$ satisfying~\eqref{cond:sigma} we have $\|\nabla\varphi\|_{L^2(\Omega^\sigma(\mu) \sm \Gamma)}=\|\nabla v\|_{L^2(\Omega^\sigma(\mu))}<+\infty$, so that $\varphi\in H^1(\Omega^\sigma(\mu) \sm \Gamma)$ (we can always assume that $\|\varphi\|_{L^\infty(\Omega)}\leq 2\pi M$). This proves~\ref{item:H1}. \EEE 
  \end{step}

  \begin{step}{2}{Proof of \ref{item:varphi SBV}--\ref{item:jump varphi constant}} We recall that $S_\varphi \subset \Gamma$, by Step~1. Let us fix an \emph{indecomposable stacking fault} $(x_h, \bar x)$ in the decomposition of $\Gamma$, with $h \in \{1,\dots, M\}$ and $\bar x \in \supp \mu \cup \de \Omega$. Let us show that $[\varphi]$ is constant on $(x_h, \bar x)$. Let us fix $x \in (x_h, \bar x)$ and $\sigma_0 > 0$ such that $B_{\sigma_0}(x) \cap \Gamma = B_{\sigma_0}(x) \cap (x_h, \bar x)$. By~\eqref{eq:stokes}, since $(x_h, \bar x) \cap \supp \mu = \emptyset$, and using the identity $\nabla\varphi=2j(v)$ obtained above we have that for a.e.\ $\sigma \in (0, \sigma_0)$
  \[
    0 =  \deg(v,\partial B_\sigma(x))=\frac{1}{2\pi}\int_{\partial B_\sigma(x)} \nabla\varphi\cdot\tau_{\partial B_\sigma(x)}\dH= - \frac{1}{2\pi}[\varphi](x+\sigma e_1) + \frac{1}{2\pi}[\varphi](x-\sigma e_1)  \, .
  \]
  With this technique, we cover the segment $(x_h, \bar x)$ to deduce that $[\varphi]$ is constant on $(x_h, \bar x)$. It follows that, up to $\H^1$-negligible sets,
  \[
  S_\varphi = \bigcup \{(x_h, \bar x) \text{ \emph{indecomposable stacking fault} in $\Gamma$ such that }  [\varphi]_{|(x_h, \bar x)} \neq 0 \} \in \Ten(\mu,\Omega) \, .
  \]
  We have proven~\ref{item:varphi SBV}--\ref{item:jump varphi constant}. 
  \end{step}

  \begin{step}{3}{$[\varphi] \in 2 \pi \Z$}
    We relate the constant values of $[\varphi]$ to the degree of the singularities by distinguishing two cases as follows. If $\Pi^-(x_h) \cap \Gamma = \{x_h\}$ (\ie $x_h$ is the first singularity on this horizontal line), then for a.e.\ $\sigma>0$ satisfying~\eqref{cond:sigma} we have that 
  \begin{equation} \label{eq:degre and jump-1}
    d_h = \deg(v,\partial B_\sigma(x_h))=\frac{1}{2\pi}\int_{\partial B_\sigma(x_h)} \nabla\varphi\cdot\tau_{\partial B_\sigma(x_h)}\dH=  - \frac{1}{2\pi} [\varphi]_{|(x_h, \bar x)} \,,
  \end{equation}
  hence $[\varphi]_{|(x_h, \bar x)} = - 2 \pi d_h \in\{-2\pi,0,2\pi\}$. Otherwise, if $\Pi^-(x_h) \cap \Gamma \neq \{x_h\}$, then there exists $x_{h'} \in \supp \mu$ with $h' \in \{1, \dots, M\}$ such that $(x_{h'}, x_{h})$ is an \emph{indecomposable stacking fault} in the decomposition of $\Gamma$ contiguous to $(x_h, \bar x)$. Then for a.e.\ $\sigma>0$ satisfying~\eqref{cond:sigma} we have that 
  \begin{equation} \label{eq:degre and jump-2}
    d_h = \deg(v,\partial B_\sigma(x_h)) =  - \frac{1}{2\pi} [\varphi]_{|(x_h, \bar x)} + \frac{1}{2\pi} [\varphi]_{|(x_{h'}, x_h)} \,,
  \end{equation}
  hence $[\varphi]_{|(x_h, \bar x)} = [\varphi]_{|(x_{h'}, x_h)} - 2 \pi d_h$. Using the relations~\eqref{eq:degre and jump-1}--\eqref{eq:degre and jump-2} iteratively on \emph{indecomposable stacking faults} lying on the same horizontal line, we deduce two facts: 
  \begin{enumerate}[label=\alph*)]
    \item $[\varphi]_{|(x_h, \bar x)} \in 2 \pi \Z$ for every \emph{indecomposable stacking fault} $(x_h, \bar x)$ in the decomposition of $\Gamma$;
    \item if $(x_{h'}, x_h), (x_h, \bar x)$ are contiguous \emph{indecomposable stacking faults} in the decomposition of $\Gamma$, then we have the dichotomy: one and only one between $[\varphi]_{|(x_h, \bar x)}$ and $[\varphi]_{|(x_{h'}, x_h)}$ belongs to $4 \pi \Z$. 
  \end{enumerate}
  \end{step}

  \begin{step}{4}{Proof of~\ref{item:v root}} 
    By the chain rule for $BV$ functions~\cite[Theorem 3.96]{AFP} we deduce that $v^{1/2} = \exp(\iota \frac{\varphi}{2}) \in SBV(\Omega;\S^2)$. By Step~1, $(v^{1/2})^2 = \exp(\iota \varphi) = v$ a.e.\ in $\Omega$ and, in particular, $J((v^{1/2})^2) = J(v) = \pi \mu$. By property~\ref{item:H1} and by the chain rule in $H^1$, we deduce that $v^{1/2} \in H^1(\Omega^\sigma(\mu) \sm \Gamma)$ for every $\sigma>0$ satisfying~\eqref{cond:sigma}. It follows that $v^{1/2} \in SBV^2_\loc(\Omega \sm \supp \mu ; \S^1)$ and $S_{v^{1/2}} \subset \Gamma$. From the latter fact it follows that $\H^1(S_{v^{1/2}} \cap \Omega) < +\infty$ and $|\nu_{v^{1/2}} \cdot e_1| = 0$, $\H^1$-a.e.\ on $S_{v^{1/2}}$. With this, we proved all the properties to conclude that $v^{1/2} \in \D_{M,\hor}^{1/2}(\Omega)$. 

    To infer that $S_{v^{1/2}} \in \Ten(\mu,\Omega)$, we rely once more on the chain rule for $BV$ functions to deduce that, up to $\H^1$-negligible sets, 
    \begin{equation} \label{eq:jump set of v root}
    S_{v^{1/2}} = \{ x \in S_{\varphi} \ : \ [\varphi] \in 2 \pi \Z \sm 4 \pi \Z \}   \, .
    \end{equation}
    Using~\ref{item:varphi SBV}--\ref{item:jump varphi constant}, this implies that $S_{v^{1/2}}$ consists of the union of a subfamily of the \emph{indecomposable stacking faults} in the decomposition of $S_\varphi$, hence $S_{v^{1/2}} \in \Ten(\mu,\Omega)$.  

    It remains to prove that $S_{v^{1/2}}$ \emph{resolves dislocations tension} according to Definition~\ref{def:resolve}. Let us consider a singularity $x_h \in \supp \mu$ with $h \in \{1,\dots, M\}$. If $\Pi^-(x_h) \cap \Gamma = \{x_h\}$, then Step~2 yields $[\varphi] = - 2\pi d_h \in 2 \pi \Z \sm 4 \pi \Z$ on $\Pi^+(x_h) \cap B_{\sigma}(x_h)$. By~\eqref{eq:jump set of v root}, it follows that $S_{v^{1/2}} \cap B_{\sigma}(x_h) = \Pi^+(x_h) \cap B_{\sigma}(x_h)$. Otherwise, if $\Pi^-(x_h) \cap \Gamma \neq \{x_h\}$, the dichotomy shown in Step~3b) implies that 
    \[
    S_{v^{1/2}} \cap B_\sigma(x_h) = \Pi^+(x_h) \cap B_{\sigma}(x_h) \quad \text{or}  \quad S_{v^{1/2}} \cap B_\sigma(x_h) = \Pi^-(x_h) \cap B_{\sigma}(x_h) \, ,
    \]
    depending on whether $[\varphi]_{| \Pi^-(x_h) \cap B_{\sigma}(x_h)} \in 4 \pi \Z$ or $[\varphi]_{| \Pi^+(x_h) \cap B_{\sigma}(x_h)} \in 4 \pi \Z$, respectively. This agrees with Definition~\ref{def:resolve}.
  \end{step}

  \begin{step}{5}{Proof of~\ref{item:approximation}}
    We approximate $\varphi$ by piecewise smooth functions. Let $\sigma>0$ satisfying~\eqref{cond:sigma} be fixed. Since $v\in H^1(\Omega^\sigma(\mu);\S^1)$, the approximation theorem for Sobolev maps with values in compact manifolds~\cite[Section 4]{SU83} provides us with a sequence $(v_n)\subset C^\infty(\Omega^\sigma(\mu);\S^1)$ satisfying $v_n\to v$ strongly in $H^1(\Omega^\sigma(\mu);\S^1)$. By the continuity of the degree, for $n$ sufficiently large we know that $\deg(v_n,\partial B_\sigma(x_h))=d_h$ for any $h\in\{1,\ldots,M\}$. Moreover, for any $n\in\N$ there exists a lifting $\varphi_n^\sigma\in C^\infty(\Omega^\sigma(\mu) \sm \Gamma)$ of $v_n$, which thanks to the degree condition and~\eqref{eq:degre and jump-1}--\eqref{eq:degre and jump-2} satisfies for $n$ sufficiently large $[\varphi_n^\sigma] =[\varphi]$, $\H^1$-a.e.\ on $\Gamma$. Eventually, $\nabla\varphi_n^\sigma\to\nabla \varphi$ in $L^2(\Omega^\sigma(\mu) \sm \Gamma;\R^2)$. To see this, it suffices to recall that $\nabla \varphi_n^\sigma=2j(v_n)$ and $\nabla\varphi= 2j(v)$, so that
      \begin{align*}
      \int_{\Omega^\sigma(\mu) \sm \Gamma}|\nabla\varphi_n^\sigma-\nabla\varphi|\dx=2\int_{\Omega^\sigma(\mu) \sm \Gamma}|j(v_n)-j(v)|\dx\ \to 0\;\text{ as}\ n\to+\infty\,
      \end{align*}
    and
      \begin{align*}
      \|\nabla\varphi_n^\sigma\|_{L^2(\Omega^\sigma(\mu) \sm \Gamma)}=\|\nabla v_n\|_{L^2(\Omega^\sigma(\mu) \sm \Gamma)}\to \|\nabla v\|_{L^2(\Omega^\sigma(\mu) \sm \Gamma)}=\|\nabla\varphi\|_{L^2(\Omega^\sigma(\mu) \sm \Gamma)}\;\text{ as}\ n\to+\infty\,.
      \end{align*}
    Hence the claim follows from the Radon-Riesz Theorem.  Up to replacing $\varphi_n^\sigma$ by $\varphi_n^\sigma-\fint_{\Omega^\sigma(\mu) \sm \Gamma} (\varphi_n^\sigma-\varphi)$ we deduce the convergence $\|\varphi_n^\sigma-\varphi\|_{H^1(\Omega^\sigma(\mu) \sm \Gamma)}\to 0$ by applying the Poincar\'{e} inequality.
    \end{step}

  %Since $\varphi\in W^{1,1}(\Omega_\Gamma(\mu))\cap H^1(\Omega_\Gamma^\sigma(\mu))$ for every $\sigma>0$, the approximation theorem for Sobolev functions provides us with a sequence $(\varphi_n)\subset W^{1,1}(\Omega_\Gamma(\mu))\cap C^\infty(\Omega_{\Gamma}(\mu))$ such that $\varphi_n\in H^1(\Omega_\Gamma^\sigma(\mu))$ for every $\sigma>0$ and satisfying~\eqref{eq:approx-phi}. It remains to show that $[\varphi_n]=[\varphi]$ on $\Gamma_1\cup\ldots\cup\Gamma_N$ for $n$ sufficiently large. 
  
  \end{proof}
  %
  % \begin{remark}\label{rem:lifting-v-2}
  % Using the lifting provided by Lemma~\ref{lem:lifting-v} we obtain for any function $v\in\D_M(\Omega)$ an admissible function $w\in \D_{M,\hor}^{1/2}(\Omega)$ with $w^2=v$ by setting $w=\exp(\iota\frac{\varphi}{2})$. In this way, $S_w \subset \Gamma_1\cup\ldots\cup\Gamma_N$, so that $\H^1(S_w)<+\infty$ and $|\nu_w\cdot e_1|=0$ on $S_w$. 
  % \end{remark}
  % %

  \begin{figure}[ht]
    \begin{center}
        \includegraphics{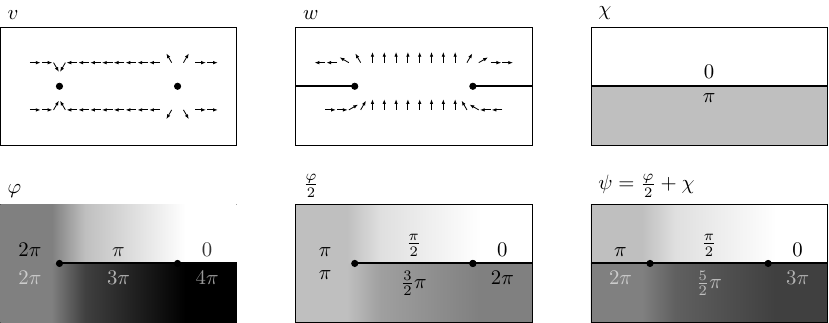}
    \end{center}
    \caption{A schematic example of $w\in\D_{M,\hor}^{1/2}(\Omega)$, $v=w^2 \in \D_M(\Omega)$, and their liftings obtained in Lemma~\ref{lem:lifting-w}. The spin field $v$ has two singularities of degree $1$. The spin field $w$ has two half singularities and jumps on the horizontal segments. The lifting $\varphi$ of $v$ jumps on the horizontal line ($\Gamma$, in the notation introduced above) that contains both singularities. Note that $\frac{\varphi}{2}$ is not a lifting of $w$, since $v^{1/2} = \exp(\iota \frac{\varphi}{2})$ jumps on the segment connecting the two singularities, not drawn in the picture. Adding the partition $\chi$, which jumps on $S_w \cup \Gamma$, gives a lifting $\psi = \frac{\varphi}{2} + \chi$ of $w$.}
    \label{fig:partition}
  \end{figure}
  
  Using similar arguments as in Lemma~\ref{lem:lifting-v} we next construct for any $w\in\D_{M,\hor}^{1/2}(\Omega)$ a suitable lifting $\psi$. % whose discontinuity set is still horizontal. 
  We stress that the spin field $v^{1/2}$ built in Lemma~\ref{lem:lifting-v} does not satisfy, in general, $w = v^{1/2}$. Hence $\frac{\varphi}{2}$ is not a lifting of $w$, see also Figure~\ref{fig:partition}. The next lemma shows how to build the correcting term $\chi$ for the lifting.

  \begin{lemma}\label{lem:lifting-w}
  Let $\mu=\sum_{h=1}^Md_h\delta_{x_h} \in X_M(\Omega)$, let $w\in\D_{M,\hor}^{1/2}(\Omega)$ with $v\defas w^2 \in \D_M(\Omega) $ satisfying $J(v)=\pi\mu$, and let \,%$\Gamma_1,\ldots,\Gamma_N$, $N\leq M$ be as above. 
  $\Gamma$ be as in Example~\ref{ex:Gamma}. Let moreover $\varphi$ be the lifting of $v$ as in Lemma~\ref{lem:lifting-v}. Then there exists $\chi\in SBV(\Omega;\{0,\pi\})$ such that the angle defined by $\psi\defas\frac{\varphi}{2}+\chi$ satisfies $w=\exp(\iota\psi)$ a.e.\ in $\Omega$ and, up to $\H^1$-negligible sets,
    \begin{equation}\label{inclusion:jumpset-varphi}
    S_w \subset S_\psi \subset 
    S_w\cup \Gamma \,. 
    %S_w\cup\Gamma_1\cup\ldots\cup\Gamma_N\;\text{ up to an $\H^1$-negligible set}\,.
	\end{equation}
Moreover, we have, up to $\H^1$-negligible sets,
\begin{equation} \label{cond:jump w - jump psi}
  S_w=\big\{x\in S_{\psi}\colon [\psi](x)\in \pi\Z\sm 2\pi\Z\big\}\,. 
\end{equation} 
	% \begin{equation}\label{cond:jump-varphi}
	% [\psi]_{|S_\psi\cap S_w}\in\pi\Z\sm 2\pi\Z\quad\text{and}\quad[\psi]_{|S_{\psi}\sm S_w}\in 2\pi\Z\,.
	% \end{equation}
\end{lemma}
\begin{proof}
For $w\in\D_{M,\hor}^{1/2}(\Omega)\subset SBV(\Omega;\S^1)$ we use again~\cite[Theorem 1.1 and Remark 4]{DI03} to find a function $\vartheta\in SBV(\Omega)$ satisfying $w=\exp(\iota\vartheta)$ and consequently $v=\exp(2\iota\vartheta)$ and $2\nabla\vartheta=2j(v)=\nabla\varphi$ a.e.\ in $\Omega$. This implies that $2\vartheta-\varphi\in 2\pi\Z$ and $\nabla(2\vartheta-\varphi)=0$ a.e.\ in $\Omega$.
Thus, there exists a Caccioppoli partition $(U_j)_{j\in\N}$ %\RRR \sout{subordinated to $S_\vartheta\cup\Gamma_1\cup\ldots\cup\Gamma_N$} \EEE 
such that
	\begin{equation}\label{eq:cacciopoli}
	2\vartheta-\varphi=\sum_{j\in\N}z_j\mathds{1}_{U_j}\;\text{ with}\ z_j\in 2\pi\Z\,,
	\end{equation}
  and such that the discontinuity set of $\sum_{j\in\N}z_j\mathds{1}_{U_j}$ is contained, up to $\H^1$-negligible sets, in $S_\vartheta\cup\Gamma$. We then construct  $\chi\in SBV(\Omega;\{0,\pi\})$ by setting
	\begin{equation*}
	E\defas\bigcup_{\begin{smallmatrix}j\in\N\\ z_j\in 4\pi\Z\end{smallmatrix}} U_j\quad\text{and}\quad\chi\defas \pi \mathds{1}_{\Omega\sm E}
	\end{equation*}	 
Note that \eqref{eq:cacciopoli} implies that $x\in E$ if and only if $w(x)=\exp(\iota\frac{\varphi}{2}(x))$ and $x\in \Omega\sm E$ if and only if $w(x)=\exp(\iota\frac{\varphi}{2}(x)+\iota\pi)$. In particular, $\psi\defas\frac{\varphi}{2}+\chi$ is a lifting of $w$ and it remains to verify~\eqref{inclusion:jumpset-varphi} and~\eqref{cond:jump w - jump psi}. It follows by construction that $[\psi]\in\pi\Z$, $\H^1$-a.e.\ on $S_\psi$. Moreover, again by chain rule, we have $D^jw=\big(\exp(\iota{\psi}^+)-\exp(\iota{\psi}^-)\big) \otimes \nu_{\psi}\H^1\mres S_{\psi}$, whence \eqref{cond:jump w - jump psi}. 
	% \begin{equation}\label{inclusion-Jw-Jvarphi}
	% S_w=\big\{x\in S_{\psi}\colon [\psi](x)\in \pi\Z\sm 2\pi\Z\big\}\,.
	% \end{equation}
Since for any $x\in S_{\psi}$ with $[\psi](x)=\big[\frac{\varphi}{2}\big](x)+[\chi](x)\in 2\pi\Z$ we necessarily have that $\big[\frac{\varphi}{2}\big]\neq 0$, we deduce~\eqref{inclusion:jumpset-varphi} from~\eqref{cond:jump w - jump psi} together with the definition of $\varphi$. % Moreover,~\eqref{inclusion-Jw-Jvarphi} clearly implies~\eqref{cond:jump-varphi}.
\end{proof}
Eventually, the lifting result proven in Lemma~\ref{lem:lifting-w} allows us to deduce a strong structure result for the discontinuity set $S_w$ of a generic $w\in\D_{M,\hor}^{1/2}(\Omega)$ analogous to the one satisfied by $v^{1/2}$ in Lemma~\ref{lem:lifting-v}. 

\begin{proposition} \label{prop:countable segments}
  Let $\mu \in X_M(\Omega)$, let $w\in\D_{M,\hor}^{1/2}(\Omega)$ with $v = w^2 \in \D_M(\Omega)$ satisfying $J(v) = \pi \mu$. Then $S_w \in \Ten(\mu,\Omega)$. Moreover, $S_w$ \emph{resolves dislocations tension} according to Definition~\ref{def:resolve}.
\end{proposition}
\begin{proof}
  Let $v = w^2 \in \D_M(\Omega)$ and let $\mu,\varphi, \psi, \chi$ be as in Lemma~\ref{lem:lifting-w}. By the formula $\chi = \psi - \frac{\varphi}{2}$ we deduce that $S_\chi \subset S_{\psi} \cup S_{\varphi}$. By Lemma~\ref{lem:lifting-v} we have that $S_\varphi \subset \Gamma$ and by Lemma~\ref{lem:lifting-w} we have that $S_\psi \subset S_w \cup \Gamma$ (both inclusions up to $\H^1$-negligible sets). We deduce that $S_\chi \subset S_w \cup \Gamma$, up to $\H^1$-negligible sets. Since by definition of $\D_{M,\hor}^{1/2}(\Omega)$ we have that the approximate normal to $S_w$ is vertical $\H^1$-a.e.\ on $S_w$, we infer that 
  \begin{equation} \label{eq:chi horizontal jumps}
    \nu_\chi \cdot e_1 = 0 \quad \H^1\text{-a.e.\ on }  S_\chi\, .
  \end{equation}
  However, $\chi \in SBV(\Omega;\{0,\pi\})$ is a partition function, hence the previous condition enforces lamination. More precisely, $S_\chi \in \Ten(\mu,\Omega)$, where, in fact, $S_\chi$ is the countable union of horizontal segments connecting boundary points. Moreover, $[\chi]$ is constant $\H^1$-a.e.\ on these horizontal segments. These facts are classical, but we provide the details for the reader's convenience in Steps~1--3 below.

  \begin{step}{1}{Independence from $x_1$ in a rectangle} We start by working on the rectangle $\Omega = (0,1) \x (-\ell,\ell)$, with $\ell > 0$. By~\eqref{eq:chi horizontal jumps} we have that $D \chi \cdot e_1 = 0$ in the distributional sense. Then $\chi(x) = \chi(x + s e_1)$ for a.e.\ $x$ satisfying $x+s e_1 \in \Omega$, \ie $\chi$ is independent of $x_1$.\footnote{This fact can bee proven via an approximation argument. Fix two rectangles $\Omega'' \subset \subset \Omega' \subset \subset \Omega$ and a mollification kernel $\eta_j$.  Observe that $D(\eta_j * \chi)\cdot e_1 = \eta_j * D \chi \cdot e_1 = 0$. The functions $\eta_j * \chi$ are smooth and $\Omega'$ is convex, thus the functions $\eta_j * \chi$ depend on the $x_2$-variable only. Moreover, we have that $\eta_j * \chi(x) \to \chi(x)$ for a.e.\ $x \in \Omega''$ and $\eta_j * \chi(x) = \eta_j * \chi(x+s e_1) \to \chi(x+s e_1)$ for a.e.\ $x \in \Omega''$. Hence $\chi(x) = \chi(x + s e_1)$ for a.e.\ $x \in \Omega''$. The argument can then be repeated to cover the whole $\Omega$.}
  \end{step}

  \begin{step}{2}{The discontinuity set in a rectangle consists of finitely many horizontal lines} Let $\Omega = (0,1) \x (-\ell,\ell)$, with $\ell > 0$. By the previous step we have that, for $r$ small enough and $c\in\R$, 
    \[
     \fint_{B_r(x + s e_1)} |\chi(y) - c|\d y = \fint_{B_r(x)} |\chi(y + s e_1) - c|\d y = \fint_{B_r(x)} |\chi(y ) - c| \d y \, ,
    \]
   hence $x\in S_\chi \iff x + s e_1 \in S_\chi$. This implies that there exists a family $\{L_i\}_{i \in \mathcal{I}}$ of disjoint lines of the form $L_i = (0,1) \x \{t_i\}$ such that $S_\chi \cap \Omega = \bigcup_{i \in \mathcal{I}} L_i$. The set of indices $\mathcal{I}$ is, in fact, a finite set. Indeed, since $\H^1(S_\chi \cap \Omega) < +\infty$, we get that for every finite subset $\overline{\mathcal{I}} \subset \mathcal{I}$
   \[
   \# \overline{\mathcal{I}} = \sum_{j \in \overline{\mathcal{I}}} \H^1(L_j) \leq \H^1(S_\chi \cap \Omega) < +\infty \, ,
   \]
   hence the cardinality of $\overline{\mathcal{I}}$ is equibounded. 
  \end{step}

  \begin{step}{3}{The discontinuity set in $\Omega$ consists of segments that connect boundary points} Let $\Omega$ be as in the rest of the paper, see Subsection~\ref{sec:limit-fields}. We cover $\Omega$ with countably many open squares and we apply Steps~1--2 to deduce that $S_\chi$ is a countable union of horizontal segments. Let us fix one of these segments, let $x \in S_\chi$ be a point of this segment, and let us consider the horizontal line $\Pi(x)$. The section $\Pi(x) \cap \Omega$ is relatively open on the line $\Pi(x)$ and is thus the countable union of its connected components, given by segments whose extrema lie on $\de \Omega$. Let $(y_1,y_2)$ be the connected component of $\Pi(x) \cap \Omega$ containing $x$, with $y_1,y_2 \in \de \Omega$. We claim that $(y_1,y_2) \subset S_\chi$. Indeed, for every $\delta > 0$ there exists $\ell > 0$ such that the rectangle $(y_1 + \delta e_1, y_2 - \delta e_1) \times \ell (-e_1,e_2)$ is fully contained in~$\Omega$. (For, otherwise, one could find a sequence of points outside $\Omega$ converging to a point of $\de \Omega$ on the closed segment $[y_1  + \delta e_1 , y_2  - \delta e_2]$, contradicting the connectedness of $(y_1, y_2)$.) Applying Step~1--2, we deduce that $(y_1  + \delta e_1 , y_2  - \delta e_2) \subset S_\chi$. Letting $\delta \to 0$, we prove the claim. This concludes the proof of the fact that $S_\chi \in \Ten(\mu,\Omega)$ and $S_\chi$ is the countable union of horizontal segments connecting boundary points.
  \end{step}

  \smallskip 

  We are now in a position to prove that  $S_w \in \Ten(\mu,\Omega)$ using the closure properties of $\Ten(\mu,\Omega)$ illustrated in Remark~\ref{rem:Sc algebra}. We define the spin field $R := \exp(\iota \chi)$, which with the identification $\R^2 \simeq \C$ can be interpreted as a piecewise constant rotation acting on $\C$. By the chain rule for $BV$ functions~\cite[Theorem 3.96]{AFP}, we have that $R \in SBV(\Omega; \{e_1, -e_1\})$ with $S_R=S_\chi$. In particular, $S_R \in \Ten(\mu,\Omega)$ and the \emph{indecomposable stacking faults} in the decomposition of $S_R$ connect boundary points. By Lemma~\ref{lem:lifting-w}, we have that 
  \[
    w = \exp(\iota \psi) = \exp\Big(\iota \frac{\varphi}{2} + \iota \chi\Big) = \exp\Big(\iota \frac{\varphi}{2}\Big) \exp(\iota \chi ) = v^{1/2} R \, ,
  \]
  where $v^{1/2}$ is the spin field defined in Lemma~\ref{lem:lifting-v}. By the product rule, see~\cite[Example~3.97]{AFP}, we have that, up to $\H^1$-negligible sets, $S_w \subset S_{v^{1/2}} \cup S_R$ and 
  \[
  [w] = \begin{cases}
    [v^{1/2}] \tilde  R \, , & \text{$\H^1$-a.e.\ on } S_{v^{1/2}} \sm S_R \, , \\ 
    \tilde{v^{1/2}} [R] \, , & \text{$\H^1$-a.e.\ on } S_R \sm S_{v^{1/2}} \, , \\ 
    [ v^{1/2} R ] \, ,   & \text{$\H^1$-a.e.\ on } S_{v^{1/2}} \cap S_R \, ,
  \end{cases}
  \]
  where $\tilde{v^{1/2}}$ and $\tilde R$ are the precise representatives of $v^{1/2}$ and $R$, respectively. However, let us fix $x \in S_{v^{1/2}} \cap S_R$ such that the jumps are well-defined (this occurs $\H^1$-a.e.). Then, by~\eqref{eq:jump set of v root} we have that $[\frac{\varphi}{2}](x) \in \pi \Z \sm 2 \pi \Z$. On the other hand, since $S_R = S_\chi$, we have that $[\chi](x) = \pm \pi$. It follows that $[\psi](x) = [\frac{\varphi}{2}](x) + [\chi](x) \in 2 \pi \Z$. Hence, by~\eqref{cond:jump w - jump psi}, $\H^1$-a.e.\ of such $x$'s do not belong to $S_w$. We conclude that, up to $\H^1$-negligible sets,
  \begin{equation} \label{eq:Sw}
    S_w = (S_{v^{1/2}} \sm S_R) \cup (S_R \sm S_{v^{1/2}}) \, .
  \end{equation}
  Figure~\ref{fig:partition} shows an instance of this equality.  Since both $S_{v^{1/2}}, S_R \in \Ten(\mu,\Omega)$, by the closure properties in Remark~\ref{rem:Sc algebra} we deduce that $S_w \in \Ten(\mu,\Omega)$.

  \smallskip
  
  We are left to prove that $S_w$ \emph{resolves dislocations tension}. Let us fix $x_h \in \supp \mu$ with $h \in \{1,\dots, M\}$. A ball $B_\sigma(x_h)$ intersects finitely many \emph{indecomposable stacking faults} in the decomposition of $S_w$, see Remark~\ref{rem:countable components}. Thus, considering $\sigma > 0$ small enough, we can assume that  $B_\sigma(x_h)$ intersects only one \emph{indecomposable stacking fault}. By Lemma~\ref{lem:lifting-v}-\ref{item:v root}, $S_{v^{1/2}}$ \emph{resolves dislocations tension}, hence $S_{v^{1/2}} \cap B_\sigma(x_h) = \Pi^{\pm}(x_h) \cap B_\sigma(x_h)$ by Definition~\ref{def:resolve}.  If $S_R \cap B_\sigma(x_h) = \emptyset$, by~\eqref{eq:Sw} this implies that, up to $\H^1$-negligible sets, 
  \[
    S_{w} \cap B_\sigma(x_h) = (S_{v^{1/2}} \sm S_R)  \cap B_\sigma(x_h) = \Pi^{\pm}(x_h) \cap B_\sigma(x_h) \, .
    \]
    Otherwise, if $S_R \cap B_\sigma(x_h) \neq \emptyset$, then $S_R \cap B_\sigma(x_h) = \Pi(x_h) \cap B_\sigma(x_h)$, since the \emph{indecomposable stacking faults} in the decomposition of $S_R$ connect boundary points. By~\eqref{eq:Sw}, it follows that, up to $\H^1$-negligible sets, 
    \[
      S_{w} \cap B_\sigma(x_h) = (S_R \sm S_{v^{1/2}})  \cap B_\sigma(x_h) = (\Pi(x_h) \sm \Pi^{\pm}(x_h) )\cap B_\sigma(x_h) = \Pi^\mp(x_h) \cap B_\sigma(x_h) \, .
      \]
    In both cases, the condition in Definition~\ref{def:resolve} is satisfied, and we obtain the desired result.

\end{proof}

\begin{proposition} \label{prop:w given S}
  Let $\mu \in X_M(\Omega)$ and let $S \in \Ten(\mu,\Omega)$ \emph{resolve dislocations tension}. Let $v \in \D_M(\Omega)$ satisfy $J(v) = \pi \mu$. Then there exists $w \in \D_{M,\hor}^{1/2}(\Omega)$ such that $w^2 = v$ and $S_w = S$ up to an $\H^1$-negligible set. 
\end{proposition}
\begin{proof}
We apply Lemma~\ref{lem:lifting-v} to construct $\varphi$ and $v^{1/2}$. By the properties of $\Ten(\mu,\Omega)$, we have that 
\begin{equation} \label{eq:Sv12 S}
\tilde S = (S_{v^{1/2}} \sm S) \cup (S \sm S_{v^{1/2}})   \in \Ten(\mu,\Omega) \, .
\end{equation}
Both $S_{v^{1/2}}$ and $S$ \emph{resolve dislocations tension}. By Definition~\ref{def:resolve} there exists $\sigma > 0$ such that, up to $\H^1$-negligible sets, 
\[
\tilde S \cap B_\sigma(x_h) = \Pi(x_h) \cap B_\sigma(x_h) \quad \text{or} \quad   \tilde S \cap B_\sigma(x_h) = \emptyset \, , \quad \text{for } h = 1,\dots,M \, .
\]
This implies that, up to $\H^1$-negligible sets, $\tilde S = \bigcup_{j \in \N} (y_j^1, y_j^2)$, where $(y_j^1, y_j^2) \subset \Omega$ and $y_j^1, y_j^2 \in \de \Omega$.
 
We prove below that there exists $\chi \in SBV(\Omega;\{0,\pi\})$ such that $S_\chi = \tilde S$. Given such a $\chi$, we set $\psi := \frac{\varphi}{2} + \chi$ and $w := \exp(\iota \psi)$. By definition, $w^2 = v$. As in the proof of~\eqref{eq:Sw}, by~\eqref{eq:Sv12 S} we have that 
\[
S_w = (S_{v^{1/2}} \sm \tilde S) \cup (\tilde S \sm S_{v^{1/2}}) = S \, .
\]
This would conclude the proof. 

The rest of the proof is dedicated to the proof of existence of $\chi$, which can be interpreted as a ``two color theorem''. We stress that this proof require some care, as the topology of $\Omega$ might act as an obstruction: We need to exploit the simply connectedness of $\Omega$, as dropping this assumption allows one to exhibit counterexamples to the existence.\footnote{If $\Omega$ is convex, the proof can be simplified.}

We split the proof in steps. 

\begin{step}{1}{Segments disconnect $\Omega$} We claim that $\Omega \sm (y_j^1,y_j^2)$ is disconnected for every $j \in \N$. Let us argue by contradiction, assuming that $\Omega \sm (y_j^1, y_j^2)$ is connected. Let us fix $x_0 \in (y_j^1,y_j^2)$ and $\sigma > 0$ such that $B_\sigma(x_0) \subset \Omega$. Let $x_1, x_2 \in B_\sigma(x_0)$ be such that $x_1 \cdot e_2 > x_0 \cdot e_2$ and $x_2 \cdot e_2 < x_0 \cdot e_2$. Since we are assuming that $\Omega \sm (y_j^1,y_j^2)$ is connected (thus arc-connected), there exists a simple curve $\gamma \colon [0,1] \to \Omega \sm (y_j^1,y_j^2)$ such that $\gamma(0) = x_1$, $\gamma(1) = x_2$. Since $(\Omega \sm (y_j^1,y_j^2)) \sm [x_1,x_2]$ is connected too,\footnote{This is a consequence of an application of the Mayer-Vietoris sequence.} we can assume that $\gamma$ does not intersect the segment $[x_1,x_2]$. By concatenating $\gamma$ and $[x_1,x_2]$, we construct a simple loop $\tilde \gamma \colon \S^1 \to \Omega$. By Jordan's curve theorem, $\tilde \gamma$ disconnects $\R^2$ in two connected components. By construction, one of the two points $y_j^1$ and $y_j^2$ (say $y_j^1$) belongs to the bounded connected component, hence $\tilde \gamma \colon \S^1 \to \R^2 \sm \{y_j^1\}$ is in the nontrivial homotopy class. This contradicts the fact that $\tilde \gamma$ is homotopic in $\Omega$ to a point.
\end{step}

\begin{step}{2}{Segments disconnect $\Omega$ in exactly two connected components} Let us show that $\Omega \sm (y_j^1,y_j^2)$ consists of two connected components. Let us argue by contradiction, assuming that $\Omega \sm (y_j^1,y_j^2)$ consists of at least three connected components. There exist two points $x_1,x_2 \in \Omega \sm (y_j^1,y_j^2)$ belonging to two distinct connected components such that $x_1 \cdot e_2 > y_j^1 \cdot e_2$ and $x_2 \cdot e_2 < y_j^1 \cdot e_2$. Let $x_3 \in \Omega \sm (y_j^1,y_j^2)$ be a point in a third distinct connected component with $x_3 \cdot e_2 \neq y_j^1 \cdot e_2$. We assume without loss of generality that $x_3 \cdot e_2 > y_j^1 \cdot e_2$. By the path-connectedness of $\Omega$ there is a continuous curve $\gamma \colon [0,1] \to \Omega$ such that $\gamma(0) = x_1$, $\gamma(1) = x_3$. We observe that 
  \[
  \{ t \in [0,1] \text{ such that } \gamma(t) \in (y_j^1,y_j^2)  \}  \neq \emptyset \, ,
  \]
  since $x_1$ and $x_3$ belong to distinct connected components of $\Omega \sm (y_j^1,y_j^2)$. We let 
  \[
  t_1 := \inf   \{ t \in [0,1] \text{ such that } \gamma(t) \in (y_j^1,y_j^2)  \}  \, , \quad  t_2 := \sup   \{ t \in [0,1] \text{ such that } \gamma(t) \in (y_j^1,y_j^2)  \}  \, .
  \]
  Observe that $\gamma(t_1)$, $\gamma(t_2) \in (y_j^1,y_j^2)$. Fix $\eta > 0$ such that $[\gamma(t_1),\gamma(t_2)]+B_\eta \subset \Omega$. There exists $\delta > 0$ such that $\tilde x_1 := \gamma(t_1 - \delta)$, $\tilde x_2 := \gamma(t_2 + \delta) \in [\gamma(t_1),\gamma(t_2)]+B_\eta$ and $\tilde x_1 \cdot e_2 > y_j^1 \cdot e_2$ and $\tilde x_2 \cdot e_2 > y_j^1 \cdot e_2$. The segment $[\tilde x_1, \tilde x_2]$ is contained in $[\gamma(t_1),\gamma(t_2)]+B_\eta$. By concatenating $\gamma_{| [0,t_1-\delta]}$, $[\tilde x_1, \tilde x_2]$, and $\gamma_{| [t_2+\delta,1]}$ we construct a path in $\Omega \sm (y_j^1,y_j^2)$ connecting $x_1$ and $x_3$, contradicting the initial assumption.
  
\end{step}

\begin{step}{3}{Construction of $\chi$} We construct $\chi$ with an induction argument on the index $j$ of the collection of segments. We set $\chi_0 := 0$. Assume that we have defined $\chi_{j-1}$ for $j \geq 1$. By the previous step, we have that $\Omega \sm (y_j^1,y_j^2) = \Omega_-^j  \cup  \Omega_+^j$, where $\Omega_-^j$ ans $\Omega_+^j$ are the two distinct connected components. Then we set 
  \[
  \chi_j = \chi_{j-1} \mathds{1}_{\Omega_+^j} + (\pi - \chi_{j-1}) \mathds{1}_{\Omega_-^j} \, .  
  \]
  We observe that $\chi_j \in SBV(\Omega;\{0,\pi\})$ and $S_{\chi_j} = \bigcup_{i=1}^j (y_i^1,y_i^2)$. This is obtained by observing that for every $x \in \Omega \sm \bigcup_{i=1}^j (y_i^1,y_i^2)$ there exists $\sigma > 0$  such that $B_\sigma(x) \subset \Omega \sm \bigcup_{i=1}^j (y_i^1,y_i^2)$ and $\chi_j$ is constantly equal to either $0$ or $\pi$ in  $B_\sigma(x)$. If instead $x \in (y_i^1,y_i^2)$, there exists $\sigma > 0$  such that $B_\sigma(x)$ intersect only the segment $(y_i^1,y_i^2)$ and $\chi_j$ is equal to $0$ on one side $(y_i^1,y_i^2)$ of and $\pi$ on the other. Eventually, the function $\chi$ is defined as the weak* limit in $SBV$ of the sequence $\chi_j$. One sees that it satisfies $S_\chi = \bigcup_{j \in \N} (y_j^1,y_j^2)$ by noticing that for each $\Omega' \subset \subset \Omega$ and for $j$ large enough, the sequence $\chi_j$ restricted to $\Omega'$ is constant in $j$. 
  
\end{step}

\end{proof}

%We stress that the peculiar properties of $w \in \D_{M,\hor}^{1/2}(\Omega)$ guarantee the structure of $S_w$. In general, a $\H^1$-rectifiable set might not have the same structure, as the next example shows. 
\EEE
\section{$\Gamma$-limit of $F_\e^\p$}\label{sec:proof-main-field}

\subsection{Statement of the result} We are now in a position to state the following compactness and $\Gamma$-convergence result for the energies $F_\e^\p$ defined in~\eqref{def:F-part} after removing the logarithmic contribution of $M$ limiting singularities. It can be seen as a more general version of Theorem~\ref{thm:main2}. In fact here we consider limiting energies that still depend on the fields, in contrast with Theorem~\ref{thm:main2} where the energy is minimised in the fields for any given configuration of singularities. The result will thus be the key ingredient to prove Theorem~\ref{thm:main2} and is at the same time the second main result of this paper.

\smallskip
We recall that $\Omega\subset\R^2$ is an open, bounded, and simply connected subset of $\R^2$ with Lipschitz boundary and  $M\in\N$ is a fixed positive integer. 
For any two functions $w_1,w_2\in\D_{M,\hor}^{1/2}(\Omega)$ satisfying $w_1^2=w_2^2=:v$ we set
	\begin{equation}\label{def:gamma-limit-w}
	F^\p(w_1,w_2,\Omega)\defas \frac{M}{4}\gamma+\frac{1}{4}\WW(v,\Omega)+ \frac{\alpha}{2} \big(\H^1(S_{w_{1}})+\H^1(S_{w_{2}})\big)\,,  
	\end{equation}
where $\WW(v,\Omega)$ and $\gamma$ are as in~\eqref{def:ren-energy-v} and~\eqref{def:gamma}, respectively.
The meaning of the two configurations $w_1$, $w_2$ will become clear after the statement of the following compactness and $\Gamma$-convergence result.
\begin{theorem}\label{thm:main}
Let $F_\e^\p$ and $F^\p$ be as in~\eqref{def:F-part} and~\eqref{def:gamma-limit-w}, respectively. Then the following $\Gamma$-convergence result holds.
\begin{enumerate}[label=(\roman*)]
\item (Compactness) Suppose that $(u_\e)_\e$ is a sequence of configurations $u_\e\in\AD_\e$ satisfying
	\begin{equation}\label{uniformbound}
	\sup_{\e>0}\Big(F_\e^\p(u_\e,\Omega)-\frac{M}{4} \pi |\log\e|\Big)<+\infty\,.
	\end{equation}
Then, up to a subsequence, $\mu_{2u_\e}\mres\Omega\fto\mu$ for some $\mu=\sum_{h=1}^Nd_h\delta_{x_h}\in X(\Omega)$ with $|\mu|(\Omega)\leq M$. 
%with $N\in\N$, $d_h\in\Z\sm\{0\}$, $x_h\in\Omega$ pairwise different, and $\sum_{h=1}^N|d_h|\leq M$. 
Moreover, if $|\mu|(\Omega)=M$, then $N=M$ and $|d_h|=1$ for every $h\in\{1\,,\ldots ,N\}$ (\ie $\mu\in X_M(\Omega)$) and there exist $w_{\even}, w_{\odd}\in\D_{M,\hor}^{1/2}(\Omega)$ with $w_{\even}^2=w_{\odd}^2=:v$ and $J(v)=\pi\mu$ such that (up to further subsequences) $w_{2\e,s_0},w_{2\e,s_1}\to w_\even$ and $w_{2\e,s_2},w_{2\e,s_3}\to w_\odd$ in $L^1(\Omega;\R^2)$. 
\item (Lower bound) Let $w_{\rm even}, w_{\rm odd}\in \D_{M,\hor}^{1/2} (\Omega)$ with $w_{\rm even}^2=w_{\rm odd}^2=:v$ and let $u_\e\in\AD_\e$, $w_\e=\exp(2\pi\iota u_\e)\in \SF_\e$ be such that $\pi\mu_{2 u_\e}\mres\Omega\fto J(v)$, $w_{2\e,s_0},w_{2\e,s_1} \to w_{\rm even}$, and $w_{2\e,s_2},w_{2\e,s_3}\to w_{\rm odd}$ in $L^1(\Omega;\R^2)$. Then
	\begin{equation}\label{liminf:partial}
	\liminf_{\e\to 0}\Big(F_\e^\p(u_\e,\Omega)- \frac{M}{4}\pi|\log\e|\Big)\geq F^\p(w_\even,w_\odd,\Omega)\,. 
	\end{equation}
\item (Upper bound) Let $w_{\rm even}, w_{\rm odd}\in\D_{M,\hor}^{1/2}(\Omega)$ with $w_{\rm even}^2=w_{\rm odd}^2=:v$. Then there exist $u_\e\in \AD_\e$ such that $\pi\mu_{2 u_\e}\mres\Omega\fto J(v)$ and the sequence of spin fields $w_\e\defas \exp (2\pi\iota u_\e)\in\SF_\e$ satisfies $w_{2\e,s_0}, w_{2\e,s_1}\wsto w_{\rm even}$, $w_{2\e,s_2},w_{2\e,s_3}\wsto w_{\rm odd}$ in $L^1\big(\Omega;\R^2\big)$ and
	\begin{equation}\label{limsup:partial}
	\lim_{\e\to 0}\Big(F_\e^\p(u_\e,\Omega)- \frac{M}{4}\pi|\log\e|\Big) \leq F^\p(w_\even,w_\odd,\Omega)\,.
	\end{equation}
\end{enumerate}
\end{theorem}
Note that the presence of the two limiting configurations $w_\even$, $w_\odd$ is due to the separate compactness on the sublattices $\Z_{2\e}^\even$ and $\Z_{2\e}^\odd$, which in turn stems from the fact that next-to-nearest neighbour interactions decouple between those sublattices. Also observe that $w_\even$ and $w_\odd$ might be different in general and $S_{w_\even}$ and $S_{w_\odd}$ might differ as well (see Examples~\ref{ex:weven-wodd} and~\ref{ex:jumpset-we-wo}). Nevertheless, the presence of $w_\even$ and $w_\odd$ vanishes when the energies are minimised for a given configuration of singularities as in Theorem~\ref{thm:main2}. Indeed, we will see that $\F^\p(\mu,\Omega)=F^\p(w_\mu,w_\mu,\Omega)$ for a suitable function $w_\mu\in\D_{M,\hor}^{1/2}(\Omega)$ satisfying $w_\mu^2=v_\mu$ (see Remark~\ref{rem:energy-w-mu}).

\smallskip
We complete this section by showing that in general the two functions $w_\even$, $w_\odd$ obtained in the compactness statement Theorem~\ref{thm:main}(i) are different.

\begin{example}\label{ex:weven-wodd}
Let $\Omega=B_1$ and $\mu=\delta_0\in X_1(B_1)$. Let moreover $\theta_\hor^+\in C^\infty(\R^2\sm\Pi^+;[0,2\pi])$ be a lifting of $\frac{x}{|x|}$ (\ie $\frac{x}{|x|}=\exp\big(\iota\theta_\hor^+(x)\big)$. We extend $\theta$ to the whole $\R^2$ by setting $\theta_\hor^+(x)\defas\lim_{t\searrow 0}\theta_\hor^+(x+te_2)$ if $x\in\Pi^+\sm\{0\}$, $\theta_\hor^+(0)\defas 0$. Let $u_\e\in\AD_\e$ by given by
	\begin{equation*}
	u_\e(i)\defas
	\begin{cases}
	\frac{1}{4\pi}\theta_\hor^+(i) &\text{if}\ i\in 2\e\Z^2_\even\,,\\\cr
	\frac{1}{4\pi}\big(\theta_\hor^+(i)+2\pi\big) &\text{if}\ i\in 2\e\Z^2_\odd\,.
	\end{cases}
	\end{equation*}
We apply \eqref{eq:F-edge-F-part} and we estimate from above next-to-nearest neighbors interactions as follows: close to $\Pi^+$ with a constant; in an $\e$-neighborhood of 0 with a constant; far from $\Pi^+$ and outside the $\e$-neighborhood of 0 using $|\nabla \theta_\hor^+|^2 \sim \frac{1}{|x|^2}$. We obtain that 
\begin{equation}\label{est:Fe-example-0}
  \begin{split}
    & F_\e^\p(u_\e,B_1)\\
    & \quad \leq\frac{1}{4}F_\e^\edge\big(\tfrac{1}{2\pi}\theta_\hor^+,B_1\big)+\frac{C\e}{2}\#\big\{i\in\e\Z^2\cap   B_1   \colon [i,i+2\e e_2]\cap\Pi^+\neq\emptyset\big\} + C \e + C \e |\log \e| \,.
  \end{split}
	\end{equation}
  Moreover, using the fact that the jump set is horizontal gives
	\begin{equation}\label{est:Fe-example}
	F_\e^\edge\big(\tfrac{1}{2\pi}\theta_\hor^+,B_1\big) \leq F_\e^\screw\big(\tfrac{1}{2\pi}\theta_\hor^+,B_1\big)+C\,,
	\end{equation}
  where the constant $C$ accounts for interactions in an $\e$-neighborhood of 0.
  \EEE

Since $F_\e^\screw\big(\tfrac{1}{2\pi}\theta_\hor^+,B_1)\leq\pi|\log\e|+C$, we deduce from~\eqref{est:Fe-example} that the sequence $(u_\e)$ satisfies~\eqref{uniformbound}. Moreover, $\mu_{2u_\e}\mres B_1=\mu_{\frac{1}{2\pi}\theta_\hor^+}\mres B_1\fto\mu$.
Eventually, setting $w_\e\defas\exp(2\pi\iota u_\e)$ and applying standard interpolation estimates we find that
	\begin{equation*}
	w_{2\e,s_0},w_{2\e,s_1}\to \exp\Big(\iota\tfrac{\theta_\hor^+}{2}\Big)\quad\text{and}\quad w_{2\e,s_2},w_{2\e,s_3}\to \exp\Big(\iota\big(\tfrac{\theta_\hor^+}{2}+\pi\big)\Big)\;\text{ in $L^1(B_1)$ as $\e\to 0$.}
	\end{equation*}
Hence, $w_\even$ and $w_\odd$ are the two complex square roots of $\frac{x}{|x|}$ jumping in both cases across $\Pi^+$. However, also the limiting jump set may differ as the following example shows.
\end{example}
\begin{example}\label{ex:jumpset-we-wo}
Let $\Omega=B_1$, $\mu=\delta_0$ and $\theta_\hor^+$ be as in Example~\ref{ex:weven-wodd}. Let moreover $\theta_\hor^-\in C^\infty(\R^2\sm\Pi^-;[-\pi,\pi])$ be another lifting of $\frac{x}{|x|}$ jumping now across $\Pi^-$ and extended to $\Pi^-$ by setting $\theta_\hor^-(x)\defas\lim_{t\searrow 0}\theta_\hor^-(x+te_2)$ for $x\in\Pi^-\sm\{0\}$ and $\theta_\hor^-(0)\defas 0$. Let then $u_\e\in\AD_\e$ be given by
	\begin{equation*}
	u_\e(i)\defas
	\begin{cases}
	\frac{1}{4\pi}\theta_\hor^+(i) &\text{if}\ i\in 2\e\Z^2_\even\,,\\\cr
	\frac{1}{4\pi}\theta_\hor^-(i) &\text{if}\ i\in 2\e\Z^2_\odd\,.
	\end{cases}	
	\end{equation*}
Then $u_\e$ still satisfies the estimate~\eqref{est:Fe-example} with $\Pi^+$ replaced by $\Pi$ and as a consequence~\eqref{uniformbound} is fulfilled. Moreover, the sequence $w_\e\defas \exp(2\pi\iota u_\e)$ satisfies 
	\begin{equation*}
	w_{2\e,s_0}, w_{2\e,s_1}\to \exp\Big(\iota\tfrac{\theta_\hor^+}{2}\Big)\quad\text{and}\quad w_{2\e,s_2}, w_{2\e,s_3}\to \exp\Big(\iota\tfrac{\theta_\hor^-}{2}\Big)\;\text{ in $L^1(B_1)$ as $\e\to 0$.}
	\end{equation*}
In particular, we have $S_{w_\even}=\Pi^+$ and $S_{w_\odd}=\Pi^-$ up to an $\H^1$-negligible set.
\end{example}

%%%%%%%%%%%%%%%%%%%%%%%%%%%%%%%%%%%%%%%%%%%%%%%%%%%%
% Lower Bound
%%%%%%%%%%%%%%%%%%%%%%%%%%%%%%%%%%%%%%%%%%%%%%%%%%%%
\subsection{Proof of the Compactness and the Lower Bound}\label{sec:lb}
We are now in a position to prove Theorem~\ref{thm:main}(i) and (ii).  

\begin{proof}[Proof of Theorem~\ref{thm:main}(i) and (ii)]
  We divide the proof in several steps. 

  \begin{step}{1}{Compactness of $\mu_{2 u_\e}$} For $M\in\N$ fixed let $u_\e\in\AD_\e$ be a sequence of configurations satisfying~\eqref{uniformbound}. Then~\eqref{uniformbound} and~\eqref{eq:4F-part-larger-F-edge} imply that the sequence $(2u_\e)_\e$ satisfies~\eqref{uniformbound-F-edge}. Thus, from Theorem~\ref{thm:first-order-Fe}(i) we deduce that (up to a subsequence) $\mu_{2u_\e}\mres\Omega\fto\mu$ for some $\mu=\sum_{h=1}^Nd_h\delta_{x_h}\in X(\Omega)$ with $\sum_{h=1}^N|d_h|\leq M$. Let moreover $v_\e(i)\defas\exp(4\pi\iota u_\e(i))$ for $i\in\e\Z^2\cap\Omega$. Since $\Omega$ has Lipschitz boundary, as in~\cite[Remark 2]{AC09} we find an extension of $v_\e$ to $\e\Z^2$ (still denoted by $v_\e$) satisfying
	\begin{equation}\label{est:extension}
	\dsum{Q\in\Q_\e}{Q\cap\Omega\neq\emptyset}XY_\e(v_\e,Q)\leq CXY_\e(v_\e,\Omega)\leq C|\log\e|\,,
	\end{equation}
where the last estimate follows from~\eqref{uniformbound} together with~\eqref{eq:F-edge-F-part} and~\eqref{lb:edge-SD}.
It is not restrictive to assume that we still have $v_\e=\exp(4\pi\iota u_\e)$ on $\e\Z^2\sm\Omega$, since otherwise we could choose an arbitrary angular lifting $\varphi_\e$ of $v_\e$ and replace $u_\e$ by $\frac{\varphi_\e}{2}$ on $\e\Z^2\sm\Omega$. This would not change the energy $F_\e^\p(u_\e,\Omega)$ and thanks to Footnote~\ref{foot:conv-mu} would not affect the convergence properties of $\mu_{2u_\e}$.

\smallskip Let us now suppose that $\sum_{h=1}^N|d_h|=M$. Then Theorem~\ref{thm:first-order-Fe}(i) yields that $N=M$ and $|d_h|=1$ for every $h\in\{1\,,\ldots ,N\}$ and ensures the existence of $v\in\D_M(\Omega)$ with $J(v)=\pi\mu$ such that (up to a further subsequence) $\hat{v}_\e\wto v$ in $H^1_\loc\big(\Omega\sm\supp\mu;\R^2\big)$, where we recall that $\hat{v}_\e$ denotes that piecewise affine interpolation of $v_\e$ on $\T_\e$ (see Section~\ref{sec:interpolation}). Below we establish the existence of $w_{\even}, w_{\odd}\in\D_{M,\hor}^{1/2}(\Omega)$ satisfying $w_\even^2=w_\odd^2=v$.  
\end{step}

\begin{step}{2}{Common limit of dislocations defined on sublattices}
In this step we consider the measures $\mu^j_{2\e} := \mu_{2u_{2\e,s_j}}$ and we show that 
	\begin{equation}\label{conv:m2ue}
    \mu^j_{2\e} \mres\Omega\fto \mu\;\text{ for}\ j\in\{0\,,\ldots,3\}\,.
	\end{equation}
  To show this, from Remark~\ref{rem:F-edge-double}, and~\eqref{uniformbound}, we deduce that 
	\begin{equation}\label{est:F-edge-shift-uniform}
	\sup_{\e>0}\frac{1}{|\log 2\e|}F_{2\e, s_j}^{\edge}(2u_{2\e,s_j},A)<+\infty\;\text{ for every}\ j\in\{0\,,\ldots ,3\}\,.
	\end{equation}  
Theorem~\ref{thm:zero-order} yields that, up to a subsequence, $\mu_{2\e}^{j}\mres\Omega\fto\mu^{j}$ for some $\mu^{j}\in X(\Omega)$. It remains to show that $\mu^{j}=\mu$ for every $j\in\{0\,,\ldots ,3\}$. Thanks to~\cite[Proposition 5.2]{ACP11} the claim follows if we show that 
	\begin{equation}\label{conv-diff-jac}
	\big(J(\hat{v}_\e)-J(\hat{v}_{2\e,s_j})\big)\mres\Omega\fto 0\;\text{ for every}\ j\in\{0\,,\ldots,3\}\,.
	\end{equation}
To prove this, we shall apply Lemma~\ref{lem:conv-diff-jac}. We observe that~\eqref{eq:gradient-affine} together with~\eqref{est:extension} yields
	\begin{equation*}\label{est:grad-v}
	\int_{\Omega}|\nabla\hat{v}_\e|^2\dx\leq \dsum{Q\in\Q_\e}{Q\cap\Omega\neq\emptyset}XY_\e(v_\e,Q)\leq C|\log\e|\,.
	\end{equation*}
Since the analogue of~\eqref{est:extension} holds with $\Q_\e$ replaced by $\Q_{2\e,s_j}$, we obtain in a similar way that $\int_{\Omega}|\nabla\hat{v}_{2\e,s_j}|^2\dx\leq C|\log\e|$.
Moreover, using the explicit expression of $\hat{v}_\e$, $\hat{v}_{2\e,s_j}$ one gets  
	\begin{equation}\label{est:difference-L2-shift}
	\int_{\Omega}|\hat{v}_\e-\hat{v}_{2\e,s_j}|^2\dx\leq C\e^2 \dsum{Q\in\Q_{\e}}{Q\cap\Omega\neq\emptyset}XY_\e(v_\e,Q)\leq C\e^2|\log\e|\,.
	\end{equation}
Combining the above estimates yields 
	\begin{equation*}
	\|\hat{v}_\e-\hat{v}_{2\e,s_j}\|_{L^2(\Omega)}\Big(\|\nabla\hat{v}_\e\|_{L^2(\Omega)}+\|\nabla\hat{v}_{2\e,s_j}\|_{L^2(\Omega)}\Big)\leq C\e|\log\e|\; \to 0\;\text{ as}\ \e\to 0\,,
	\end{equation*}
hence~\eqref{conv-diff-jac} follows from Lemma~\ref{lem:conv-diff-jac}.
\end{step}

\begin{step}{3}{Weak* compactness on sublattices}
  First of all, we observe that $\sup_{\e > 0} \|w_\e\|_{L^\infty} \leq 1$, hence there exists a subsequence (not relabeled) such that for every $j \in \{0,\dots,3\}$ we have the following convergence for the restrictions on the doubled-spaced lattices:
  \[
  w_{2\e, s_j} \wsto w_j \, , \quad \text{weakly* in } L^\infty(\Omega; \R^2) \, .
  \]
  We will improve this convergence in the next steps. 
\end{step}
  
\begin{step}{4}{Lower bound by weak-membrane energy}  
  We fix $R > 1$ and $\sigma > 0$ satisfying~\eqref{cond:sigma}. By Lemma~\ref{lem:Fpart-XYgen-WM} we infer that 
  \begin{equation}\label{eq:4F-part-larger-WM}
    F^\p_\e(u_\e,\Omega^\sigma(\mu))\geq\frac{1}{4}\sum_{j=0}^3 W\!M_{2\e,s_j}^{R,\alpha}(w_{2\e, s_j},\Omega^\sigma(\mu)) + o(1) \, ,
  \end{equation}
  where $o(1) \to 0$ as $\e \to 0$. 
  Using~\eqref{eq:4F-part-larger-F-edge}, we estimate
	\begin{equation}\label{est:liminf-splitting}
	4F_\e^\p(u_\e,\Omega)-M\pi|\log\e| \geq\sum_{h=1}^M \Big(F_\e^\edge\big(2u_\e,B_{\sigma}(x_h)\big)-\pi|\log\e|\Big)+4 F_\e^\p\big(u_\e,\Omega^{\sigma}(\mu)\big)\,.
	%&\hspace*{1em}+\frac{1}{4}\sum_{j=0}^3 W\!M_{2\e}^{R,s_j}\big(w_\e,\Omega^{\sigma_k}(\mu)\big)
	%&\geq \frac{M}{4}\overline{C}+F_\e^\p\big(u_\e,\Omega^{\sigma_k}(\mu)\big)-\frac{M\pi}{4}|\log\sigma_k|-1\,,
	\end{equation} 
Note that $\mu_{2u_\e}\mres B_{\sigma}(x_h)\fto d_h\delta_{x_h}$. Thus, applying Theorem~\ref{thm:first-order-Fe}(ii) locally on $B_{\sigma}(x_h)$ yields
	\begin{equation}\label{est:liminf-core}
	\begin{split}
	\liminf_{\e\to 0}\Big(F_\e^\edge\big(2u_\e,B_{\sigma}(x_h)\big)-\pi|\log\e|\Big) &\geq\gamma+\W\big(d_h\delta_{x_h},B_{\sigma}(x_h)\big)\\
	&=\gamma-\pi|\log\sigma|\;\text{ for every}\ h\in\{1\,,\ldots, M\}\,,
	\end{split}
	\end{equation}
where the last equality follows from~\eqref{eq:ren-energy-ball}.
  It follows that 
  \[
    F_\e^\p(u_\e, \Omega^\sigma(\mu)) + \frac{\gamma}{4} - \frac{\pi}{4} |\log \sigma| \leq F_\e^\p(u_\e, \Omega) - \frac{M}{4} \pi |\log \e|  + o(1) \, ,
  \]
  where $o(1) \to 0$ as $\e \to 0$. In particular, by~\eqref{uniformbound}, we deduce that
  \[ 
    \sup_{\e > 0} F_\e^\p(u_\e, \Omega^\sigma(\mu)) < + \infty \, .
  \]
  Then~\eqref{eq:4F-part-larger-WM} yields that
  \begin{equation} \label{eq:bound on WM}
    \sup_{\e > 0} W\!M_{2\e,s_j}^{R,\alpha}(w_{2\e,s_j},\Omega^\sigma(\mu)) < + \infty\,, \quad \text{for } j \in \{0,\dots,3\} \,.
  \end{equation}
\end{step}

\begin{step}{5}{Compactness of $w_{2\e,s_j}$}
  Exploiting~\eqref{eq:bound on WM} and the equiboundedness of the $L^\infty$ norms of $w_{2\e,s_j}$, by Theorem~\ref{thm:WM compactness} we obtain that, for every $\sigma > 0$, $w_{2\e,s_j} \to w_j$ strongly in $L^1(\Omega^\sigma(\mu);\R^2)$ and $w_j \in SBV^2(\Omega^\sigma(\mu);\R^2)$, for $j = 0,\dots,3$, where the $w_j$'s are the weak* limits obtained above. By the arbitrariness of $\sigma$ and the equiboundedness of the $L^\infty$ norms of $w_{2\e,s_j}$, we conclude that $w_{2\e,s_j} \to w_j$ strongly in $L^1(\Omega;\R^2)$ and $w_j \in SBV_\loc^2(\Omega\sm\supp\mu;\R^2)$.
\end{step}

\begin{step}{6}{Identification of $w_\even$ and $w_\odd$}
We show that
\begin{equation} \label{eq:w-even-w-odd}  
\| w_{2\e,s_0}- w_{2\e,s_1}\|_{L^1(\Omega)}\to 0\qquad\text{and}\qquad\|w_{2\e,s_2}-w_{2\e,s_3}\|_{L^1(\Omega)}\to 0\,.
\end{equation}  
From~\eqref{eq:w-even-w-odd} and the convergence proved in Step~5, we deduce that $w_0 = w_1 =: w_{\even}$ and $w_2 = w_3 =: w_{\odd}$. 

We only establish the first convergence in~\eqref{eq:w-even-w-odd}, since the second one follows by a similar argument. Fix $\Omega' \subset \subset \Omega$; given $i_0 \in 2 \e \Z^2$, such that $Q_{2\e}(i_0) \cap \Omega' \neq \emptyset$. We define the rectangles $R^\pm_\e(i_0)\defas Q_{2\e}(i_0)\cap Q_{2\e}(i_0 \pm \e e_1)$ and observe that
\[
\| w_{2\e,s_0}- w_{2\e,s_1} \|_{L^2(R^\pm_\e(i_0))}^2 = 2 \e^2 |w_{\e}(i_0) - w_{\e}(i_0 \pm \e e_1)|^2  \, .
\]
It follows that 
\[
\| w_{2\e,s_0}- w_{2\e,s_1} \|_{L^2(Q_{2\e}(i_0))}^2 = 2 \e^2 |w_{\e}(i_0) - w_{\e}(i_0 + \e e_1)|^2  + 2 \e^2 |w_{\e}(i_0) - w_{\e}(i_0 - \e e_1)|^2  \, .
\]
Summing over all $i_0$ as above and noticing that we have that $Q_{2\e}(i_0), Q_{2\e}(i_0 \pm \e e_1) \subset \Omega$ for $\e$ small enough, we deduce that 
\[
  \| w_{2\e,s_0}- w_{2\e,s_1} \|_{L^2(\Omega')}^2 \leq C \e^2 F^\p_\e(u_\e,\Omega) \leq C \e^2 |\log \e| \to 0 \quad \text{as } \e \to 0 \, .
\]
By the boundedness in $L^\infty$ of $w_{2\e,s_j}$, we deduce that $\| w_{2\e,s_0}- w_{2\e,s_1} \|_{L^2(\Omega)} \to 0$ and, in particular, the first convergence in~\eqref{eq:w-even-w-odd}.
\end{step}

\begin{step}{7}{$w_\even$ and $w_\odd$ belong to $\D_{M,\hor}^{1/2}(\Omega)$}
First of all, we show that $w_\even^2 = w_\odd^2 = v  \in \D_{M}(\Omega)$, where $v$ is the limit found in Step 1. For this, we need to compare the piecewise affine interpolations $\hat v_{2\e,s_j}$ and the piecewise constant functions $v_{2\e, s_j} := w_{2\e,s_j}^2$. Let us fix $\Omega' \subset \subset \Omega$. Given $i_0 \in 2 \e \Z^2$, such that $Q_{2\e}(i_0) \cap \Omega' \neq \emptyset$, we have that 
\[
 \begin{split}
	& \| \hat v_{2 \e, s_j} - v_{2 \e, s_j} \|_{L^2(Q_{2\e}(i_0))}^2  \\
  & \quad \leq  C \e^2  \Big(|v_\e(i_0+2\e e_1)-v_\e(i_0)|^2+|v_\e(i_0+2\e(e_1+e_2))-v_\e(i_0+2\e e_1)|^2\\
	& \hphantom{\quad \leq  C \e^2  \Big(  }+|v_\e(i_0+2\e e_2)-v_\e(i_0)|^2+|v_\e(i_0+2\e(e_1+e_2))-v_\e(i_0+2\e e_2)|^2\Big)\,. 
	\end{split}
\]
Since $Q_{2\e}(i_0) \subset \Omega$ for $\e$ small enough, by~\eqref{est:extension} it follows that 
\[
  \| \hat v_{2 \e, s_j} - v_{2 \e, s_j} \|_{L^2(\Omega')}^2 \leq C \e^2 XY_\e(v_\e, \Omega) \leq C \e^2 |\log \e| \to 0 \quad \text{as } \e \to 0 \, .
\]
By the boundedness in $L^\infty$ of $v_{2\e,s_j}$ and $\hat v_{2\e,s_j}$, we deduce that $\| \hat v_{2\e,s_j}- v_{2\e,s_j} \|_{L^2(\Omega)} \to 0$. Hence, $\| v_{2\e,s_j} - v \|_{L^2(\Omega)} \to 0$. Combining this with the convergences $\|w_{2\e,s_0}^2 - w_\even^2 \|_{L^1(\Omega)} \to 0$ and $\|w_{2\e,s_2}^2 - w_\odd^2 \|_{L^1(\Omega)} \to 0$, we conclude that $w_\even^2 = w_\odd^2 = v$.

It remains to show that $\H^1(S_{w_\even} \cap \Omega) < +\infty$, $\H^1(S_{w_\odd} \cap \Omega) < +\infty$, and
	\begin{equation}\label{cond:jumpset}
	|\nu_{w_\even}\cdot e_1|=0\;\text{ $\H^1$-a.e.\! on $S_{w_\even}$,}\qquad |\nu_{w_\odd}\cdot e_1|=0\;\text{ $\H^1$-a.e.\! on $S_{w_\odd}$.}
	\end{equation} 
In doing so we will essentially establish the liminf inequality. 
Let us fix $\sigma > 0$ such that~\eqref{cond:sigma} is satisfied. (At the end of the step, we shall let $\sigma \to 0$.) 
% Using~\eqref{eq:4F-part-larger-F-edge}, we estimate
% 	\begin{equation}\label{est:liminf-splitting}
% 	4F_\e^\p(u_\e,\Omega)-M\pi|\log\e| \geq\sum_{h=1}^M \Big(F_\e^\edge\big(2u_\e,B_{\sigma}(x_h)\big)-\pi|\log\e|\Big)+4 F_\e^\p\big(u_\e,\Omega^{\sigma}(\mu)\big)\,.
% 	%&\hspace*{1em}+\frac{1}{4}\sum_{j=0}^3 W\!M_{2\e}^{R,s_j}\big(w_\e,\Omega^{\sigma_k}(\mu)\big)
% 	%&\geq \frac{M}{4}\overline{C}+F_\e^\p\big(u_\e,\Omega^{\sigma_k}(\mu)\big)-\frac{M\pi}{4}|\log\sigma_k|-1\,.
% 	\end{equation} 
% Note that $\mu_{2u_\e}\mres B_{\sigma}(x_h)\fto d_h\delta_{x_h}$. 
% Thus, applying Theorem~\ref{thm:first-order-Fe}(ii) locally on $B_{\sigma}(x_h)$ yields
% 	\begin{equation}\label{est:liminf-core}
% 	\begin{split}
% 	\liminf_{\e\to 0}\Big(F_\e^\edge\big(2u_\e,B_{\sigma}(x_h)\big)-\pi|\log\e|\Big) &\geq\gamma+\W\big(d_h\delta_{x_h},B_{\sigma}(x_h)\big)\\
% 	&=\gamma-\pi|\log\sigma|\;\text{ for every}\ h\in\{1\,,\ldots, M\}\,.
% 	\end{split}
% 	\end{equation}
Since $w_{2\e,s_0},w_{2\e,s_1} \to w_\even$ in $L^1(\Omega^{\sigma}(\mu);\R^2)$, we infer from Lemma~\ref{lem:Fpart-XYgen-WM} together with Theorem~\ref{thm:WM} that
	\begin{equation}\label{est:liminf-wm}
	\begin{split}
	& \liminf_{\e\to 0} F_\e^\p\big(u_\e,\Omega^{\sigma}(\mu)\big) \geq\frac{1}{4}\sum_{j=0}^3 \liminf_{\e\to 0}  W\!M_{2\e,s_j}^{R,\alpha} \big(w_{2\e,s_j},\Omega^{\sigma}(\mu)\big)\\
	&\geq   \frac{1}{4}  \int_{\Omega^{\sigma}(\mu)}|\nabla w_\even|^2+|\nabla w_\odd|^2\dx +\frac{1}{2}\int_{S_{w_\even}\cap\Omega^{\sigma_k}(\mu)}\hspace*{-0.5em}  \big(R|\nu_{w_\even}\cdot e_1| +  \alpha  |\nu_{w_\even}\cdot e_2|\big)\dH  \\
	& \hphantom{ \geq \int_{\Omega^{\sigma}(\mu)}|\nabla w_\even|^2+|\nabla w_\odd|^2\dx } +\frac{1}{2}\int_{S_{w_\odd}\cap\Omega^{\sigma}(\mu)}\hspace*{-0.5em} \big(R|\nu_{w_\odd}\cdot e_1| +   \alpha   |\nu_{w_\odd}\cdot e_2|\big)\dH   \\
  & \geq  \frac{1}{8}  \int_{\Omega^{\sigma}(\mu)}|\nabla v|^2 \dx +\frac{1}{2}\int_{S_{w_\even}\cap\Omega^{\sigma}(\mu)}\hspace*{-0.5em}  \big(R|\nu_{w_\even}\cdot e_1| +  \alpha  |\nu_{w_\even}\cdot e_2|\big)\dH \\
  &  \hphantom{ \geq \frac{1}{2}\int_{\Omega^{\sigma}(\mu)}|\nabla v|^2 \dx } +\frac{1}{2}\int_{S_{w_\odd}\cap\Omega^{\sigma}(\mu)}\hspace*{-0.5em}  \big(R|\nu_{w_\odd}\cdot e_1| +  \alpha  |\nu_{w_\odd}\cdot e_2|\big)\dH  \, , 
 	\end{split}
	\end{equation}
for every $R>0$, where we have used that $|\nabla w_\even|^2 = |\nabla w_\odd|^2 = \frac{1}{4}|\nabla v|^2$ since $w_\even^2 = w_\odd^2 = v$ in the sense of complex squares. 
Combining~\eqref{est:liminf-splitting}, \eqref{est:liminf-core}, and~\eqref{est:liminf-wm} and using the definition of $\WW$ in~\eqref{def:ren-energy-v}  thus gives 
	\begin{equation}\label{est:liminf-almost-final}
	\begin{split}
	& \liminf_{\e\to 0}\Big(4F_\e^\p(u_\e,\Omega)-M\pi|\log\e|\Big) \\
  &\geq M\gamma+\WW(v,\Omega)
	%\frac{1}{2}\int_{\Omega^{\sigma_k}(\mu)}|\nabla v|^2\dx-M\pi|\log\sigma_k|\\
	+  2   \int_{S_{w_\even}\cap\Omega^\sigma(\mu)}\hspace*{-0.5em}  \big(R|\nu_{w_\even}\cdot e_1| +   \alpha   |\nu_{w_\even}\cdot e_2|\big)  \dH\\
	& \hphantom{\ \geq M\gamma+\WW(v,\Omega)}+ 2   \int_{S_{w_\odd}\cap\Omega^\sigma(\mu)}\hspace*{-0.5em} \big(R|\nu_{w_\odd}\cdot e_1| +   \alpha |\nu_{w_\odd}\cdot e_2|\big)   \dH+r(\sigma) 
	\end{split}
	\end{equation}
with $r(\sigma)\to 0$ as $\sigma \to 0$. First of all, together with the bound~\eqref{uniformbound} this implies that 
\[
  \begin{split}
    & \sup_{\sigma > 0} \Big( \int_{S_{w_\even}\cap\Omega^\sigma(\mu)}\hspace*{-0.5em}  \big(R|\nu_{w_\even}\cdot e_1| +   \alpha   |\nu_{w_\even}\cdot e_2|\big) \dH \Big) < +\infty  \\
    & \sup_{\sigma > 0} \Big(  \int_{S_{w_\odd}\cap\Omega^\sigma(\mu)}\hspace*{-0.5em} \big(R|\nu_{w_\odd}\cdot e_1| +   \alpha |\nu_{w_\odd}\cdot e_2|\big)   \dH \Big) < +\infty \, ,
  \end{split}
\]
hence $\H^1(S_{w_\even}\cap\Omega)=\H^1(S_{w_\even}\cap\bigcup_{\sigma}\Omega^{\sigma}(\mu))<+\infty$ and similarly for $S_{w_\odd}$.
Passing to the limit as $\sigma \to 0$ in~\eqref{est:liminf-almost-final} and taking the supremum over $R>0$, we then deduce from~\eqref{uniformbound} that
	\begin{equation*}
	\sup_{R>0}R\bigg(\int_{S_{w_\even}\cap\Omega}|\nu_{w_\even}\cdot e_1|\dH+\int_{S_{w_\odd}\cap\Omega}|\nu_{w_\odd}\cdot e_1|\dH\bigg)\leq C\,,
	\end{equation*}
which is only possible if~\eqref{cond:jumpset} holds.

% Moreover, since $\H^1(S_{w_\even}\cap\Omega)=\H^1(S_{w_\even}\cap\bigcup_{\sigma}\Omega^{\sigma}(\mu))<+\infty$ and similarly for $S_{w_\odd}$, we know that
% \begin{equation*}
% \sum_{h=1}^M\Big(\H^1(S_{w_\even}\cap B_{\sigma}(x_h))+\H^1(S_{w_\odd}\cap B_{\sigma}(x_h))\Big)\ \to 0\;\text{ as}\ \sigma \to 0\,.
% \end{equation*}
\end{step}

\begin{step}{8}{Liminf inequality}
  To conclude, it suffices to observe that for any $w\in\D_{M,\hor}^{1/2}(\Omega)$ we have $|\nu_w\cdot e_1|=0$ and hence $|\nu_{w}\cdot e_2|=1$ $\H^1$-a.e.\! on $S_w$. Thus~\eqref{liminf:partial} follows from~\eqref{est:liminf-almost-final} by letting $\sigma \to 0$.
  \end{step}
\end{proof}

We have established the proof of Theorem~\ref{thm:main}(i) and (ii). The proof of Theorem~\ref{thm:main}(iii) will heavily rely on a suitable characterisation of the core contribution $\gamma$ in terms of asymptotic minimisation problems involving the whole energy $F_\e^\p$. This is the content of the following subsection.
%%%%%%%%%%%%%%%%%%%%%%%%%%%%%%%%%%%%%%%
%            Core energy
%%%%%%%%%%%%%%%%%%%%%%%%%%%%%%%%%%%%%%%
\subsection{The Core energy}
In this subsection we will show that the core energy $\gamma$ in~\eqref{def:gamma} can be equivalently obtained as a double limit of minimisation problems involving the energy $F_\e^\p$. Namely, we fix an angular lifting $\theta_\hor\in C^\infty(\R^2\sm\Pi)$ of $x/|x|$ and for every $x_0\in\R^2$, $\e>0$, and $\sigma>8\e$ we set
	\begin{equation}\label{def:core-part}
	\gamma_\e^\p\big(B_\sigma(x_0)\big)\defas\min\Big\{F_\e^\p\big(u,B_\sigma(x_0)\big)\colon 2 u(i)=\frac{1}{2\pi}\theta_\hor(i-x_0) \;\text{ for }\;i\in\partial_{2\e} B_\sigma(x_0)\Big\}\,,
	\end{equation}
where with a slight abuse of notation we have set 
	\begin{equation*}
	\partial_{2\e}B_\sigma(x_0)\defas\e\Z^2\cap\bigcup_{j=0}^3\partial \Big(\bigcup_{Q\in \Q_{2\e,s_j}(B_\sigma(x_0))}Q\Big)\,. 
	\end{equation*}
Note that the boundary conditions are prescribed in the larger boundary layer $\partial_{2\e}B_\sigma(x_0)$, since we need to include next-to-nearest neighbour interactions.

\begin{remark} \label{rem:invariance rotation}
  Since the energy is invariant under translations (it depends on $\d^e u$), for any $a \in \R$ we have that 
  \[
    \gamma_\e^\p\big(B_\sigma(x_0)\big) = \min\Big\{F_\e^\p\big(u,B_\sigma(x_0)\big)\colon 2 u(i)=\frac{1}{2\pi}(\theta_\hor(i-x_0) + a) \;\text{ for }\;i\in\partial_{2\e} B_\sigma(x_0)\Big\}\, .
  \]
  We observe that $\theta_\hor(x-x_0) + a$ is a lifting of a rotation of $\frac{x-x_0}{|x-x_0|}$.
\end{remark}

The following holds true.
\begin{proposition}\label{prop:core-partials}
For $x_0\in\R^2$, $\e>0$, and $\sigma>8\e$ let $\gamma_\e^\p$ be as in~\eqref{def:core-part}. Let moreover $\gamma$ be as in~\eqref{def:gamma}; then we have
	\begin{equation}\label{eq:gamma-part}
	\gamma=\lim_{\sigma\to 0}\limsup_{\e\to 0}\Big(4\gamma_\e^\p\big(B_\sigma(x_0)\big)-\pi\log\frac{\sigma}{\e}\Big)=\lim_{\sigma\to 0}\liminf_{\e\to 0}\Big(4\gamma_\e^\p\big(B_\sigma(x_0)\big)-\pi\log\frac{\sigma}{\e}\Big)
	\end{equation}
for every $x_0\in\R^2$.
\end{proposition}

\begin{remark}
  We stress that in~\eqref{eq:gamma-part} the quantities $\limsup_{\e\to 0}\big(4\gamma_\e^\p\big(B_\sigma(x_0)\big)-\pi\log\frac{\sigma}{\e}\big)$ and $\liminf_{\e\to 0}\big(4\gamma_\e^\p\big(B_\sigma(x_0)\big)-\pi\log\frac{\sigma}{\e}\big)$ may depend on $\sigma$, thus the limit as $\sigma \to 0$ is necessary. This is in contrast to the formula defining $\gamma$ in~\eqref{def:gamma}, where the limit as $\e \to 0$ is independent of~$\sigma$. 
\end{remark}

A crucial ingredient for the proof of Proposition~\ref{prop:core-partials} is the following theorem contained in~\cite[Theorem 2.3 and Remark 2.4]{GPS}, which shows that the minimisation problem defining $\gamma_\e^\screw$ admits a solution without dipoles.
\begin{theorem}\label{thm:core-screw-dipoles}
Let $x_0\in\R^2$, $\e>0$, and $\sigma>4\e>0$; let moreover $\theta$ be a lifting of $x/|x|$ in the cut domain $\R^2\sm\Pi^+$. Then there exists $\zeta_{x_0,\e}^\sigma\in \AD_\e$ satisfying  
\begin{enumerate}[label=($\gamma$\arabic*)]
\item \label{bc-min-core} $\zeta_{x_0,\e}^\sigma =\frac{1}{2\pi}\theta(\cdot-x_0)$ on $\partial_\e B_\sigma(x_0)$
\item \label{vorticity-min-core} $\mu_{\zeta_{x_0,\e}^\sigma} \mres B_\sigma(x_0)= \delta_{\overline{x}_\e} $  with $\overline{x}_\e = b(Q_\e(\overline{\imath}_\e  ))$  for some $ \overline{\imath}_\e = \overline{\imath}_\e (x_0,\e,\sigma)\in\e\Z^2\cap B_\sigma(x_0)$
\end{enumerate}
and such that
	\begin{equation}\label{eq:solution-gamma-screw}
	F_\e^\screw\big(\zeta_{x_0,\e}^\sigma,B_\sigma(x_0)\big)=\gamma_\e^\screw\big(B_\sigma(x_0)\big)\,.
	\end{equation}
In particular, $\zeta_{x_0,\e}^\sigma$ is a solution to~\eqref{def:core-screw}.
\end{theorem}

We are now in a position to prove Proposition~\ref{prop:core-partials}. 

\begin{proof}[Proof of Proposition~\ref{prop:core-partials}]
Let $x_0\in\R^2$ be arbitrary and suppose that $u_{\e,\sigma}$ is a competitor for $\gamma_\e^\p\big(B_\sigma(x_0)\big)$ as in~\eqref{def:core-part}. Then $2u_{\e,\sigma}$ is a competitor for $\gamma_\e^\screw\big(B_\sigma(x_0)\big)$ in~\eqref{def:core-screw}. Hence, passing to the infimum, we immediately deduce from~\eqref{eq:4F-part-larger-F-edge} and~\eqref{lb:edge-SD} that $\gamma_\e^\p\big(B_\sigma(x_0)\big)\geq\frac{1}{4}\gamma_\e^\screw\big(B_\sigma(x_0)\big)$, which by~\eqref{def:gamma} in turn implies that
	\begin{equation}\label{est:core-lb}
	\liminf_{\sigma\to 0}\liminf_{\e\to 0}\Big(4\gamma_\e^\p\big(B_\sigma(x_0)\big)-\pi\log\frac{\sigma}{\e}\Big)\geq \gamma\,.
	\end{equation}
Thus, to obtain~\eqref{eq:gamma-part} it suffices to show that
	\begin{equation}\label{est:core-ub}
	\limsup_{\rho\to 0}\limsup_{\e\to 0}\Big(4\gamma_\e^\p\big(B_\rho(x_0)\big)-\pi\log\frac{\rho}{\e}\Big)\leq \gamma\,.
	\end{equation}
%In order to simplify the notation, we only prove~\eqref{est:core-ub} for $x_0=0$. 
We only prove~\eqref{est:core-ub} in the case $\theta_\hor\in C^\infty(\R^2\sm\Pi^+)$, then the general case follows as described in Remark~\ref{rem:core-partials} below.  
To establish~\eqref{est:core-ub}, let $\e>0$, let $\sigma>8\e$, let $\rho> 6\sigma$, and let $ \zeta_{x_0,\e}^\sigma\in\AD_\e$ be as in Theorem~\ref{thm:core-screw-dipoles} satisfying, in particular,
	\begin{equation}\label{eq:energy-core-small}
	F_\e^\screw\big(\zeta_{x_0,\e}^\sigma, B_\sigma(x_0)\big)=\gamma_\e^\screw\big(B_\sigma(x_0)\big)\,.
	\end{equation}
Below we suitably  modify  $\zeta_{x_0,\e}^\sigma$ to obtain a competitor $u_{\e,\rho}$ for $\gamma_\e^\p\big(B_{\rho}(x_0)\big)$ satisfying
	\begin{equation}\label{est:energy-core-final}
	4 F_\e^\p\big(u_{\e,\rho},B_\rho(x_0)\big)\leq \gamma_\e^\screw\big(B_\sigma(x_0)\big)+\pi\log\frac{\rho}{\sigma}+r(\e,\sigma,\rho)
	\end{equation}
with $r(\e,\sigma,\rho)\to 0$ when letting first $\e\to 0$, then $\sigma\to 0$, and eventually $\rho\to 0$. Suppose for the moment that~\eqref{est:energy-core-final} holds true. Since $u_{\e,\rho}$ is a competitor for $\gamma_\e^\p\big(B_\rho(x_0)\big)$ this implies that
	\begin{equation*}
	4\gamma_\e^\p\big(B_\rho(x_0)\big)-\pi\log\frac{\rho}{\e}\leq\gamma_\e^\screw\big(B_\sigma(x_0)\big)+\pi\log\frac{\rho}{\sigma}-\pi\log\frac{\rho}{\e}+r(\e,\sigma,\rho)\,.
	\end{equation*}
Thanks to~\eqref{def:gamma} we find~\eqref{est:core-ub} by letting in order $\e\to 0$, $\sigma\to 0$, and $\rho\to 0$.

By Remark~\ref{rem:invariance rotation} it suffices to construct $u_{\e,\rho}$ such that $2 u(i) = \frac{1}{2\pi} ( \theta_\hor(i-x_0) + a_{\e,\rho})$ and~\eqref{est:energy-core-final} is satisfied, where $a_{\e,\rho}$ is a constant.  

\smallskip
We establish~\eqref{est:energy-core-final} in several steps.

\begin{step}{1}{Construction of the competitor $u_{\e,\rho}$ in $B_{\frac{\rho}{2}}(x_0)$}
Given $\zeta_{x_0,\e}^\sigma $ as above we define $u_{\e,\rho}$ in the ball $B_{\frac{\rho}{2}}(x_0)$ via a shifting and smoothing procedure. 
%Let $\e>0$, let $\sigma'>2\sigma>2\e$, and let $u_{x_0,\e}^\sigma$ be as in Theorem~\ref{thm:core-screw-dipoles}. Let moreover $\overarc{v}_{x_0,\e}^\sigma$ be the $\S^1$-valued interpolation of $u_{x_0,\e}^\sigma$ as in Lemma~\ref{lem:S1-interpolation}.
Recall that $\mu_{\zeta_{x_0,\e}^\sigma} \mres B_{\sigma}(x_0)= \delta_{\overline{x}_\e}$ with $ \overline{x}_\e =\overline{\imath}_\e  +\frac{\e}{2}(e_1+e_2)$ for some $\overline{\imath}_\e=\overline{\imath}_\e(x_0,\e,\sigma)\in \e \Z^2 \cap B_\sigma(x_0)$ (see Theorem~\ref{thm:core-screw-dipoles}, \ref{vorticity-min-core}). Let now
	\begin{equation}\label{def:integer-part}
	x_0^\e\defas\e\Big\lfloor\frac{x_0\cdot e_1}{\e}\Big\rfloor e_1+\e\Big\lfloor\frac{x_0\cdot e_2}{\e}\Big\rfloor e_2\qquad\text{and}\qquad z_\e\defas \big( (\overline{\imath}_\e-x_0^\e)\cdot e_2 \big) e_2 \in \e\Z e_2 
	\end{equation}	 
be the component-wise lower integer part of $x_0$ and the projection of its deviation from $\overline{\imath}_\e$ onto the vertical axis, respectively, see Figure~\ref{fig:shifting-smoothing}. Note that $|x_0-x_0^\e|\leq\sqrt{2}\e$ and hence $|z_\e|\leq \sigma+\sqrt{2}\e$. In particular, the choice of $\rho$ and $\sigma$ ensures that $B_{2\sigma}(x_0-z_\e)\subset B_\frac{\rho}{2}(x_0)$.  

To define the competitor  $u_{\e,\rho}$, we start by defining an auxiliary displacement $\tilde{u}_{\e,\rho}$ on  $\e\Z^2 \cap B_{\frac{\rho}{2}}(x_0)$ via  
	\begin{equation}\label{def:phi-e-rho}
	 \tilde{u}_{\e,\rho}(i)  \defas
	\begin{cases}
	 \frac{1}{2} \zeta_{x_0,\e}^\sigma(i+z_\e)  &\text{if}\ i\in\e\Z^2\cap B_\sigma(x_0-z_\e)\,,\\
	\tfrac{1}{4\pi}  \theta_\hor  (i-x_0+z_\e) &\text{if}\ i\in \e\Z^2 \cap ( B_{\frac{\rho}{2}}(x_0) \sm B_\sigma(x_0-z_\e) ) \,.
	\end{cases}
	\end{equation}
 For~\eqref{def:phi-e-rho} to be well-defined even when $x_0 \cdot e_2  \in \e \Z$, $\theta_\hor$ is extended on the discontinuity half-line $\Pi^+(x_0 - z_\e)$ by approximating from the bottom  $\theta_\hor(x) := \lim_{t \searrow 0} \theta_\hor(x - t e_2)$ for $x \in \Pi^+(x_0 - z_\e)$.

The displacement $\tilde{u}_{\e,\rho}$ requires a modification in order to be extended to the correct boundary conditions. The difficulty in the extension procedure lies in the mismatch between the discontinuity line $\Pi^+(x_0 - z_\e)$ of $\tfrac{1}{2 \pi}  \theta_\hor  (\cdot -x_0+z_\e)$ (value at $\de B_{\frac{\rho}{2}}(x_0)$) and the discontinuity line $\Pi^+(x_0)$ of $\tfrac{1}{2 \pi}  \theta_\hor  (\cdot -x_0)$ (desired value at $\de B_{\rho}(x_0)$). This mismatch does not allow for an immediate interpolation of the values via a cut-off function, but requires some care. Only after the aforementioned modification we will be in a position to define the competitor $u_{\e,\rho}$ in $B_{\frac{\rho}{2}}(x_0)$ and extend this to $B_{\rho}(x_0)$.  

To resolve the mismatch of the discontinuity lines, we construct a suitable representative $\xi_{\e,\rho}$ in the $\Z$-equivalence class of $2 \tilde{u}_{\e,\rho}$ which is basically discontinuous on $\Pi^+(x_0)$. Here and in Step~2, we provide all the details for this  procedure. 
 We consider the spin field $v_\e \in \SF_\e$ defined by 
 \begin{equation} \label{eq:veps}
  v_\e := \begin{cases}
    \exp( 2 \pi \iota 2 \tilde{u}_{\e,\rho})   &\text{if } i \in \e \Z^2 \cap B_{\frac{\rho}{2}}(x_0) \,,\\
    \exp(\iota \theta_\hor  (i-x_0+z_\e)) &\text{otherwise.}
  \end{cases}
\end{equation}
 We let $\overarc v_\e$ be its $\S^1$-interpolation as in Lemma~\ref{lem:S1-interpolation}. We start by lifting $\overarc{v}_\e$ in the set $ B_{\frac{\rho}{2}}(x_0)\sm \Pi^+(\overline{x}_\e -z_\e)$ by applying Remark~\ref{rem:lifting-S1-interpolation}. (In the next step we will show that $\Pi^+(\overline x_\e-z_\e)$ and $\Pi^+(x_0)$ have approximately the same height.) The set $B_{\frac{\rho}{2}}(x_0)\sm \Pi^+(\overline{x}_\e -z_\e)$ is open, bounded, and simply connected. Moreover, $\supp \mu_{2 \tilde{u}_{\e,\rho}}   \cap ( B_{\frac{\rho}{2}}(x_0)\sm \Pi^+(\overline{x}_\e -z_\e) ) =\emptyset$. Indeed, by  construction, we have   $\mu_{2 \tilde{u}_{\e,\rho}}(Q)=0$  for every $Q\in\Q_\e\big(B_{\frac{\rho}{2}}(x_0)\sm B_\sigma(x_0-z_\e)\big)$. Moreover, thanks to~\ref{bc-min-core} the same holds true for $Q\in\Q_\e\big(B_{\frac{\rho}{2}}(x_0)\big)$ with  $Q\cap \partial B_\sigma(x_0-z_\e)\neq\emptyset$.    Together with~\ref{vorticity-min-core} this implies that  $\mu_{2 \tilde{u}_{\e,\rho}}\mres B_{\frac{\rho}{2}}(x_0)=\delta_{\overline{x}_\e-z_\e}$, hence the desired condition on the support.  By Remark~\ref{rem:lifting-S1-interpolation}, there exists a lifting $2\pi \hat{\xi}_{\e,\rho} \in W^{1,\infty}_\loc(B_{\frac{\rho}{2}}(x_0)\sm \Pi^+(\overline{x}_\e -z_\e))$ satisfying 
\begin{equation} \label{eq:overarc-v hat-xi}
  \overarc{v}_\e(x) = \exp(2 \pi \iota \hat{\xi}_{\e,\rho}(x))  \text{ for  }x \in   B_{\frac{\rho}{2}}(x_0)\sm \Pi^+(\overline{x}_\e -z_\e) \, .
\end{equation}  

Finally, we set 
\[
\xi_{\e,\rho}(i) := \hat{\xi}_{\e,\rho}(i) \, , \quad u_{\e,\rho}(i) := \frac{1}{2} \xi_{\e,\rho}(i) \, ,  \quad \text{for } i \in \e \Z^2 \cap B_{\frac{\rho}{2}}(x_0) \, .
\]
Note that 
\begin{equation} \label{eq:2tildeu-2u-xi}
  2 \tilde{u}_{\e,\rho} \eqZ 2 u_{\e,\rho } = \xi_{\e,\rho} \, , \quad \text{on } \e \Z^2 \cap B_{\frac{\rho}{2}}(x_0) \, .
\end{equation}
Moreover, we recall that, by construction of $\overarc{v}_\e$, the function  $\hat{\xi}_{\e,\rho}$ is  affine on lattice edges $[i,j] \subset B_{\frac{\rho}{2}}(x_0)$ with $i,j \in \e \Z^2$ satisfying $|i-j|=\e$ and $[i,j] \cap \Pi^+(\overline{x}_\e - z_\e)=\emptyset$. Moreover, the slope on $[i,j]$ is $\frac{1}{\e}$-proportional to $\d^e (2 \tilde{u}_{\e,\rho})(i,j) = \d^e \xi_{\e,\rho}(i,j)$.

% Eventually, we define $u_{\e,\rho}$ on $B_{\frac{\rho}{2}}(x_0)$ by setting
% 	\begin{equation*}
% 	u_{\e,\rho}(i)\defas\frac{1}{2}\RRR \hat{\xi}_{\e,\rho}(i) \EEE \;\text{ for every}\ i\in\e\Z^2\cap B_{\frac{\rho}{2}}(x_0)\,.
% 	\end{equation*}
  % \RRR 
  % Note that $\e \Z^2 \cap B_{\frac{\rho}{2}}(x_0) \subset B_{\frac{\rho}{2}}(x_0)\sm \Pi^+(\overline{x}_\e -z_\e)$, since $\overline{x}_\e - z_\e \notin \e \Z^2$, hence the previous definition is well-posed.
  % \EEE

\end{step}

\begin{step}{2}{Comparison between the half-lines $\Pi^+(\overline{x}_\e - z_\e)$ and $\Pi^+(x_0)$} 
  We note that the height of the discontinuity half-line $\Pi^+(\overline{x}_\e -z_\e)$ and the height of the half-line $\Pi^+(x_0)$ are so close that they are indistinguishable by the lattice $\e \Z^2$, see also Figure~\ref{fig:shifting-smoothing}. Indeed, by~\eqref{def:integer-part} we have that  
  \[
    (\overline{x}_\e - z_\e) \cdot e_2 = \Big(\overline{\imath}_\e + \frac{\e}{2}(e_1+e_2) - \big( (\overline{\imath}_\e-x_0^\e)\cdot e_2 \big) e_2  \Big) \cdot e_2 = \Big( x_0^\e + \frac{\e}{2} e_2 \Big) \cdot e_2 \, .
  \]
  If $x_0 \cdot e_2 \in \e \Z$, then $x_0 \cdot e_2 = x_0^\e \cdot e_2$ and  $|(\overline{x}_\e - z_\e) \cdot e_2 - x_0 \cdot e_2| = \frac{\e}{2}$. Otherwise, if $x_0 \cdot e_2 \notin \e \Z$, then $|(\overline{x}_\e - z_\e) \cdot e_2 - x_0 \cdot e_2| < \frac{\e}{2}$.
  
\end{step}

\begin{figure}[ht]
  \begin{center}
  \includegraphics{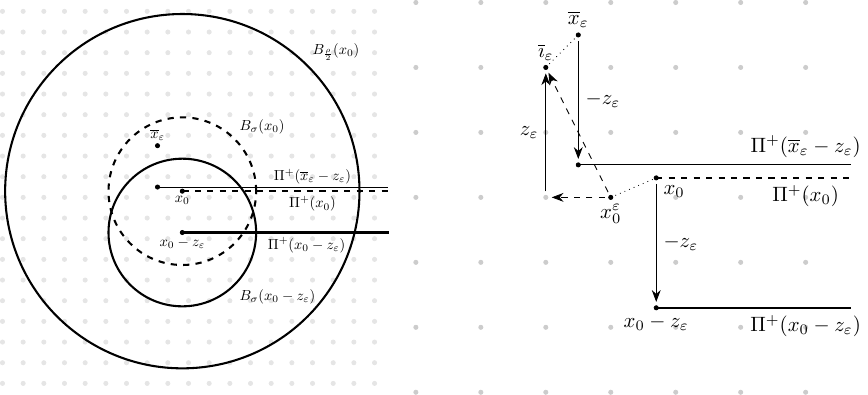}
  \caption{Sets used in the shifting and smoothing procedure in the proof of Proposition~\ref{prop:core-partials}. For the sake of clarity, the picture does not respect the sizes imposed in the proof on the radii  $\sigma > 8\e$  and $\rho > 6\sigma$. On the left: balls used in the construction and half-lines relevant in the shifting and smoothing procedure. On the right: zoomed-in picture of the points involved in the shifting procedure. The picture further shows that the vertical distance between $\Pi^+(x_0)$ and $\Pi^+(\overline{x}_\e - z_\e)$ is less than $\frac{\e}{2}$.}
  \label{fig:shifting-smoothing}
  \end{center}
\end{figure}

\begin{step}{3}{Value at $\de B_{\frac{\rho}{2}}(x_0)$}
  The observation in Step~2 allows us to describe $\xi_{\e,\rho}$ close to $\de B_{\frac{\rho}{2}}(x_0)$ in a more convenient way useful for the extension of $u_{\e,\rho}$ outside $B_{\frac{\rho}{2}}(x_0)$. 
  
  The set  $( \R^2 \sm B_\sigma(x_0-z_\e) ) \sm\Pi^+(x_0)$ is simply connected. Hence, there exists a lifting of $\frac{x-x_0+z_\e}{|x-x_0+z_\e|}$ with a shifted discontinuity, \ie $\thetas(\cdot\,-x_0+z_\e)\in C^\infty\big( ( \R^2 \sm B_\sigma(x_0-z_\e) ) \sm\Pi^+(x_0)\big)$ satisfying  
  \begin{equation} \label{eq:tilde-theta lifting}
    \exp( \iota \thetas(x-x_0+z_\e)) = \frac{x-x_0+z_\e}{|x-x_0+z_\e|} \quad  \text{for } x \in (  \R^2 \sm B_\sigma(x_0-z_\e) ) \sm\Pi^+(x_0) \, .
  \end{equation}
    We extend  $ \thetas(\cdot - x_0 + z_\e)$ to the discontinuity half-line $\Pi^+(x_0)$ by approximating from the bottom $\thetas(x - x_0 + z_\e) := \lim_{t \searrow 0} \thetas(x - x_0 + z_\e - t e_2)$ for $x \in \Pi^+(x_0)$. 
    
    \smallskip
    
    In this step we show that 
  \begin{equation} \label{eq:hatxi-equals-tildetheta}
    2 \pi \xi_{\e,\rho}(i) = \thetas(i-x_0+z_\e) + 2 \pi k \quad  \text{ for } i \in \e \Z^2 \cap B_{\frac{\rho}{2}}(x_0)\sm B_{\sigma}(x_0-z_\e) \, ,
  \end{equation}
  for some $k \in \Z$.  To prove~\eqref{eq:hatxi-equals-tildetheta} we start by observing that~\eqref{def:phi-e-rho}--\eqref{eq:tilde-theta lifting} and the identity $\overarc{v}_{\e | \e \Z^2} = v_\e$ imply that 
  \[
    \begin{split}
     \exp\big(2 \pi \iota \xi_{\e,\rho}(i)\big)  & = \overarc{v}_\e(i) = v_\e(i)  = \exp\big(2\pi \iota 2 u_{\e,\rho}(i) \big) = \exp\big(\iota \theta_\hor(i-x_0+z_\e)\big) \\
      & = \frac{i - x_0 + z_\e}{|i - x_0 + z_\e|}  = \exp\big(\iota \thetas(i - x_0 + z_\e )\big) \quad \text{ for } i \in \e \Z^2 \cap (B_{\frac{\rho}{2}}(x_0)\sm B_\sigma(x_0-z_\e)) \, .
    \end{split}
   \]
  This implies that for every $i \in \e \Z^2 \cap (B_{\frac{\rho}{2}}(x_0)\sm B_\sigma(x_0-z_\e))$ there exists $k(i) \in \Z$ such that 
  \[
    2 \pi  \xi_{\e,\rho}(i) = \thetas(i - x_0 + z_\e ) + 2 \pi k(i)  \, .
  \]
  To prove~\eqref{eq:hatxi-equals-tildetheta} it is enough to show that 
  \begin{equation} \label{eq:k is constant}
    k(i) = k(j) \text{ for every } i,j \in \e \Z^2 \cap (B_{\frac{\rho}{2}}(x_0)\sm B_{\sigma}(x_0-z_\e))\, .
  \end{equation}

  Let us start by proving~\eqref{eq:k is constant} for $i, j \in \e \Z^2 \cap ( B_{\frac{\rho}{2}}(x_0)\sm B_{\sigma}(x_0-z_\e))$ satisfying $|j - i| = \e$ (\ie they are nearest neighbors) and  $[i,j] \cap \Pi^+(\overline{x}_\e -z_\e) = \emptyset$. We write 
  \begin{equation} \label{eq:decomposing k}
    2 \pi (k(j) - k(i)) = 2 \pi (\xi_{\e,\rho}(j) - \xi_{\e,\rho}(i)) + ( \thetas(i - x_0 + z_\e) - \thetas(j - x_0 + z_\e) )\, .
  \end{equation}
  Let us study the term $\thetas(i - x_0 + z_\e) - \thetas(j - x_0 + z_\e)$ in~\eqref{eq:decomposing k}.  
  
  {\itshape Case 1:} If $x_0 \cdot e_2 \notin \Z$, then Step 2 gives that $[i,j] \cap \Pi^+(x_0) = \emptyset$. By~\eqref{eq:tilde-theta lifting} and since $[i,j] \subset \R^2 \sm B_{\frac{\sigma}{2}}(x_0-z_\e)$, we infer that 
  \begin{equation} \label{eq:increment tilde theta}
    \begin{split}
      &  | \thetas(j - x_0 + z_\e) -   \thetas(i - x_0 + z_\e)|   \leq \e \sup_{x\in [i,j]}|\nabla  \thetas(x - x_0 + z_\e)| = \e \sup_{x\in[i,j]} \Big|\nabla \frac{x - x_0 + z_\e}{|x - x_0 + z_\e|}\Big| \\
      & \quad \leq \e \sup_{x\in[i,j]}  \frac{1}{|x - x_0 + z_\e|}  \leq \frac{2}{\sigma} \e < \frac{\pi}{2} \, ,
    \end{split}
  \end{equation}
  where the last inequality is true for $\e$ small enough.

  {\itshape Case 2:} If $x_0 \cdot e_2 \in \Z$ and $[i,j] \cap \Pi^+(x_0) = \emptyset$, we argue as in {\itshape Case 1}. If $x_0 \cdot e_2 \in \Z$ and $[i,j] \cap \Pi^+(x_0) \neq \emptyset$, then Step~2 and the fact that $[i,j] \cap \Pi^+(\overline{x}_\e -z_\e) = \emptyset$ imply that either $j \in \Pi^+(x_0)$ or both $i,j \in \Pi^+(x_0)$. In either case, we obtain~\eqref{eq:increment tilde theta}  by approximating from the bottom $\thetas(x - x_0 + z_\e) := \lim_{t \searrow 0} \thetas(x - x_0 + z_\e - t e_2)$ for $x \in \Pi^+(x_0)$ and arguing as in {\itshape Case 1}.

  \smallskip

  Let us study the term $2 \pi (\xi_{\e,\rho}(j) - \xi_{\e,\rho}(i))$ in~\eqref{eq:decomposing k}. Recall that $[i,j] \cap \Pi^+(\overline{x}_\e -z_\e) = \emptyset$ and that on $[i,j]$ the lifting $\hat{\xi}_{\e,\rho}$ is  affine with slope $\frac{1}{\e}$-proportional to $\d^e (2 \tilde{u}_{\e,\rho})(i,j) = \d^e \theta_\hor(\cdot -x_0+z_\e)(i,j)$. Since $\theta_\hor(\cdot - x_0 + z_\e) \eqZ \thetas(\cdot - x_0 + z_\e)$, by~\eqref{eq:increment tilde theta} we have that 
  \[
    |\d^e \theta_\hor(\cdot - x_0 + z_\e)(i,j)| = |\d^e \thetas(\cdot - x_0 + z_\e)(i,j)| =  | \thetas(j - x_0 + z_\e) -   \thetas(i - x_0 + z_\e)| < \frac{\pi}{2} \, ,
  \]
  for $\e$ small enough.  We deduce that 
  \begin{equation} \label{eq:hatxi small}
    2\pi | \xi_{\e,\rho}(j) - \xi_{\e,\rho}(i) | \leq  \e 2\pi \sup_{[i,j]}| \nabla  \hat{\xi}_{\e,\rho } | < \frac{\pi}{2}\, .
  \end{equation}
  for $\e$ small enough. Putting together~\eqref{eq:decomposing k}, \eqref{eq:increment tilde theta}, and~\eqref{eq:hatxi small} we obtain that $|k(j) - k(i)| < \frac{1}{2}$. Hence $k(j) = k(i)$  and~\eqref{eq:k is constant} is proven for nearest neighbors.

  \smallskip

  We obtain~\eqref{eq:k is constant} by connecting every pair $i,j \in \e \Z^2 \cap (B_{\frac{\rho}{2}}(x_0)\sm B_{\sigma}(x_0-z_\e))$ with a chain of points $i = i_0, i_1, \dots, i_{N-1}, i_N = j \in B_{\frac{\rho}{2}}(x_0)\sm B_{\sigma}(x_0-z_\e)$ satisfying $|i_h - i_{h-1}| = \e$ and $[i_{h-1}, i_h] \cap \Pi^+(\overline{x}_\e -z_\e) = \emptyset$. By the previous argument, we obtain the chain of equalities 
  \[
  k(i) = k(i_0) = k(i_1) = \dots = k(i_{N-1}) = k(i_{N}) = k(j) \, .  
  \]

  In conclusion, we can redefine $\thetas$ by subtracting $2\pi k(i)$ to obtain that 
  \[
  2 u_{\e,\rho} = \xi_{\e,\rho} = \frac{1}{2\pi}\thetas(\cdot - x_0 + z_\e) \quad \text{on } \e \Z^2 \cap B_{\frac{\rho}{2}}(x_0) \sm B_{2 \sigma}(x_0-z_\e) \, ,
  \]
  where $\thetas$ is discontinuous on $\Pi^+(x_0)$. 
  
  Note that for every $i,j\in\e\Z^2\cap B_{\rho}(x_0)$ with $|i-j|=\e$ 
	\begin{equation}\label{eq:du-dxi}
	\begin{split}
	|\d ( 2 u_{\e,\rho})(i,j)| &=\dist\big(\d (2 u_{\e,\rho})(i,j);\Z\big) \quad \text{if }  j = i + \e e_1 \,, \\
  |\d ( 2 u_{\e,\rho})(i,j)| &=\dist\big(\d (2 u_{\e,\rho})(i,j);\Z\big) \quad \text{if } j = i + \e e_2 \text{ and } [i,j]\cap\Pi(x_0)=\emptyset\,,
	\end{split}
	\end{equation}
  where $\Pi(x_0)$ is the full horizontal line.\footnote{The condition $[i,j]\cap\Pi(x_0) = \emptyset$ is a strong condition that guarantees that the points are far from the discontinuity lines $\Pi^+(x_0)$ and $\Pi^+(\overline{x}_\e - z_\e)$. Formula~\eqref{eq:du-dxi} could be improved to a more precise statement, but it is not needed for the purposes of this proof.}

\end{step}

\begin{step}{4}{Extension of $u_{\e,\rho}$ outside $B_{\frac{\rho}{2}}(x_0)$}
We are now in a position to interpolate $2 u_{\e,\rho} = \xi_{\e,\rho}$ to $\tfrac{1}{2 \pi}  \theta_\hor  (\cdot -x_0)$ (plus a constant) in the cut annulus $A_{\frac{\rho}{2},\rho}(x_0) \sm \Pi^+(x_0)= ( B_{\rho}(x_0) \sm B_{\frac{\rho}{2}}(x_0) ) \sm \Pi^+(x_0)$. Specifically, we will interpolate to $\tfrac{1}{2 \pi}  \theta_\hor  (\cdot -x_0) + \frac{1}{2\pi}a_{\e,\rho}$ where the constant $a_{\e,\rho}$ is defined by 
\begin{equation}\label{def:alpha}
a_{\e,\rho} := \fint_{A_{\frac{\rho}{2},\rho}(x_0)} \thetas (x-x_0+z_\e)\dx - \fint_{A_{\frac{\rho}{2},\rho}(x_0)}\theta_\hor(x-x_0)\dx\,.
	\end{equation}
This constant is needed later in order to apply a Poincar\'e inequality.

We introduce a radial cut-off function $\eta\in C^\infty\big([0,+\infty);[0,1])$ be such that $\eta\equiv 0$ on $[0, \frac{5}{8}]$ and $\eta\equiv 1$ on $[\frac{7}{8},+\infty)$. We then define $\vartheta_{\e,\rho}\in C^\infty\big(\R^2 \sm\Pi^+(x_0)\big)$ by setting
	\begin{equation} \label{eq:cutoff interpolation}
	 \vartheta_{\e,\rho}(x)\defas \thetas(x-x_0+z_\e)+\eta\Big(\frac{|x-x_0|}{\rho}\Big)\big(\theta_\hor(x-x_0) + a_{\e,\rho} - \thetas(x-x_0+z_\e)\big)\,,
	\end{equation}
so that $\vartheta_{\e,\rho}(x)= \thetas (x-x_0+z_\e)$ on $B_{\frac{5}{8}\rho}(x_0)$ and $\vartheta_{\e,\rho}(x)=\theta_\hor(x-x_0) + a_{\e,\rho}$ on $B_{\rho}(x_0) \sm B_{\frac{7}{8}\rho}(x_0)$. Eventually, we set
	\begin{equation*}
    \xi_{\e,\rho}(i) \defas \frac{1}{2\pi}\vartheta_{\e,\rho}(i) \, , \quad  u_{\e,\rho}(i)\defas\frac{1}{2}\xi_{\e,\rho}(i) \quad \text{for } i\in\e\Z^2\cap A_{\frac{\rho}{2},\rho}(x_0)\,. % \;\text{ and }\; \psi_{\e,\rho}(i)\defas 2u_{\e,\rho}(i)\;
	\end{equation*}
We have that for every $i,j\in\e\Z^2\cap A_{\frac{\rho}{2},\rho}(x_0)$ with $|i-j|=\e$ 
	\begin{equation}\label{eq:d-u-e-r}
    \begin{split}
      |\d ( 2 u_{\e,\rho})(i,j)| &=\dist\big(\d (2 u_{\e,\rho})(i,j);\Z\big) \quad \text{if }  j = i + \e e_1 \,, \\
      |\d ( 2 u_{\e,\rho})(i,j)| &=\dist\big(\d (2 u_{\e,\rho})(i,j);\Z\big) \quad \text{if } j = i + \e e_2 \text{ and } [i,j]\cap\Pi^+(x_0)=\emptyset\,,
      \end{split}
	\end{equation}
  The computations for~\eqref{eq:d-u-e-r} analogous to those in Step~3.
  % (To be precise, recalling that on the discontinuity line $\Pi^+(x_0)$ both  $\theta_\hor(x - x_0)$ and $\thetas(\cdot - x_0 + z_\e)$ are approximated with points approaching from the bottom, formula~\eqref{eq:d-u-e-r} holds true also when $x_0 \cdot e_2 \in \Z$ and $[i,j) \cap\Pi^+(x_0)=\emptyset$.)
  
  \smallskip
   
  We conclude this step by observing that $u_{\e,\rho}$ is a competitor for $\gamma_\e^\p\big(B_\rho(x_0)\big)$. It remains to establish the energy estimate~\eqref{est:energy-core-final}.
\end{step}

\begin{step}{5}{Energy estimate I: From $F_\e^\p$ to $F_\e^\screw$}
In this step we show that
	\begin{equation}\label{est:energy-core-part-screw}
	4 F_\e^\p\big(u_{\e,\rho},B_\rho(x_0)\big)\leq(1+C\e)F_\e^\screw\big(2 u_{\e,\rho},B_\rho(x_0)\big)+C\rho
	\end{equation}
for some constant $C>0$. In order to simplify notation we will establish~\eqref{est:energy-core-part-screw} in the case $x_0=0$, since the estimates will not rely on the fact that $0$ is a lattice point.

\smallskip
Using the expression of $F_\e^\p$ in~\eqref{eq:F-edge-F-part} and using~\eqref{eq:du-dxi} and~\eqref{eq:d-u-e-r} we write
	\begin{equation}\label{eq:energy-core-part-edge}
	\begin{split}
	F_\e^\p\big(u_{\e,\rho},B_\rho\big) &=\frac{1}{4}F_\e^\edge\big(2 u_{\e,\rho},B_\rho\big)+ \frac{\alpha}{\pi^2}\e \sum_{i\in\Z_\e^{2e_2}(B_{\rho})}\hspace*{-1em}f_1\big(\d u_{\e,\rho}(i,i+2\e e_2)\big)\\
	&=\frac{1}{4}F_\e^\screw\big(2 u_{\e,\rho} ,B_\rho\big)+ \frac{\alpha}{\pi^2} \e \sum_{i\in\Z_\e^{2e_2}(B_{\rho})}\hspace*{-1em}f_1\big(\d u_{\e,\rho}(i,i+2\e e_2)\big)\,.
	\end{split}
	\end{equation}
It thus remains to estimate the last term on the right-hand side of~\eqref{eq:energy-core-part-edge}.
Suppose first that $i\in\Z_\e^{2e_2}(B_\rho)$ is such that $[i,i+2\e e_2]\cap\Pi\neq\emptyset$, where $\Pi = \Pi(0)$.
In this case we use the trivial estimate $f_1(t) \leq\frac{\pi^2}{2}$ to infer that
	\begin{equation}\label{est:energy-core-jump}
	\begin{split}
	\dsum{i\in\Z_\e^{2 e_2}(B_{\rho})}{ [i,i+2\e e_2] \cap\Pi  \neq\emptyset}\hspace*{-2em}f_1\big(\d u_{\e,\rho}(i,i+2\e e_2)\big)&\leq \frac{\alpha}{2}\#\{i\in\e\Z^2\cap B_{\rho}\colon [i,i+2\e e_2] \cap\Pi \neq\emptyset\}\\
	&\leq\frac{C}{\e}\H^1\big(B_{\rho}\cap \Pi \big)\leq\frac{C\rho}{\e}\,.
	\end{split}
	\end{equation}
If instead $[i,i+2\e e_2]  \cap\Pi=\emptyset$, we use convexity to deduce that
	\begin{equation}\label{est:nnn-convexity}
    \begin{split}
      f_1\big(\d u_{\e,\rho}(i,i+2\e e_2)\big) & \leq 2\pi^2\big|\d u_{\e,\rho}(i,i+2\e e_2)\big|^2 \\
      & \leq 4\pi^2\Big(\big|\d u_{\e,\rho}(i,i+\e e_2)\big|^2 +\big|\d u_{\e,\rho}(i+\e e_2,i+2\e e_2)\big|^2\Big)\,.
    \end{split}
	%\leq 2 \Big(f_1\big(\d u_{\e,\rho}(i,i+\e e_2)\big)+f_1\big(\d u_{\e,\rho}(i+\e e_2,i+2\e e_2)\big)\Big)\,.
	\end{equation}
Moreover, since $[i,i+2\e e_2] \cap\Pi=\emptyset  $, we deduce from~\eqref{eq:du-dxi} and~\eqref{eq:d-u-e-r} that
	\begin{equation*}
	4\pi^2\big(\d u_{\e,\rho}(i,i+\e e_2)\big)^2=\pi^2\dist^2\big(\d (2 u_{\e,\rho})(i,i+\e e_2);\Z\big)=\frac{1}{2}f_1\big(\d (2 u_{\e,\rho})(i,i+\e e_2)\big)
	\end{equation*}
and the same equality holds true for $\d u_{\e,\rho}(i+\e e_2,i+2\e e_2)$. Together with~\eqref{est:nnn-convexity} this gives
	\begin{equation}\label{est:energy-core-no-jump}
	\dsum{i\in\Z_\e^{2 e_2}(B_{\rho})}{ [i,i+2\e e_2)\cap\Pi=\emptyset}\hspace*{-2em}f_1\big(\d u_{\e,\rho}(i,i+2\e e_2)\big)\leq F_\e^\screw\big(2 u_{\e,\rho},B_\rho\big)\,.
	\end{equation}
Now~\eqref{est:energy-core-part-screw} follows from~\eqref{eq:energy-core-part-edge}, \eqref{est:energy-core-jump}, and~\eqref{est:energy-core-no-jump}.
\end{step}

\begin{step}{6}{Energy estimate II: Bound on $F_\e^\screw$}
In this step we show that
	\begin{equation}\label{est:energy-core-screw-final}
	F_\e^\screw\big(2 u_{\e,\rho},B_\rho(x_0)\big)\leq \gamma_\e^\screw\big(B_\sigma(x_0)\big)+\pi\log\frac{\rho}{\sigma}+r(\e,\sigma,\rho)
	\end{equation}
with $\lim_{\rho\to 0}\lim_{\sigma\to 0}\lim_{\e\to 0}r(\e,\sigma,\rho)=0$. Again we assume without loss of generality that $x_0=0$.

\smallskip
Paying attention to boundary interactions, we split the screw-dislocation energy as follows:
	\begin{equation}\label{est:core-screw-splitting}
	F_\e^\screw\big(2 u_{\e,\rho},B_\rho\big)\leq F_\e^\screw\big(2 u_{\e,\rho}, B_\sigma(-z_\e)\big)+F_\e^\screw\big(\vartheta_{\e,\rho}, B_\rho\sm B_{\sigma-2\e}(-z_\e)\big)\,.
	%+\sum_{k=1}^2\sum_{i\in Z_{\e,\rho,\sigma}^k}f_1(\d\psi_{\e,\rho}(i,i+\e e_k))\,,
	\end{equation}
%where we have set $Z_{\e,\rho,\sigma}^k\defas\{i\in\e\Z^2\cap B_\rho\colon i\pm\e e_k\in B_\rho\sm B_{\sigma}(-z_\e)\}$
Recalling~\eqref{eq:2tildeu-2u-xi} and the definition of $u_{\e,\rho} $ in~\eqref{def:phi-e-rho}, by~\eqref{eq:energy-core-small}  we infer that
	\begin{equation}\label{eq:energy-core-interior}	
	F_\e^\screw\big(2 u_{\e,\rho}, B_\sigma(-z_\e)\big) = F_\e^\screw\big( 2 \tilde{u}_{\e,\rho}  , B_\sigma(-z_\e)\big) = F_\e^\screw\big( \zeta^\sigma_{x_0,\e}  , B_\sigma(-z_\e)\big) =\gamma_\e^\screw(B_\sigma)\,.
	\end{equation}
It remains to evaluate the energy in $A_\e^{\sigma,\rho}\defas B_\rho\sm B_{\sigma-2\e}(-z_\e)$. Using standard interpolation estimates we deduce that  
	\begin{equation}\label{est:energy-core-screw-interpolation}
	\begin{split} 
    F_\e^\screw\big(\vartheta_{\e,\rho}, A_\e^{\sigma,\rho}\big)   \leq \frac{1}{2} \int_{A_\e^{\sigma,\rho}}|\nabla \vartheta_{\e,\rho} (x)|^2\dx + C \e \big\| | \nabla   \vartheta_{\e,\rho} | \cdot | \nabla^2  \vartheta_{\e,\rho} | \big\|_{L^\infty(A_\e^{\sigma,\rho})}  |A_\e^{\sigma,\rho}| \, .
	\end{split}
	\end{equation}
 By construction, see~\eqref{eq:cutoff interpolation}, we have $\nabla \vartheta_{\e,\rho} (x)=\nabla  \thetas  (x+z_\e)$ on $B_{\frac{\rho}{2}}\sm B_{\sigma-2\e}(-z_\e)$, while on $A_{\frac{\rho}{2},\rho}$ we write
	\begin{equation}\label{eq:grad-phi}
	\begin{split}
    \nabla \vartheta_{\e,\rho} (x) &  =\nabla \thetas (x+z_\e) +\eta\bigg(\frac{|x|}{\rho}\bigg)\big(\nabla\theta_\hor(x)-\nabla \thetas (x+z_\e)\big) \\
  & \quad +\frac{1}{\rho}\eta'\bigg(\frac{|x|}{\rho}\bigg)\frac{x}{|x|}\big(\theta_{\hor}(x) + a_{\e,\rho} - \thetas  (x+z_\e)\big)\,.
	\end{split}
	\end{equation}
This implies, in particular, that $|\nabla\vartheta_{\e,\rho} |\leq\frac{C}{\rho-\sigma+2\e}$ on $A_\e^{\sigma,\rho}$. In a similar way we obtain $|\nabla^2\vartheta_{\e,\rho}|\leq\frac{C}{(\rho-\sigma+2\e)^2}$ on $A_\e^{\sigma,\rho}$.
Thus, a combination of~\eqref{est:core-screw-splitting}, \eqref{eq:energy-core-interior}, and~\eqref{est:energy-core-screw-interpolation} yields
	\begin{equation}\label{est:energy-core-screw}
	F_\e^\screw\big(2 u_{\e,\rho},B_\rho\big)\leq\gamma_\e^\screw\big(B_\sigma\big)+ \frac{1}{2} \int_{A_\e^{\sigma,\rho}}|\nabla \vartheta_{\e,\rho}(x)|^2\dx+\frac{C\e}{(\rho-\sigma+2\e)^2}\,.
	\end{equation} 
We are then left to estimate
	\begin{equation}\label{eq:splitting-grad-phi}
	\begin{split}
	\int_{A_\e^{\sigma,\rho}}|\nabla \vartheta_{\e,\rho}(x)|^2\dx
	&=\int_{A_\e^{\sigma,\rho}}|\nabla\thetas (x+z_\e)|^2\dx
	+\int_{A_{\frac{\rho}{2},\rho}}|\nabla \vartheta_{\e,\rho} (x)-\nabla \thetas (x+z_\e)|^2\dx\\
	&\hspace*{1em}+2\int_{A_{\frac{\rho}{2},\rho}}\nabla \thetas  (x+z_\e)\cdot\big(\nabla \vartheta_{\e,\rho} (x)-\nabla \thetas (x+z_\e)\big)\dx\,.
	\end{split}
	\end{equation}
Recalling that $B_\rho\subset B_{\rho+\sigma+\sqrt{2}\e}(-z_\e)$ and using a change of variables $y=x+z_\e$ we obtain 
	\begin{equation}\label{est:core-log-contribution}
	\begin{split}
	\int_{A_\e^{\sigma,\rho}}|\nabla \thetas  (x+z_\e)|^2\dx 
	& \leq\int_{B_{\rho+\sigma+\sqrt{2}\e}\sm B_{\sigma-2\e}} |\nabla \thetas (y)|^2\dy=2\pi\log\frac{\rho+\sigma+\sqrt{2}\e}{\sigma-2\e}\\
  & =2\pi\log\frac{\rho}{\sigma}+r_1(\e,\sigma,\rho)
	\end{split}
	\end{equation}
with $r_1(\e,\sigma,\rho)=2\pi\log\frac{(\rho+\sigma+\sqrt{2}\e)\sigma}{(\sigma-2\e)\rho}\to 0$ as $\e\to 0$, $\sigma\to 0$, and $\rho\to 0$ subsequently.

\smallskip
It remains to show that the last terms on the right-hand side of~\eqref{eq:splitting-grad-phi} vanish. Recalling~\eqref{eq:grad-phi} and using Young's inequality we get
	\begin{equation*}
	|\nabla \vartheta_{\e,\rho}(x)-\nabla \thetas (x+z_\e)|^2 \leq 2|\nabla \thetas (x+z_\e)-\nabla\theta_\hor(x)|^2+\frac{2\|\eta'\|_{L^\infty}^2}{\rho^2}|\thetas (x+z_\e)-\theta_\hor(x) -  a_{\e,\rho}|^2\,.
	\end{equation*}
Note that $\thetas  (\cdot\,+z_\e)-\theta_{\hor}  -  a_{\e,\rho}\in H^1(A_{\frac{\rho}{2},\rho})$. In view of~\eqref{def:alpha}, an application of Poincar\'e inequality yields
%Moreover, the continuity of the shift-operator in Sobolev spaces (see \eg \cite[Proposition 9.3]{BrezisBook}) ensures that
	\begin{equation}\label{est:core-poincare}
	\begin{split}
	\frac{1}{\rho^2}\int_{A_{\frac{\rho}{2},\rho}}|\thetas (x+z_\e)-\theta_\hor(x) -  a_{\e,\rho}|^2\dx\leq C\int_{A_{\frac{\rho}{2},\rho}}|\nabla \thetas (x+z_\e)-\nabla\theta_\hor(x) |^2\dx\,.
	\end{split}
	\end{equation}
We show now how to bound the right-hand side of~\eqref{est:core-poincare} using the continuity of the shift-operator in Sobolev spaces (see \eg \cite[Proposition 9.3]{BrezisBook}). 
We start by observing that $\nabla \thetas =\nabla\theta_\hor$ a.e.\ in $A_{\frac{\rho}{2},\rho}(z_\e)$, since both $\thetas$ and $\theta_\hor$ are liftings of $x/|x|$. Moreover, $\nabla\theta_\hor\in H^1(A_{\frac{\rho}{2},\rho}\sm\Pi^+)$ and setting
	\begin{equation*}
	E_{\e}^{\sigma,\rho}\defas A_{\frac{\rho}{2}+\sigma+\sqrt{2}\e,\rho-\sigma-\sqrt{2}\e}\cap\big\{\dist(x,\Pi^+)>\sigma+\sqrt{2}\e\big\}\wcont A_{\frac{\rho}{2},\rho}\sm\Pi^+
	\end{equation*}
we have $|z_\e|<\dist(E_\e^{\sigma,\rho},\partial(A_{\frac{\rho}{2},\rho}\sm\Pi^+))$.
The continuity of the shift operator then implies that
	\begin{equation}\label{est:core-shift-operator}
	\begin{split}
	\int_{E_\e^{\sigma,\rho}}|\nabla \thetas (x+z_\e)-\nabla\theta_\hor(x)|^2\dx &\leq (\sigma+2\sqrt{\e})^2\int_{A_{\frac{\rho}{2},\rho}\sm\Pi^+}|\nabla^2\theta_\hor(x)|^2\dx\leq C\frac{\sigma^2}{\rho^2}\,.
	\end{split}
	\end{equation}
Since moreover $|A_{\frac{\rho}{2},\rho}\sm E_\e^{\sigma,\rho}|\leq C\sigma\rho$ and
	\begin{align*}
	|\nabla\theta_\hor(x)|^2 \leq\frac{C}{\rho^2}\,,\quad|\nabla \thetas (x+z_\e)|^2\leq\frac{C}{(\rho-\sigma-\sqrt{2}\e)^2}\leq\frac{C}{\rho^2}\;\text{ for}\ x\in A_{\frac{\rho}{2},\rho}\,,
	\end{align*}
we finally deduce from~\eqref{est:core-shift-operator} that
	\begin{equation}\label{est:core-shift-final}
	\int_{A_{\frac{\rho}{2},\rho}}|\nabla\thetas(x+z_\e)-\nabla\theta_\hor(x)|^2\dx\leq C\frac{\sigma^2+\sigma\rho}{\rho^2}\,.
	\end{equation}
Eventually, an application of H\"older's inequality together with~\eqref{est:core-log-contribution} and~\eqref{est:core-shift-final} yields
	\begin{equation}\label{est:core-hoelder}
	\begin{split}
	& \int_{A_{\frac{\rho}{2},\rho}}\nabla \thetas (x+z_\e)\cdot\big(   \nabla \vartheta_{\e,\rho}(x)-\nabla\thetas (x+z_\e)\big)\dx \\
	&\quad \leq \|\nabla \thetas(\cdot\, +z_\e)\|_{L^2(A_{\frac{\rho}{2},\rho})}\|\nabla \vartheta_{\e,\rho}  -\nabla \thetas (\cdot\,+z_\e)\|_{L^2(A_{\frac{\rho}{2},\rho})}\\
	&\quad \leq C\Big(\log\frac{\rho}{\sigma}+r_2(\e,\sigma,\rho)\Big)^\frac{1}{2}\Big(\frac{\sigma^2+\sigma \rho}{\rho^2}\Big)^\frac{1}{2}\,.
	\end{split}
	\end{equation}
Since the right-hand side of~\eqref{est:core-hoelder} vanishes as $\e\to 0$ and $\sigma\to 0$, we obtain~\eqref{est:energy-core-screw-final} by combining~\eqref{est:energy-core-screw} with~\eqref{eq:splitting-grad-phi}--\eqref{est:core-poincare} and~\eqref{est:core-shift-final}--\eqref{est:core-hoelder}.
\end{step}

\begin{step}{7}{Conclusion}
Gathering~\eqref{est:energy-core-part-screw} and~\eqref{est:energy-core-screw-final} we deduce that the competitor $u_{\e,\rho}$ satisfies
	\begin{equation*}
	4F_\e^\p\big(u_{\e,\rho},B_\rho(x_0)\big)\leq (1+\e)\Big(\gamma_\e^\screw\big(B_\sigma(x_0)\big)+\pi\log\frac{\rho}{\sigma}+r(\e,\sigma,\rho)\Big)\,.
	\end{equation*}
In view of~\eqref{def:gamma} it is not restrictive to assume that $\e$ was chosen sufficiently small so that $\gamma_\e^\screw(B_\sigma(x_0))\leq \gamma+\pi\log\frac{\sigma}{\e}+1$. Thus, the above estimate yields~\eqref{est:energy-core-final} upon replacing $r(\e,\sigma,\rho)$ by $(1+\e)r(\e,\sigma,\rho)+\e(\gamma + \pi\log\frac{\rho}{\e}+1 )+C\rho$. As indicated above this gives~\eqref{est:core-ub}, which together with~\eqref{est:core-lb} finally yields~\eqref{eq:gamma-part}.
\end{step}
\end{proof}
\begin{remark}\label{rem:core-partials}
In the proof of Proposition~\ref{prop:core-partials} we have established~\eqref{est:core-ub} under the additional assumption that $\theta_\hor\in C^\infty(\R^2\sm\Pi^+)$. Suppose now that $\theta_{\hor}\in C^\infty(\R^2\sm\Pi)$; then there exist $\tilde{\theta}_\hor\in C^\infty(\R^2\sm\Pi^+)$ and $z\in\Z$ such that $\tilde{\theta}_\hor=\theta_\hor$ on $\R^2\cap\{x_2\geq 0\}$ and $\tilde{\theta}_\hor=\theta_\hor+2\pi z$ on $\R^2\cap\{x_2 <0\}$. Let now $x_0\in\R^2$ and for $\e>0$, and $\rho>8\e$; for any competitor $u_{\e,\rho}\in\AD_\e$ satisfying $u_{\e,\rho}=\frac{1}{4\pi}\tilde{\theta}_{\hor}(\cdot-x_0)$ on $\partial_{2\e}B_\rho(x_0)$ we obtain a competitor $\hat{u}_{\e,\rho}\in\AD_\e$ with $\hat{u}_{\e,\rho}=\frac{1}{4\pi}\theta_{\hor}(\cdot-x_0)$ on $\partial_{2\e}B_\rho(x_0)$ by setting $\hat{u}_{\e,\rho}\defas u_{\e,\rho}+\frac{z}{2}\mathds{1}_{\{(x-x_0)\cdot e_2<0\}}$. Moreover, $\hat{u}_{\e,\rho}$ satisfies
	\begin{equation*}
	F_\e^\p\big(\hat{u}_{\e,\rho},B_\rho(x_0)\big)\leq F_\e^\p\big(u_{\e,\rho},B_\rho(x_0)\big)+C\rho\,.
	\end{equation*}
Passing to the infimum and letting $\e\to 0$, $\rho\to 0$ we thus deduce~\eqref{est:core-ub} in the general case.
\end{remark}
\begin{remark}\label{rem:est-Fp-Fscrew}
A similar estimate as in Step 5 of the proof of Proposition~\ref{prop:core-partials} holds in the following more general setting. Let $A\subset\R^2$ be a Borel set and $S$ a countable collection of horizontal segments.
Suppose that $u\in\AD_\e$ satisfies
	\begin{align*}
	\d u (i,i+\e e_1)=\dist(\d u (i,i+\e e_1);\Z)\; \text{ for all }i\in\Z_\e^{e_1}(A)
	\end{align*}	 
and
	\begin{align}\label{cond:rem-Fp-Fscrew}
	\dist (\d u (i,i+\e e_2);\Z)=\dist (\d u (i,i+\e e_2);\tfrac{1}{2}\Z)\; \text{ for all }i\in\Z_\e^{e_2}\text{ with }(i,i+\e e_2]\cap S=\emptyset\,.
	\end{align}
Then we have that
	\begin{align}\label{est:rem-Fp-Fscrew}
	4F_\e^\p(u,A)\leq (1+C\e)F_\e^\screw(2u,A)+2\alpha\e\#\{i\in\Z_\e^{2e_2}(A)\colon (i,i+2\e e_2]\cap S\neq\emptyset\}\,.
	\end{align}
To see this, note that~\eqref{eq:energy-core-part-edge} and the first estimate in~\eqref{est:energy-core-jump} remain unchanged. Finally, if $i\in\Z_\e^{2e_2}(A)$ is such that $(i,i+2\e e_2]\cap S=\emptyset$ then a similar argument as in Step 5 above but now using~\eqref{cond:rem-Fp-Fscrew} yields 
	\begin{align*}
	f_1\big(\d u(i,i+2\e e_2)\big) &\leq 2\Big(f_1\big(\d u(i,i,+\e e_2)\big)+f_1\big(\d u(i+\e e_2,i,+2\e e_2)\big)\Big)\\ 
	&=2\Big(f_{\frac{1}{2}}\big(\d u(i,i,+\e e_2)\big)+f_{\frac{1}{2}}\big(\d u(i+\e e_2,i,+2\e e_2)\big)\Big)\\
	&=\frac{1}{2}\Big(f_1\big(\d (2u)(i,i,+\e e_2)\big)+f_1\big(\d (2u)(i+\e e_2,i,+2\e e_2)\big)\Big)\,.
	\end{align*}
Thus we obtain~\eqref{est:rem-Fp-Fscrew} by summing up over all contributions. 
\end{remark}
%%%%%%%%%%%%%%%%%%%%%%%%%%%%%%%%%%%%%%
% Upper Bound
%%%%%%%%%%%%%%%%%%%%%%%%%%%%%%%%%%%%%%
\subsection{Proof of the  Upper Bound}\label{sec:ub}
Based on Proposition~\ref{prop:core-partials} we finally prove Theorem~\ref{thm:main}(iii). 
\begin{proof}[Proof of Theorem~\ref{thm:main}(iii)]
Let $w_{\even},w_{\odd}\in\D_{M,\hor}^{1/2}(\Omega)$ with $w_{\even}^2=w_{\odd}^2=:v$ satisfying $J(v)=\pi\mu$, $\mu=\sum_{h=1}^M d_h\delta_{x_h}$ with $d_h\in\{-1,1\}$, $x_h\in\Omega$. It is not restrictive to assume that $\WW(v,\Omega)<+\infty$. Moreover, we first assume that $w_\even=w_\odd=:w$ and we split the proof of~\eqref{limsup:partial} for such $w$ into several steps.  

\begin{step}{1}{Construction of an approximating sequence}
By applying Proposition~\ref{prop:countable segments}, we find that, up to $\H^1$-negligible sets, $S_w = \bigcup_{j\in\N}(y_j^1,y_j^2) \in \Ten(\mu, \Omega)$ \emph{resolves dislocations tension}. 
Moreover, we let $\varphi$, $\chi$, and $\psi$ be as in Lemma~\ref{lem:lifting-v} and Lemma~\ref{lem:lifting-w}. For any $\sigma>0$ such that $2\sigma$ satisfies~\eqref{cond:sigma} we let $\varphi_n^\sigma\subset C^\infty(\Omega^{\frac{\sigma}{4}}(\mu) \sm \Gamma)\cap H^1(\Omega^{\frac{\sigma}{4}}(\mu) \sm \Gamma)$ be an approximation of $\varphi$ provided by Lemma~\ref{lem:lifting-v}  with $\sigma$ replaced by $\frac{\sigma}{4}$.
We also set $\psi_n^\sigma\defas\frac{\varphi_n^\sigma}{2}+\chi$ and $v_n^\sigma\defas\exp(\iota\varphi_n^\sigma)$. For every $h\in\{1,\ldots,M\}$ we have $\varphi_n^\sigma\in C^\infty(A_{\frac{\sigma}{4},2\sigma}(x_h)\sm\Pi(x_h))$ and $\deg(v_n^\sigma,\partial B_\rho(x_h))=d_h$ for every $\rho\in (\frac{\sigma}{4},2\sigma)$, and thus we can find a lifting  $\theta_h\in C^\infty(A_{\frac{\sigma}{4},2\sigma}(x_h)\sm\Pi(x_h))$ of a rotation of $\frac{x-x_h}{|x-x_h|}$ such that 
	\begin{equation}\label{choice:lifting-theta-h}
	\varphi_n^\sigma-d_h\theta_h\in H^1\big(A_{\frac{\sigma}{4},2\sigma}(x_h)\big)\quad\text{and}\quad \fint_{A_{\frac{\sigma}{2},\sigma}(x_h)}\big(\varphi_n^\sigma-d_h\theta_h\big) \, \d x=0  \,.
	\end{equation}
Let now  $\eta\in C^\infty\big([0,+\infty);[0,1]\big)$ with $\eta\equiv 0$ on $[0,\frac{5}{8}]$ and $\eta\equiv 1$ on $[\frac{7}{8},+\infty)$ be a smooth cut-off function and for $h\in\{1,\ldots,M\}$ set
\begin{align*}
	\vartheta_n^{\sigma,h}(x)\defas d_h\theta_h(x) + \eta\bigg(\dfrac{|x-x_h|}{\sigma}\bigg)\big(\varphi_n^\sigma(x) - d_h\theta_h(x)\big)\;\text{ for every}\ x\in A_{\frac{\sigma}{4},2\sigma} (x_h) \,.
	\end{align*}
By the choice of $\theta_h$ we have that $\vartheta_n^{\sigma,h}\in C^\infty(A_{\frac{\sigma}{4},2\sigma} (x_h)\sm \Pi(x_h))$. Moreover, we have
	\begin{equation}\label{bc:phi-n-k}
	\begin{split}
	\vartheta_n^{\sigma,h} \equiv d_h\theta_h\;\text{ on}\  A_{\frac{\sigma}{4},\frac{5\sigma}{8}}(x_h) \quad\text{and}\quad \vartheta_n^{\sigma,h} \equiv\varphi_n^\sigma\;\text{ on}\ A_{\frac{7\sigma}{8},2\sigma}(x_h)\,.
	\end{split}
	\end{equation}
Eventually, for every $h\in\{1,\ldots,M\}$ by Remark~\ref{rem:invariance rotation} we let $u_{\e}^{\sigma,h}\in\AD_\e$ be such that $u_{\e}^{\sigma,h}(i)=\frac{1}{4\pi}\theta_h(i)$ on $\partial_{2\e}B_{\frac{\sigma}{2}}(x_h)$ and $F_\e^\p\big(u_{\e}^{\sigma,h},B_{\frac{\sigma}{2}}(x_h)\big)=\gamma_\e^\p\big(B_{\frac{\sigma}{2}}(x_h)\big)$. We define $u_{\e,n}^\sigma$ on $\e\Z^2\cap\Omega$ by setting 
	\begin{equation}\label{def:recovery} 
	u_{\e,n}^\sigma(i)\defas
	\begin{cases} 
	d_hu_{\e}^{\sigma,h}(i)+\dfrac{1}{2\pi} \chi(i) &\text{if}\ i\in\e\Z^2\cap B_{\frac{\sigma}{2}}(x_h)\ \text{for an}\ h\in\{1,\ldots,M\}\,,\\\cr
	\dfrac{1}{4\pi} \vartheta_n^{\sigma,h}(i) + \dfrac{1}{2\pi} \chi(i) &\text{if}\ i\in\e\Z^2\cap \overline{\Asigma(x_h)}\ \text{for an}\ h\in\{1,\ldots,M\}\,,\\\cr
	\dfrac{1}{4\pi} \varphi_n^\sigma(i) + \dfrac{1}{2\pi}\chi(i) = \dfrac{1}{2\pi}\psi_n^\sigma(i) &\text{if}\ i\in\e\Z^2\cap \Omega^{\sigma}(\mu)\,.
	\end{cases} 
	\end{equation} 
In the previous formula, if $i$ lies on the discontinuity set of $\chi$, $\theta_n^{\sigma,h}$ or $\varphi_n^\sigma$, we use as value the trace from above. To define $u_{\e,n}^\sigma$ on $\e\Z^2\sm\Omega$, we extend $v_{\e,n}^\sigma\defas\exp(2\pi\iota 2 u_{\e,n}^\sigma)$ to $\e\Z^2\sm\Omega$ as in~\eqref{est:extension} and we choose $u_{\e,n}^\sigma$ to be an angular lifting of $\frac{1}{2} v_{\e,n}^\sigma$ on $\e\Z^2\sm\Omega$.
By the definition of $u_{\e,n}^\sigma$ and thanks to~\eqref{bc:phi-n-k} we have that
	\begin{equation}\label{splitting-energy-recovery}
	F_\e^\p\big(u_{\e,n}^\sigma,\Omega\big) \leq \sum_{h=1}^M F_\e^\p\big(u_{\e,n}^\sigma,B_\sigma(x_h)\big)+F_\e^\p\big(\tfrac{1}{2\pi}\psi_n^\sigma,\Omega^{\sigma-2\e}(\mu)\big) \, ,
	\end{equation} 
and below we estimate separately the two contributions on the right-hand side of~\eqref{splitting-energy-recovery}.
\end{step}

\begin{step}{2}{Energy estimate in $B_{\sigma}(x_h)$}
In this step we show that for every $h\in\{1,\ldots,M\}$ we have
	\begin{equation}\label{est:recovery-core}
	\limsup_{\sigma\to0}\limsup_{n\to+\infty}\limsup_{\e\to 0}\Big(4F_\e^\p\big(u_{\e,n}^\sigma,\overline{B_{\sigma}(x_h)}\big)-\pi\log\frac{\sigma}{\e}\Big)\leq \gamma\,.
	\end{equation}
%We start estimating the energy on $B_{\frac{\sigma}{2}}(x_1)$ by comparing it with $F_\e^\p(u_{\e}^\sigma,B_{\frac{\sigma}{2}}(x_1))$. 
%Since 
%	\begin{equation}\label{limsup:Spsi}
%	\chi(i)\in\{0,\pi\}\quad\text{and}\quad S_{\chi}\cap B_{\sigma}(x_1)\subset \Gamma_1\cap B_{\sigma}(x_1)\,,
%	\end{equation}
Let $h\in\{1,\ldots,M\}$ be fixed. Since the discontinuity set of $\chi$ is horizontal, for all $i\in\Z_\e^{e_1}(B_{\frac{\sigma}{2}}(x_h))$ we have that $\d u_{\e,n}^\sigma(i,i+\e e_1)=d_h\d u_{\e}^{\sigma,h}(i,i+\e e_1)$. Moreover, by Lemma~\ref{lem:lifting-w}, $S_\chi \cap B_{\sigma}(x_h) \subset \Pi(x_h) \cup S_w$. Possibly fixing a smaller $\sigma$, since $S_w$ \emph{resolves dislocations tension} we have that the conditions in Definition~\ref{def:resolve} are satisfied, so that, in particular, $S_w \cap B_{\sigma}(x_h)  \subset  \Pi(x_h)$.  Hence, if $i\in\Z_\e^{2e_2}(B_{\frac{\sigma}{2}}(x_h))$ with $[i,i+2\e e_2]\cap \Pi(x_h)=\emptyset$, we have that $\d u_{\e,n}^\sigma(i,i+2\e e_2)=d_h\d u_{\e}^{\sigma,h}(i,i+2\e e_2)$. In a similar way we can argue on $\Asigma(x_h)$ comparing $\d u_{\e,n}^\sigma$ with $\d \vartheta_n^{\sigma,h}$. Using in addition that $\chi\in\{0,\pi\}$ and $|d_h|=1$ we thus obtain
	\begin{equation}\label{est:recovery-core1}
	\begin{split}
	F_\e^\p\big(u_{\e,n}^\sigma,\overline{B_{\sigma}(x_h)}\big) &\leq F_\e^\p\big(u_{\e}^{\sigma,h},B_{\frac{\sigma}{2}}(x_h)\big)+F_\e^\p\big(\tfrac{1}{4\pi}\vartheta_n^{\sigma,h},A_{\frac{\sigma}{2}-2\e,\sigma+2\e}(x_h)\big)+C\sigma\\
	&\leq\gamma_\e^\p\big(B_{\frac{\sigma}{2}}(x_h)\big)+F_\e^\p\big(\tfrac{1}{4\pi}\vartheta_n^{\sigma,h},A_{\frac{\sigma}{2}-2\e,\sigma+2\e}(x_h)\big)+C\sigma\,,
	\end{split}
	\end{equation}
where the second estimate follows from the choice of $u_{\e}^{\sigma,h}$.
Moreover, since $\vartheta_n^{\sigma,h}\in C^\infty(A_{\frac{\sigma}{2}-2\e,\sigma}(x_h)\sm \Pi(x_h))$, in a similar way as in  Step~5 of the proof of Proposition~\ref{prop:core-partials}  (see also Remark~\ref{rem:est-Fp-Fscrew}) we deduce that
	\begin{equation}\label{est:core-cutoff-1}
	4 F_\e^\p\big(\tfrac{1}{4\pi}\vartheta_n^{\sigma,h},A_{\frac{\sigma}{2}-2\e,\sigma+2\e}(x_h)\big)\leq(1+\e)F_\e^\screw\big(\tfrac{1}{2\pi}\vartheta_n^{\sigma,h},A_{\frac{\sigma}{2}-2\e,\sigma+2\e}(x_h)\big)+C(\sigma+\e)\,.
	\end{equation}
By interpolation estimates we have that
	\begin{equation}\label{est:core-phi-interpolation}
	\limsup_{\e\to 0}F_\e^\screw\big(\tfrac{1}{2\pi}\vartheta_n^{\sigma,h},A_{\frac{\sigma}{2}-2\e,\sigma+2\e}(x_h)\big)\leq\frac{1}{2}\int_{A_{\frac{\sigma}{2},\sigma}(x_h)} |\nabla\vartheta_n^{\sigma,h}|^2\dx\,.
	\end{equation}
It remains to show that
	\begin{equation}\label{est:recovery-cutoff-final}
	\limsup_{\sigma\to 0}\limsup_{n\to+\infty}\frac{1}{2}\int_{A_{\frac{\sigma}{2},\sigma}(x_h)} |\nabla\vartheta_n^{\sigma,h}|^2\dx\leq\pi\log 2\,,
	\end{equation}
then combining~\eqref{est:recovery-core1}--\eqref{est:recovery-cutoff-final} yields
	\begin{multline*}
	\limsup_{\sigma\to0}\limsup_{n\to+\infty}\limsup_{\e\to 0}\Big(4F_\e^\p\big(u_{\e,n}^\sigma,B_{\sigma}(x_h)\big)-\pi\log\frac{\sigma}{\e}\Big)\\
	\leq\limsup_{\sigma\to 0}\limsup_{\e\to 0}\Big(4\gamma_\e^\p\big(B_{\frac{\sigma}{2}}(x_h)\big)-\pi\log\frac{\sigma}{2\e}\Big)\,,
	\end{multline*}
thus~\eqref{est:recovery-core} follows from Proposition~\ref{prop:core-partials}.
To obtain~\eqref{est:recovery-cutoff-final} we write
	\begin{equation}\label{eq:core-phi-split-grad}
    \begin{split}
      \int_{A_{\frac{\sigma}{2},\sigma}(x_h)} |\nabla\vartheta_n^{\sigma,h}|^2\dx = \int_{A_{\frac{\sigma}{2},\sigma}(x_h)} |\nabla\varphi_n^\sigma|^2\dx & + \int_{A_{\frac{\sigma}{2},\sigma}(x_h)} |\nabla\vartheta_n^{\sigma,h}-\nabla\varphi_n^\sigma|^2\dx \\
      & + 2\int_{A_{\frac{\sigma}{2},\sigma}(x_h)} \nabla\varphi_n^\sigma\cdot(\nabla\vartheta_n^{\sigma,h}-\nabla\varphi_n^\sigma)\dx
    \end{split}
	\end{equation}
and we observe that thanks to~\eqref{cond:limit-dyadic-annulus}, Lemma~\ref{lem:lifting-v}, and the fact that $\WW(v,\Omega)<+\infty$ we have
	\begin{equation}\label{est:recovery-cutoff1}
	\lim_{\sigma\to 0}\lim_{n\to+\infty}\frac{1}{2}\int_{A_{\frac{\sigma}{2},\sigma}(x_h)} |\nabla\varphi_n^\sigma|^2\dx=\lim_{\sigma\to 0}\frac{1}{2}\int_{\Asigma(x_h)} |\nabla\varphi|^2\dx=\pi\log 2\,.
	\end{equation}
It is thus left to show that the remaining contributions in~\eqref{eq:core-phi-split-grad} are negligible. We have
	\begin{multline*}
	\nabla\varphi_n^\sigma(x)-\nabla\vartheta_n^{\sigma,h}(x)=\bigg(1-\eta\Big(\frac{|x-x_h|}{\sigma}\Big)\bigg)\big(\nabla\varphi_n^\sigma(x)-d_h\nabla\theta_h(x)\big)\\
	+\frac{1}{\sigma}\eta'\Big(\frac{|x-x_h|}{\sigma}\Big)\frac{x-x_h}{|x-x_h|}\big(d_h\theta_h(x)-\varphi_n^\sigma(x)\big)
	\end{multline*}
for $x\in A_{\frac{\sigma}{2},\sigma}(x_h)$. Thanks to~\eqref{choice:lifting-theta-h}, an application of Poincar\'e inequality together with~\eqref{limit-dyadic-lifting} thus yields
	\begin{equation*}
	\int_{A_{\frac{\sigma}{2},\sigma}(x_h)} |\nabla\vartheta_n^{\sigma,h}-\nabla\varphi_n^\sigma|^2\dx\leq C\int_{A_{\frac{\sigma}{2},\sigma}(x_h)} |d_h\nabla\theta_h-\nabla\varphi_n^\sigma|^2\dx\ \to 0\;\text{ as $n\to +\infty$ and $\sigma\to 0$.}
	\end{equation*}
Thus~\eqref{est:recovery-cutoff-final} follows from~\eqref{eq:core-phi-split-grad} and~\eqref{est:recovery-cutoff1} together with an application of H\"older's inequality. 
\end{step}

\begin{step}{3}{Energy estimate on $\Omega^\sigma(\mu)$}
In this step we show that
	\begin{equation}\label{est:recovery-far-field}
	\limsup_{n\to+\infty}\limsup_{\e\to 0}4F_\e^\p\big(\tfrac{1}{2\pi}\psi_n^\sigma,\Omega^{\sigma-2\e}(\mu)\big)\leq\frac{1}{2}\int_{\Omega^\sigma(\mu)}|\nabla v|^2\dx+ 4 \alpha \H^1(S_w)\,.
	\end{equation}
%with $r_2(\e,n,\sigma)\to 0$ as $\e\to 0$, $n\to+\infty$, and $\sigma\to 0$.
%
%\smallskip
In view of~\eqref{cond:jump w - jump psi} we find that $\psi_n^\sigma$ satisfies~\eqref{cond:rem-Fp-Fscrew} in Remark~\ref{rem:est-Fp-Fscrew} with $A=\Omega^{\sigma-2\e}(\mu)$ and $S=S_w$ (recall that $[\varphi_n^\sigma]=[\varphi]$ for $n$ sufficiently large and thus~\eqref{cond:jump w - jump psi} holds with $\psi_n^\sigma$ in place of $\psi$). Thus~\eqref{est:rem-Fp-Fscrew} implies that 
%Using~\eqref{eq:F-edge-F-part} and the fact that $\chi\in\{0,\pi\}$, we find that
%	\begin{equation}\label{est:recovery-far-field1}
%	\begin{split}
%	&4F_\e^\p\big(\tfrac{1}{2\pi}\psi_n^\sigma,\Omega^{\sigma-2\e}(\mu)\big)=F_\e^\edge\big(\tfrac{1}{2\pi} \varphi_n^\sigma,\Omega^{\sigma-2\e}(\mu)\big)+ \BBB \frac{\alpha}{\pi^2} \e \EEE \hspace*{-1em}\sum_{i\in\Z_\e^{2e_2}(\Omega^{\sigma-2\e}(\mu))}\hspace*{-1.5em}f_1\big(\tfrac{1}{2\pi}\d\psi_n^\sigma(i,i+2\e e_2)\big)\\
%	&\hspace*{1em}\leq(1+\e)F_\e^\screw\big(\tfrac{1}{2\pi} \varphi_n^\sigma,\Omega^{\sigma-2\e}(\mu)\big) +  \BBB 2 \alpha \e \EEE \#\bigg\{i\in\Z_\e^{2e_2}(\Omega^{\sigma-2\e}(\mu))\colon (i,i+2\e e_2]\cap \bigcup_{j\in\N} (y_j^1,y_j^2)\neq \emptyset\bigg\}
%	\end{split}
%	\end{equation}
	\begin{equation}\label{est:recovery-far-field1}
	\begin{split}
	4F_\e^\p\big(\tfrac{1}{2\pi}\psi_n^\sigma,\Omega^{\sigma-2\e}(\mu)\big) &\leq(1+C\e)F_\e^\screw\big(\tfrac{1}{\pi} \psi_n^\sigma,\Omega^{\sigma-2\e}(\mu)\big) \\
	&+  2 \alpha \e \#\bigg\{i\in\Z_\e^{2e_2}(\Omega^{\sigma-2\e}(\mu))\colon (i,i+2\e e_2]\cap \bigcup_{j\in\N} (y_j^1,y_j^2)\neq \emptyset\bigg\}
	\end{split}
	\end{equation}
%where the estimate in the second line follows from~\eqref{cond:jump w - jump psi} together with a similar argument as in Step 5 of the proof of Proposition~\ref{prop:core-partials}. 
Moreover, since $\chi$ takes values in $\{0,\pi\}$, we have that $\frac{1}{\pi}\psi_n^\sigma\eqZ\frac{1}{2\pi}\varphi_n^\sigma$, hence
	\begin{align}\label{eq:recovery-far-modZ}
	F_\e^\screw\big(\tfrac{1}{\pi} \psi_n^\sigma,\Omega^{\sigma-2\e}(\mu)\big)=F_\e^\screw\big(\tfrac{1}{2\pi} \varphi_n^\sigma,\Omega^{\sigma-2\e}(\mu)\big)\,.
	\end{align}\EEE
Again by interpolation estimates we have that
	\begin{equation}\label{est:recovery-far-interpolation}
	\limsup_{n\to +\infty}\limsup_{\e\to 0}F_\e^\screw\big(\tfrac{1}{2\pi} \varphi_n^\sigma,\Omega^{\sigma-2\e}(\mu)\big)\leq\lim_{n\to+\infty}\frac{1}{2}\int_{\Omega^\sigma(\mu)}|\nabla\varphi_n^\sigma|^2\dx=\frac{1}{2}\int_{\Omega^\sigma(\mu)}|\nabla v|^2\dx\,,
	\end{equation}
where the last equality follows from Lemma~\ref{lem:lifting-v} (v) together with the equality $|\nabla\varphi|=|\nabla v|$ in $\Omega^\sigma(\mu)$. Moreover, for any $j\in\N$ we have 
	\begin{equation*}
	\#\big\{i\in\Z_\e^{2e_2}(\Omega^{\sigma-2\e}(\mu))\colon (i,i+2\e e_2]\cap (y_j^1, y_j^2)\neq \emptyset\big\}\leq\frac{2(y_j^2-y_j^1)}{\e}\,,
	\end{equation*}
which implies that
	\begin{equation}\label{est:recovery-surface}
	\begin{split}
	& 2 \alpha \e  \#\bigg\{i\in\Z_\e^{2e_2}(\Omega^{\sigma-2\e}(\mu))\colon (i,i+2\e e_2]\cap \bigcup_{j\in\N} (y_j^1,y_j^2)\neq \emptyset\bigg\}\\
	&\hspace*{1em}\leq 2 \alpha \e \sum_{j\in\N}\#\big\{i\in\Z_\e^{2e_2}(\Omega^{\sigma-2\e}(\mu))\colon (i,i+2\e e_2]\cap  (y_j^1,y_j^2)\neq \emptyset\big\}\\
  &\hspace*{1em}\leq\sum_{j\in\N} 4 \alpha (y_j^2-y_j^1)=  4 \alpha  \H^1(S_w)\,.
	\end{split}
	\end{equation}
Eventually, \eqref{est:recovery-far-field} follows by gathering~\eqref{est:recovery-far-field1}--\eqref{est:recovery-surface}.
\end{step}

\begin{step}{4}{Conclusion}
Combining~\eqref{est:recovery-core} and~\eqref{est:recovery-far-field} we find that
	\begin{equation}\label{est:recovery-complete}
	\begin{split}
	&\limsup_{\sigma\to 0}\limsup_{n\to+\infty}\limsup_{\e\to 0}\big(4 F_\e^\p(u_{\e,n}^\sigma,\Omega)-M\pi|\log\e|\big)\\
	&\hspace*{1em}\leq \limsup_{\sigma\to 0}\bigg(\frac{1}{2}\int_{\Omega^\sigma(\mu)}|\nabla v|^2\dx-M\pi|\log\sigma|\bigg)+M\gamma+ 4 \alpha  \H^1(S_w)\\
	&\hspace*{1em}= \WW(v,\Omega)+M\gamma+ 4 \alpha  \H^1 (S_w)\,.
	\end{split}
	\end{equation}
Moreover, we have $\mu_{2u_{\e,n}^\sigma}\mres\Omega\fto\mu$ as $\e\to 0$, $\sigma\to 0$.
We finally set $w_{\e,n}^\sigma\defas \exp(2\pi\iota u_{\e,n}^\sigma)$ and we claim that
	\begin{equation}\label{convergence:recovery}
	\|w_{\e,n}^\sigma-w\|_{L^1(\Omega;\R^2)}\to 0\;\text{ as}\ \e\to 0\,,\ n\to\infty\,,\ \sigma\to 0\,.
	\end{equation}
Note that from~\eqref{convergence:recovery} we can deduce the corresponding convergence of $w_{2\e,s_j}^{n,\sigma}$ using the same arguments as in the proof of~\eqref{eq:w-even-w-odd}.
Thus, once the claim is established we can conclude via a diagonal argument. To obtain~\eqref{convergence:recovery} it is convenient to set $w_n^\sigma(x)\defas\exp\big(\iota\big(\frac{\varphi_n^\sigma(x)}{2} + \chi(x)\big)\big)$ for every $x\in\Omega^\sigma(\mu)$, so that $w_n^\sigma(i)=w_{\e,n}^\sigma(i)$ for every $i\in\e\Z^2\cap\Omega^\sigma(\mu)$.
We have that
	\begin{equation}\label{conv:recovery-decomposition-cubes}
	\begin{split}
	\|w_{\e,n}^\sigma-w\|_{L^1(\Omega;\R^2)} &=\sum_{Q_\e(i)\in\Q_\e}\int_{\Omega\cap Q_\e(i)}|w_{\e,n}^\sigma(i)-w(x)|\dx\\
	&\leq \sum_{Q_\e(i)\in\Q_\e(\Omega^\sigma(\mu)\sm S_{w_n^\sigma})}\int_{Q_\e(i)}|w_{n}^\sigma(i)-w(x)|\dx+2\hspace*{-1.5em}\dsum{Q_\e\in\Q_\e}{Q_\e\cap\R^2\sm\Omega^\sigma(\mu)\neq\emptyset}\hspace*{-1.5em}|Q_\e\cap\Omega|\\
	&\hspace*{1em}+2\e^2\#\big\{Q_\e\in\Q_\e(\Omega^\sigma(\mu))\colon Q_\e\cap S_{w_n^\sigma}\neq\emptyset\big\}\,,
	\end{split}
	\end{equation}
where we have used that $|w_{\e,n}^\sigma|, |w|\leq 1$.
For every $Q_\e\in\Q_\e$ with $Q_\e\cap\R^2\sm\Omega^\sigma(\mu) \neq \emptyset$, the inclusion $Q_\e\cap\Omega\subset\{\dist(x,\partial\Omega)\leq\sqrt{2}\e\}\cup\bigcup_{h=1}^MB_{\sigma+\sqrt{2}\e}(x_h)$ holds, from which we deduce that
	\begin{equation}
	\dsum{Q_\e\in\Q_\e}{Q_\e\cap\R^2\sm\Omega^\sigma(\mu)\neq\emptyset}\hspace*{-1.5em}|Q_\e\cap\Omega| \leq \big|\{x\in\Omega\colon\dist(x,\partial\Omega)\leq\sqrt{2}\e\}\big|+M\big|B_{\sigma+\sqrt{\e}}\big|\to 0\;\text{ as $\e\to 0$, $\sigma\to 0$,}
	\end{equation}
where we have used that $\partial\Omega$ is Lipschitz and thus admits and $(n-1)$-dimensional Minkowsky content.
Similarly, we deduce that
	\begin{equation}\label{conv:recovery-L1-jump}
	\e^2\#\big\{Q_\e\in\Q_\e(\Omega^\sigma(\mu))\colon Q_\e\cap S_{w_n^\sigma}\neq\emptyset\big\}\leq C\e\H^1(S_{w_n^\sigma}\cap\Omega^{\sigma}(\mu)) \to 0\;\text{ as}\ \e\to 0\,.
	\end{equation}		
	
The remaining term in~\eqref{conv:recovery-decomposition-cubes} can be estimated  by observing the following. For any $Q_\e(i)\in\Q_\e(\Omega^\sigma(\mu)\sm S_{w_n^\sigma})$ and any $x\in Q_\e(i)$ we have that $|w_n^\sigma(i)-w_n^\sigma(x)|\leq \sqrt{2}\e\|\nabla w_n^\sigma\|_{L^\infty(\Omega^\sigma(\mu);\R^{2\times 2})}$. From this we infer that
	\begin{equation}\label{conv:recovery-af}
	\begin{split}
	\sum_{Q_\e(i)\in\Q_\e(\Omega^\sigma(\mu)\sm S_{w_n^\sigma})}\int_{Q_\e(i)}|w_{n}^\sigma(i)-w(x)|\dx\leq \sqrt{2}\e\|\nabla w_n^\sigma\|_{L^\infty(\Omega^\sigma(\mu))}|\Omega^\sigma(\mu)|+\|w_n^\sigma-w\|_{L^1(\Omega^\sigma(\mu))}
	\end{split}
	\end{equation}
Finally, $|w_n^\sigma(x)-w(x)|\leq\big|\exp\big(\iota\frac{\varphi_n^\sigma(x)}{2}\big)-\exp\big(\iota\frac{\varphi(x)}{2}\big)\big|$ for any $x\in\Omega^\sigma(\mu)$, which together with an application of H\"older's inequality and Lemma~\ref{lem:lifting-v} implies that $\|w_n^\sigma-w\|_{L^1(\Omega^\sigma(\mu))}\to 0$ as $n\to \infty$. Together with~\eqref{conv:recovery-decomposition-cubes}--\eqref{conv:recovery-af} this gives~\eqref{convergence:recovery} and we conclude.
\end{step}
\end{proof}
\begin{remark}[The case $w_{\even}\neq w_{\odd}$]
In this case we choose $\varphi$, $\varphi_n^\sigma$ as in Lemma~\ref{lem:lifting-v} with $v=w_\even^2=w_\odd^2$ and we let $\chi_\even$, and $\chi_\odd$ be as in Lemma~\ref{lem:lifting-w} applied with $w_\even$, $w_\odd$, respectively. We then define $u_{\e,n}^{\sigma,\even}$, $u_{\e,n}^{\sigma,\odd}$ according to~\eqref{def:recovery} with $\chi$ replaced by $\chi_\even$, $\chi_\odd$, respectively. We finally set 
	\begin{equation*}
	u_{\e,n}^\sigma(i)\defas
	\begin{cases}
	u_{\e,n}^{\sigma,\even}(i) &\text{if}\ i\in 2\e\Z^2_\even\,,\\
	u_{\e,n}^{\sigma,\odd}(i) &\text{if}\ i\in 2\e\Z^2_\odd
	\end{cases}
	\end{equation*}
with $2\e\Z^2_\even$, $2\e\Z^2_\odd$ as in~\eqref{def:Z-even-Z-odd}.
In this way, $u_{\e,n}^\sigma$ still satisfies~\eqref{est:recovery-core} and it remains to estimate $F_\e^\p(u_{\e,n}^\sigma, \Omega^{\sigma-2\e}(\mu))$. Since $i\in 2\e\Z^2_\even$ implies that $i+\e e_1\in 2\e\Z^2_\even$, we obtain
	\begin{equation}\label{eq:du-e1-even}
	\d u_{\e,n}^\sigma(i,i+\e e_1)=\d u_{\e,n}^{\sigma,\even}(i,i+\e e_1)=\frac{1}{4\pi}\d \varphi_n^\sigma(i,i+\e e_1)
	\end{equation}
for every $i\in 2\e\Z^2_\even$ with $[i,i+\e e_1]\subset\Omega^{\sigma-2\e}(\mu)$. Here the second equality follows thanks to~\eqref{def:recovery} together with the fact that $S_{\chi_\even}$ and $S_{\chi_\odd}$ are horizontal. Similarly, we have
	\begin{equation}\label{eq:du-2e2-even}
	\d u_{\e,n}^\sigma(i,i+2\e e_2)=\d u_{\e,n}^{\sigma,\even}(i,i+2\e e_2)=\frac{1}{2\pi}\d \psi_n^{\sigma,\even}(i,i+2\e e_2)
	\end{equation}
for every $i\in 2\e\Z^2_\even$ with $[i,i+2\e e_2]\subset\Omega^{\sigma-2\e}(\mu)$. Using finally that both $\chi_{\even}$ and $\chi_{\odd}$ take values in $\{0,\pi\}$ we deduce that
	\begin{equation}\label{eq:du-e2-even}
	\begin{split}
	\d u_{\e,n}^\sigma(i,i+\e e_2) &=\frac{1}{4\pi}\d \varphi_n^\sigma(i,i+\e e_2)+\frac{1}{2\pi}\big(\chi_\odd(i+\e e_2)-\chi_\even(i)\big)\\
	&\stackrel{\mod \frac{1}{2}\Z}{=} \frac{1}{4\pi}\d \varphi_n^\sigma(i,i+\e e_2)
	\end{split}
	\end{equation}
for every $i\in 2\e\Z^2_\even$ with $[i,i+\e e_1]\subset\Omega^{\sigma-2\e}(\mu))$. Since the analogues of~\eqref{eq:du-e1-even}--\eqref{eq:du-e2-even} hold with $2\e\Z^2_\even$ replaced by $2\e\Z^2_\odd$, we get
	\begin{align*}
	4 F_\e^\p\big(u_{\e,n}^\sigma,\Omega^{\sigma-2\e}(\mu)\big)=F_\e^{\edge}\big(\tfrac{1}{2\pi}\varphi_n^\sigma,\Omega^{\sigma-2\e}(\mu)\big)&+\frac{\alpha}{\pi^2}  \e \hspace*{-2.5em}\dsum{i\in 2\e\Z^2_\even}{ [i,i+2\e e_2]\subset\Omega^{\sigma-2\e}(\mu)}\hspace*{-2.5em}f_1\big(\tfrac{1}{2\pi}\d\psi_n^{\sigma,\even}(i,i+2\e e_2)\big)\\
	&+\frac{\alpha}{\pi^2}  \e \hspace*{-2.5em}\dsum{i\in 2\e\Z^2_\odd}{ [i,i+2\e e_2]\subset\Omega^{\sigma-2\e}(\mu)}\hspace*{-2.5em}f_1\big(\tfrac{1}{2\pi}\d\psi_n^{\sigma,\odd}(i,i+2\e e_2)\big)\,.
	\end{align*}
Arguing as in~\eqref{est:recovery-far-field1}--\eqref{est:recovery-surface} it is then immediate to see that $u_{\e,n}^\sigma$ satisfies~\eqref{est:recovery-far-field} with $4 \alpha \H^1(S_w)$ replaced by $2 \alpha\H^1(S_{w_\even})+  2 \alpha  \H^1(S_{w_\odd})$.
\end{remark}
\section{Proof of Theorem~\ref{thm:main2}}\label{sec:proof-main-measure}

As a consequence of Theorem~\ref{thm:main} we now obtain Theorem~\ref{thm:main2}. 
Recall that for any $\mu\in X(\Omega)$ the discrete energies $\F_\e^\p(\mu,\Omega)$ are given by
	\begin{equation*}\label{def:energy-mu}
	\F_\e^\p(\mu,\Omega)\defas\inf\big\{F_\e^\p(u,\Omega)\colon u\in\AD_\e\,,\ \mu_{2 u}\mres\Omega=\mu\big\}
	\end{equation*}
with the convention $\inf\emptyset\defas+\infty$. 

For any $\mu = \sum_{i=1}^M d_h \delta_{x_h} \in X_M(\Omega)$, we define the line-tension energy 
\begin{equation} \label{eq:L}
  \LL(\mu,\Omega) \defas \inf \big\{ \H^1(S) : \ S \in \Ten(\mu,\Omega) \, , \ S \text{ \emph{resolves dislocation tension}}\big\} \, , 
\end{equation}
for which we refer to Definitions~\ref{def:connect singularities} and~\ref{def:resolve}. 

\begin{remark}
  The infimum in~\eqref{eq:L} is, in fact, a minimum and is reached on a stacking fault decomposed in a finite union of \emph{indecomposable stacking faults}. Indeed, by Definition~\ref{def:connect singularities}, a stacking fault $S \in \Ten(\mu,\Omega)$ is (up to a $\H^1$-negligible set that does not affect $\H^1(S)$) a countable union $S = \bigcup_{j \in \N} (y_j^1, y_j^2)$ of \emph{indecomposable stacking faults} $(y_j^1, y_j^2)$. By Remark~\ref{rem:countable components}, the set of indices $j \in \N$ such that at least one of $y_j^1, y_j^2$ lies on $\supp \mu$ is finite. The rest of the \emph{indecomposable stacking faults} connect two boundary points. More precisely, there are two sets of indices, $I_1$ finite and $I_2$ at most countable, such that $S = S_1 \cup S_2$ with $S_1 = \bigcup_{j \in I_1} (y_j^1, y_j^2)$ and  $S_2 = \bigcup_{j \in I_2} (y_j^1, y_j^2)$. Moreover, for $j \in I_2$, $(y_j^1, y_j^2)$ connects boundary points. Since all the  \emph{indecomposable stacking faults} are pairwise disjoint, we have that $\H^1(S) \geq \H^1(S_1)$. Finally, $S_1$ still \emph{resolves dislocation tension} in the sense of Definition~\ref{def:resolve}, since only \emph{indecomposable stacking faults} connecting two boundary points were removed from $S$. Hence, the infimum in~\eqref{eq:L} is taken over the finite family of finite unions of \emph{indecomposable stacking faults} connecting two dislocations or a dislocation to the boundary, \emph{resolving dislocation tension}. This proves the claim.
\end{remark} 
 
For any $\mu\in X_M(\Omega)$ the limiting energies are defined~by
	\begin{equation*}\label{def:limit-mu}
	\F^\p(\mu,\Omega)=\frac{M}{4}\gamma+ \frac{1}{4}\W(\mu,\Omega)+\alpha\LL(\mu,\Omega) \, .
	\end{equation*}

As a preliminary step, we characterise the quantity $\LL(\mu,\Omega)$ in terms of functions $w\in\D_{M,\hor}^{1/2}(\Omega)$ as follows. For any $\mu\in X_M(\Omega)$ we set
	\begin{equation}\label{def:L-mu}
	\Lambda(\mu,\Omega)\defas\inf\big\{\H^1(S_w)\colon w\in \D_{M,\hor}^{1/2}(\Omega)\,,\ w^2=v_\mu\ \text{a.e.\ in}\ \Omega\big\}\,,
	\end{equation}
where $v_\mu$ is the canonical harmonic map associated to $\mu$ (see Section~\ref{sec:limit-fields}). Moreover, we introduce 
\begin{equation}\label{def:L-mu-prime}
	\Lambda'(\mu,\Omega)\defas\inf\big\{\H^1(S_w)\colon w\in\D_{M,\hor}^{1/2}(\Omega)\,,\ J(w^2)=\pi\mu\big\} \, .
	\end{equation} 

Then the following result holds true.  
\begin{lemma}\label{lem:min-length}
For any $\mu\in X_M(\Omega)$  the minimisation problem defining $\Lambda(\mu,\Omega)$ in~\eqref{def:L-mu} admits a solution $w_\mu\in \D_{M,\hor}^{1/2}(\Omega)$. Moreover, $\Lambda(\mu,\Omega)=\Lambda'(\mu,\Omega)=\LL(\mu,\Omega)$.
\end{lemma}
\begin{proof}
The proof is split into two steps establishing separately the existence of a minimiser and the equality $\Lambda(\mu,\Omega)=\Lambda'(\mu,\Omega)=\LL(\mu,\Omega)$.

\begin{step}{1}{Existence of a minimiser $w_\mu$}
Similar to~\cite[Section 4]{GMM20}, the existence of a minimiser $w_\mu$ follows by the direct method. Indeed, thanks to Lemma~\ref{lem:lifting-v} there exists at least one competitor $w\in\D_{M,\hor}^{1/2}(\Omega)$ satisfying $w^2=v_\mu$, hence $\Lambda(\mu,\Omega)<+\infty$. Moreover, if $(w_k)_k \subset \D_{M,\,\hor}^{1/2}(\Omega)$ is a minimising sequence, then $|\nabla w_k|=\frac{1}{2}|\nabla v_\mu|$ a.e.\ in $\Omega$. Since $v_\mu\in W^{1,p}(\Omega;\S^1)$ for every $p\in[1,2)$, we obtain that
	\begin{equation*}
	\sup_{k\in\N}\bigg(\int_{\Omega}|\nabla w_k|^p\dx+\H^1(S_{w_k})\Big)<+\infty\;\text{ for any}\  p \in [1,2)\,. 
	\end{equation*}
Together with the uniform bound $\|w_k\|_{L^\infty(\Omega;\R^2)}=1$ and the $SBV$-compactness result~\cite[Theorem 4.8]{AFP} we deduce that up to subsequences (not relabeled) $w_k\wsto w_\mu$ for some $w_\mu\in SBV(\Omega;\R^2)$. Up to passing to a further subsequence we can additionally assume that $w_k\to w_\mu$ a.e.\ in $\Omega$, which in particular implies that $|w_\mu|=1$ and $w_\mu^2=v_\mu$ a.e.\ in $\Omega$. Moreover, the lower semicontinuity result~\cite[Theorem 4.7]{AFP} ensures that
	\begin{equation*}
	\H^1(S_{w_\mu}) \leq\liminf_{k\to+\infty}\H^1(S_{w_k})=\Lambda(\mu,\Omega)<+\infty
	\end{equation*}
and
	\begin{equation*}
	\int_{\Omega^\sigma(\mu)}|\nabla w_\mu|^2\dx \leq\liminf_{k\to+\infty}\int_{\Omega^\sigma(\mu)}|\nabla w_k|^2\dx=\frac{1}{4}\int_{\Omega^\sigma(\mu)}|\nabla v_\mu|^2\dx<+\infty\;\text{ for every}\ \sigma>0\,.
	\end{equation*}
To conclude it thus suffices to show that $|\nu_{w_\mu}\cdot e_1|=0$ $\H^1$-a.e.\ on $S_{w_\mu}$. This can be done using a similar argument as in the proof of Theorem~\ref{thm:main}(i). Namely, for every $R>0$ we set $g_R(\nu)\defas R|\nu\cdot e_1|+|\nu\cdot e_2|$ for every $\nu\in\S^1$. Clearly, the $1$-homogeneous extension of $g_R$ to $\R^2$ is convex. Hence~\cite[Theorem 5.22]{AFP} implies that
	\begin{equation*}
	\Lambda(\mu,\Omega)=\liminf_{k\to+\infty}\H^1(S_{w_k})=\liminf_{k\to+\infty}\int_{S_{w_k}}g_R(\nu_{w_k})\dH\geq\int_{S_{w_\mu}}g_R(\nu_{w_\mu})\dH\;\text{ for every}\ R>0\,,
	\end{equation*}
which is only possible if $|\nu_{w_\mu}\cdot e_1|=0$ $\H^1$-a.e.\ on $S_{w_\mu}$.
\end{step}

\begin{step}{2}{We show that $\Lambda(\mu,\Omega)=\Lambda'(\mu,\Omega)=\LL(\mu,\Omega)$} 
We have that $\Lambda(\mu,\Omega)\geq\Lambda'(\mu,\Omega)$, since $J(v_\mu)=\pi\mu$. Moreover, for every $w\in\D_{M,\hor}^{1/2}(\Omega)$ with $J(w^2)=\pi\mu$, by Propostion~\ref{prop:countable segments} we have that $S_w$ is a competitor for the minimisation problem defining $\LL(\mu,\Omega)$, hence $\Lambda'(\mu,\Omega)\geq\LL(\mu,\Omega)$. 
% Indeed, for any $w\in\D_{M,\hor}^{1/2}(\Omega)$ with $J(w^2)=\pi\mu$, by Propostion~\ref{prop:countable segments} we write, up to a $\H^1$-negligible set,  $S_w=\bigcup_{j\in\N}(y_j^1,y_j^2)$. Note that only at most $M$ of the segments $L_j$ can connect the points $x_h\in\supp\mu$ to $\partial\Omega$ or among each other. Upon relabeling we can assume that those belong to the segments $L_1,\ldots,L_M$, so that the collection of segments $(L_j)_{j=1}^M$ is admissible for the minimisation problem defining $\LL(\mu,\Omega)$. We thus obtain $\H^1(S_w)\geq\sum_{j=1}^M|y_j^1-y_j^2|\geq\LL(\mu,\Omega)$.
% Passing to the infimum over $w\in\D_{M,\hor}^{1/2}(\Omega)$ with $J(w^2)=\pi\mu$ we deduce that $\Lambda'(\mu,\Omega)\geq\LL(\mu,\Omega)$. 

\smallskip

To conclude, it thus suffices to show that $\LL(\mu,\Omega)\geq\Lambda(\mu,\Omega)$.  Let us fix $S \in \Ten(\mu,\Omega)$ \emph{resolving dislocations tension}. By applying Proposition~\ref{prop:w given S} to $S$ and the  canonical harmonic map $v_\mu \in \D_M(\Omega)$, we find $w_\mu \in \D_{M,\hor}^{1/2}(\Omega)$ such that $S_{w_\mu} = S$ up to an $\H^1$-negligible set and $w_\mu^2 = v_\mu$. Thus, $w_\mu$ is a competitor for the minimisation problem defining $\Lambda(\mu,\Omega)$. Since $\H^1(S) = \H^1(S_w)$, we deduce that $\H^1(S)\geq\Lambda(\mu,\Omega)$. We conclude by taking the infimum in $S$. 

\end{step}
\end{proof}
\begin{remark}\label{rem:energy-w-mu}
From Lemma~\ref{lem:min-length} together with~\eqref{eq:W-can-harm} we deduce in particular that
	\begin{equation*} 
	\begin{split}
	\F^\p(\mu,\Omega) &=\frac{M}{4}\gamma+\frac{1}{4}\W(\mu,\Omega)+\alpha\LL(\mu,\Omega)=\frac{M}{4}\gamma+\frac{1}{4}\W(\mu,\Omega)+\alpha\Lambda(\mu,\Omega)\\
	&=\frac{M}{4}\gamma+\frac{1}{4}\WW(v_\mu,\Omega)+\alpha\H^1(S_{w_\mu})=F^\p(w_\mu,w_\mu,\Omega)\,.
	\end{split}
	\end{equation*}
%Moreover, we will see that the constraint $w_\mu^2=v_\mu$ is not restrictive and the value of $\Lambda(\mu,\Omega)$ does not decrease when the minimisation takes place over all admissible configurations $w$ with $J(w^2)=\pi\mu$ (see~\eqref{def:L-mu-prime} for a precise definition). 
Thus, $\F^\p(\mu,\Omega)$ determines the minimal amount of energy induced by a configuration $\mu$ of limiting singularities and this amount in turn is obtained by minimising separately the necessary core contribution, far-field contribution, and surface contribution.
Moreover
	\begin{equation}\label{eq:F-mu-w}
		F^\p(w_\mu,w_\mu,\Omega)\leq F^\p(w_{\odd},w_{\even},\Omega)
\end{equation}
for all $w_{\odd}$ and $w_\even$ satisfying $J(w_\even^2)=\pi\mu$ and $J(w_\odd^2)=\pi\mu$.
\end{remark}
The proof of Theorem~\ref{thm:main2} is now based on Theorem~\ref{thm:main} and Lemma~\ref{lem:min-length}.
\begin{proof}[Proof of Theorem~\ref{thm:main2}]
The compactness statement Theorem~\ref{thm:main2}(i) is an immediate consequence of the corresponding statement Theorem~\ref{thm:main}(i). Indeed, suppose that $(\mu_\e)\subset X(\Omega)$ is given with $\sup_{\e>0}\big(\F_\e^\p(\mu_\e,\Omega)-\frac{M\pi}{4}|\log\e|\big)<+\infty$. By definition, this implies that there exist displacements $u_\e\in\AD_\e$ with $\mu_{2u_\e}\mres\Omega=\mu_\e$ and 
	\begin{equation}\label{eq:min-sequence-mue}
	\lim_{\e\to 0}\big(F_\e^\p(u_\e,\Omega)-\F_\e^\p(\mu_\e,\Omega)\big)=0\,.
	\end{equation}
Hence $u_\e$ satisfies~\eqref{uniformbound} and we deduce from Theorem~\ref{thm:main}(i) that up to a subsequence (not relabeled) $\mu_\e=\mu_{2u_\e}\mres\Omega\fto\mu$ for some $\mu\in X(\Omega)$ satisfying the required properties.

\smallskip
To establish Theorem~\ref{thm:main2}(ii) suppose in addition that $\mu_\e=\mu_{2u_\e}\mres\Omega\fto\mu$ with $\mu\in X_M(\Omega)$. Up to passing to a further subsequence we can assume that $w_{2\e,s_0},w_{2\e,s_1}\to w_\even$ and $w_{2\e,s_2},w_{2\e,s_3}\to w_\odd$ for some $w_\even,w_\odd\in\D_{M,\hor}^{1/2}(\Omega)$ with $w_\even^2=w_\odd^2=:v$ and $J(v)=\pi\mu$. Suppose without loss of generality that $\H^1(S_{w_\even})\leq\H^1(S_{w_\odd})$; then~\eqref{liminf:partial} together with~\eqref{eq:min-sequence-mue} yields
	\begin{align*}
	\liminf_{\e\to 0}\Big(4\F_\e^\p(\mu_\e,\Omega)-M\pi|\log\e|\big)&=\liminf_{\e\to 0}\Big(4 F_\e^\p(u_\e,\Omega)-M\pi|\log\e|\big)\\
	&\geq \WW(v,\Omega)+M\gamma+\H^1(S_{w_\even})\\
	&\geq\W(\mu,\Omega)+M\gamma+\Lambda'(\mu,\Omega)\,,
	\end{align*}
where the last inequality follows from~\eqref{est:W-W-ren} and by definition of $\Lambda'(\mu,\Omega)$ in~\eqref{def:L-mu-prime}. Thus~\eqref{lb:energy-mu} follows from Lemma~\ref{lem:min-length}.

\smallskip
In order to prove Theorem~\ref{thm:main2}(iii) it suffices to recall the Inequality~\eqref{eq:F-mu-w} in Remark~\ref{rem:energy-w-mu}. Indeed, thanks to~\eqref{eq:F-mu-w} we find for any $\mu\in X_M(\Omega)$ the required sequence $(\mu_\e)\subset X$ satisfying~\eqref{ub:energy-mu} by setting $\mu_\e\defas\mu_{2u_\e}\mres\Omega$, where $(u_\e)$ is the recovery sequence provided by Theorem~\ref{thm:main}(iii) for $w_\even=w_\odd=w_\mu$. Then~\eqref{ub:energy-mu} follows by combining~\eqref{limsup:partial} and~\eqref{eq:F-mu-w}.
\end{proof}
%%%%%%%%%%%%%%%%%%%%%%%%%%%%%%%%%%%%%%%%%%%%%%%%%%%%%%%%%%%%%%%%%%%%%%%%%%%%%%%%%%%%%%%%%%%%%

\smallskip

\noindent {\bf Acknowledgments.} \\
A.\ Bach has received funding from the European Union's Horizon research and innovation programme under the Marie Sk\l odowska-Curie grant agreement No 101065771.\\
A.\ Garroni acknowledges the financial support of PRIN 2022J4FYNJ ``Variational methods for stationary and evolution problems with singularities and interfaces'', PNRR Italia Domani, funded by the European Union via the program NextGenerationEU, CUP B53D23009320006. Views and opinions expressed are however those of the authors only and do not necessarily reflect those of the European Union or The European Research Executive Agency. Neither the European Union nor the granting authority can be held responsible for them.\\
G.\ Orlando is member of Gruppo Nazionale per l'Analisi Matematica, la Probabilità e le loro Applicazioni (GNAMPA) of the Istituto Nazionale di Alta Matematica (INdAM). G.\ Orlando has been partially supported by the Research Project of National Relevance ``Evolution problems involving interacting scales'' granted by the Italian Ministry of Education, University and Research (MIUR Prin 2022, project code 2022M9BKBC, Grant No. CUP D53D23005880006). G.\ Orlando acknowledges partial financial support under the National Recovery and Resilience Plan (NRRP) funded by the European Union - NextGenerationEU - Project Title ``Mathematical Modeling of Biodiversity in the Mediterranean sea: from bacteria to predators, from meadows to currents'' - project code P202254HT8 - CUP B53D23027760001 - Grant Assignment Decree No.\ 1379 adopted on 01/09/2023 by the Italian Ministry of University and Research (MUR). G.\ Orlando was partially supported by the Italian Ministry of University and Research under the Programme ``Department of Excellence'' Legge 232/2016 (Grant No.\ CUP - D93C23000100001).

\end{document}